\documentclass[12pt]{amsart}
\usepackage{amsfonts,amssymb,color}
\usepackage[mathscr]{eucal}
\usepackage{amsmath, amsthm}
\usepackage{mathrsfs}
\usepackage{amsbsy}
\usepackage{dsfont}
\usepackage{bbm}
\usepackage{wasysym}
\usepackage[utf8]{inputenc}
\usepackage[russian,english]{babel}
\usepackage{hyperref}
\usepackage[active]{srcltx} 
\input xypic
\xyoption{all}

\newcommand{\incl}{\hookrightarrow}

\newcommand{\into}{\rightarrow}
\newcommand{\isom}{\simeq}

\newcommand{\tensor}{\otimes}
\newcommand{\llb}{\left[\kern-0.15em\left[}
\newcommand{\rrb}{\right]\kern-0.15em\right]}

\newcommand{\Bei}{\text{\foreignlanguage{russian}{Б}}}

\newcommand \ZZ {{\mathbb Z}} 
\newcommand \NN {{\mathbb N}} 
\newcommand \QQ {{\mathbb Q}} 
\newcommand \PR {{\mathbb P}} 
\newcommand \AF {{\mathbb A}} 
\newcommand \GG {{\mathbb G}} 

\newcommand \bcA {{\mathscr A}}
\newcommand \bcB {{\mathscr B}}
\newcommand \bcC {{\mathscr C}}
\newcommand \bcD {{\mathscr D}}
\newcommand \bcE {{\mathscr E}}

\newcommand \bcH {{\mathscr H}}

\newcommand \bcL {{\mathscr L}}

\newcommand \bcS {{\mathscr S}}
\newcommand \bcT {{\mathscr T}}

\newcommand \bcX {{\mathscr X}}
\newcommand \bcY {{\mathscr Y}}
\newcommand \bcZ {{\mathscr Z}}

\newcommand \mathA {\mathcal{A}}

\newcommand \mathG {\mathcal{G}}

\newcommand \mathJ {\mathcal{J}}

\newcommand \mathP {\mathcal{P}}
\newcommand \mathQ {\mathcal{Q}}

\newcommand \bfW {\mathbf{W}}

\newcommand \Spec {{\rm {Spec}}}

\newcommand \Sym {{\rm {Sym}}}

\newcommand \cone {{\rm {cone}}}

\newcommand \tr {{\rm {tr}}}

\newcommand \im {{\rm im}\,} 

\newcommand \Hom {{\rm Hom}}
\newcommand \End {{\rm {End}}}
\newcommand \id {{\rm {id}}}

\newcommand \Aut {{\rm Aut}}

\newcommand \cor {{\rm {cor}}}
\newcommand \res {{\rm {res}}}
\newcommand \Gal {{\rm {Gal}}}

\newcommand \dom {{\rm {dom}}\,}
\newcommand \codom {{\rm {codom}}\,}
\newcommand \Map {{\rm {Map}}\,}

\newcommand \cell {{\rm {cell}}}
\newcommand \inj {{\rm {inj}}}
\newcommand \proj {{\rm {proj}}}

\newcommand \ob {\rm {ob}\,}
\newcommand \one {\mathds{1}}

\newcommand \shf {{\rm {a}}}

\newcommand \hocolim {{\rm hocolim}}

\newcommand \Nis {{\rm Nis}}

\newcommand \Hu {{\it {Hu}}}

\newcommand \Ab {{\mathcal{A}b}}
\newcommand \coeq {{\rm {coeq}}}

\newcommand \Const {{\rm {Const}}}

\newcommand \Lan {{\rm {Lan}}}
\newcommand \qfh {{\rm {qfh}}}
\newcommand \stab {{\rm {stab}}}

\newcommand \Nm {{\rm {Nm}}}

\newcommand \cl {{\rm {cl}}}

\newcommand{\bigslant}[2]{{\raisebox{.2em}{$#1$}\left/\raisebox{-.2em}{$#2$}\right.}}

\newcommand \red {{\rm {red}}}
\newcommand \diag {{\rm {diag}}}
\newcommand \tBox {\tilde{\Box}}

\newcommand \DM {{\sf DM}}

\newcommand \op {{\rm{op}}}
\newcommand \pr {{\rm{pr}}}

\newcommand \Set {{\mathscr S\! et}}
\newcommand \SSet {\Delta^{\op}\Set}
\newcommand \Shv {{Shv}}
\newcommand \Deltaop {\Delta^{\op}}

\newcommand \Sch {{\mathscr S\! ch}}
\newcommand \Sm {{\mathscr S\! m}}
\newcommand \Pre {{\it Pre}}
\newcommand \Ho {{\rm Ho}\,}

\newcommand \Sp {{\rm Spt}}
\newcommand \Ev {{\rm {Ev}}}
\newcommand \SH {{\mathcal S\! \mathcal H}}
\newcommand \tor {{\rm {tor}}}
\newcommand \sym {{\rm {sym}}}

 \DeclareMathOperator*{\colim}{colim}
 
\begin{document}
\baselineskip = 5mm

\newtheorem{theorem}{Theorem}
\newtheorem{lemma}[theorem]{Lemma}
\newtheorem{corollary}[theorem]{Corollary}
\newtheorem{proposition}[theorem]{Proposition}
\newtheorem{remark}[theorem]{Remark}
\newtheorem{definition}[theorem]{Definition}
\newtheorem{conjecture}[theorem]{Conjecture}
\newtheorem{example}[theorem]{Example}
\newtheorem{question}[theorem]{Question}
\newtheorem{warning}[theorem]{Warning}
\newtheorem{assumption}[theorem]{Assumption}
\newtheorem{fact}[theorem]{Fact}
\newtheorem{crucialquestion}[theorem]{Crucial Question}
\newcommand \lra {\longrightarrow}
\newcommand \hra {\hookrightarrow}
\def\blue {\color{blue}}
\def\green {\color{green}}
\newenvironment{pf}{\par\noindent{\em Proof}.}{\hfill\framebox(6,6)
\par\medskip}
\title[Geometric symmetric powers]
{\bf Geometric symmetric powers in the motivic homotopy category}
\author{Joe Palacios Baldeon}

\date{14 April 2015}   



\begin{abstract}
Symmetric powers of quasi-projective schemes can be extended, in terms of left Kan extensions, to geometric symmetric powers of motivic spaces. In this paper, we study geometric symmetric powers and compare with various symmetric powers in the unstable and stable $\AF^1$-homotopy category of schemes over a field. 
\end{abstract}






\maketitle

\tableofcontents

\section{Introduction}
\label{s-intro}

In the last two decades the development of the $\AF^1$-homotopy theory of schemes has had a noticeable impact in algebraic geometry, particularly in the successful resolution of the Bloch-Kato conjecture \cite{SusVoe99}.  In topology, the Dold-Thom theorem says that the group completion of the infinite symmetric power of a pointed connected $CW$-complex is weak equivalent to a product of the Eilenberg-MacLane spaces associated to its homology groups. Symmetric powers has been used to encode (co)-homological information of motivic spaces, as the Dold-Thom's theorem predicts.


In general terms, motivic spaces depends of two coordinates: one simplicial coordinate and one geometric coordinate, i.e. the category of schemes.  This suggests the possibility of defining symmetric powers of motivic spaces in a different approach than the categoric ones. In\cite{Voe10}, Voevodsky proved a motivic version of the Dold-Thom's theorem. The symmetric powers considered in his work are what we call geometric symmetric powers, as they are induced from the geometric coordinate.  

 An admissible category \footnote{ $f$-admissible in \cite{Voe10}.} is a subcategory of schemes over a base field that is closed under taking quotients of schemes by finite groups and contains the affine line as an interval. Geometric symmetric powers are Kan extensions of the symmetric powers of schemes considered in an admissible category \cite{Voe10}.  Categoric symmetric powers are the quotients of Cartesian powers of motivic spaces by the action of symmetric groups. A $\lambda$-structure on a model category, or in its homotopy category, is a categoric version of a $\lambda$-structure on commutative rings. As functors, categoric symmetric powers preserve $\AF ^1$-weak equivalences, and their left derived functors provide a $\lambda$-structure in cofibre sequences on the pointed and unpointed (unstable) motivic homotopy categories of an admissible category, \cite{GoGu}. The aim of the present work is to develop a systematic study of symmetric powers in the unstable and stable homotopy category of an admissible category over a field $k$.

Our first goal is to prove that geometric symmetric powers provide a $\lambda$-structure on the pointed unstable motivic homotopy category of an admissible category. For this purpose we first consider the projective cofibrant resolution on the category of simplicial Nisnevich sheaves on an admissible category, deduced from the small object argument applied to the class of morphisms resulting by multiplying representable sheaves with the generating cofibrations of the category of simplicial sets. This allows us to deduce that every motivic space is $\AF ^1$-weak equivalent to a simplicial sheaf, given termwise by coproducts of representable sheaves, as it was shown by Voevodsky in the context of radditive functors, see \cite{Voe10,Voe10-5}. The key point consists in fact that geometric symmetric powers of morphisms of ind-representable simplicial sheaves have canonical filtrations, called K\"unneth towers, and they provide a $\lambda$-structure on the motivic homotopy category. This gives the following result (Theorem \ref{radmainth1} in the text):

\medskip

\begin{itemize}

\item[]{}
{\it The left derived geometric symmetric powers provide a $\lambda $-structure on the pointed unstable motivic homotopy category of an admissible category of quasi-projective schemes over a field.
}
\end{itemize}

   \medskip

 On the other hand, in both unstable and stable case, there is a natural transformation from the categoric symmetric power $\Sym^n$ to the geometrical symmetric power $\Sym^n_g$. Let $E$ be a functor from an admissible category to the unstable (or stable) $\AF^1$-homotopy category on an admissible category. An interesting problem is to investigate whether the canonical morphisms $\vartheta^m_X:\Sym^n E(X)\into \Sym^n_gE(X)$ are isomorphisms for schemes $X$ in an admissible category. It turns out that, in the unstable case, $\vartheta^n_X$ is not always an isomorphism, for example when $X$ is the $2$-dimensional affine space $\AF^2$ and $n=2$, cf. Proposition \ref{papxz301}. Our second goal is to show that these canonical morphisms become isomorphisms in the rational stable $\AF^1$-homotopy category of schemes. However, the same result is not true on the stable $\AF^1$-homotopy category of schemes with integral coefficients (see Remark \ref{Remm74}). 

Let us explain in few words our approach for the second goal. The rationalization of a stable homotopy category causes the loss of information of the torsion objects. However, it allows us to think a rational stable homotopy category as a derived category of chain complexes, and the latter is, philosophically, more accessible to understand. Morel showed that these phenomena happen in the motivic set-up \cite{Mor04}. More precisely, the stable $\AF^1$-homotopy category of schemes is equivalent to the triangulated category of unbounded motives, cf. {\em loc.cit}.

An important ingredient to be used in this text is the notion of transfer of a morphism.  This notion appears naturally in algebraic topology. For instance, let us consider a positive integer $n$ and a $n$-sheeted covering $\pi:\tilde{X}\into X$. Let $\pi^*:H^r(X;\ZZ)\into H^r(\tilde{X};\ZZ)$ be the induced homomorphism of cohomology groups, for some $r\in\NN$. A transfer for $\pi^{\ast}$ is a homomorphism $\tr:H^r(\tilde{X};\ZZ)\into H^r(X;\ZZ)$ such that the composite $\tr\circ \pi^*$ is the multiplication by $n$.  Voevodosky proved the existence of transfers morphisms of $\qfh$-sheaves induced by finite surjective morphisms of normal connected schemes. This implies the existence transfers of morphisms of $\qfh$-motives induced by such finite morphisms of schemes, see \cite{Voe96}.  We use this notion in order to get transfers, for morphisms induced the canonical morphism $X^n\into X^n/\Sigma_n$ for a quasi-projective scheme $X$, in the rational stable $\AF^1$-homotopy category.
 Let $T$ be the projective line $\PR^1$ pointed at $\infty$, and let $E_{\QQ}$ be the canonical functor from the category of quasi-projective schemes over a field $k$ to the rational stable  $\AF^1$-homotopy category of $T$-spectra. We denote by $\Sym^n_{T}$ the $n$th fold categoric symmetric power on the category of symmetric $T$-spectra.   
Since the rational stable homotopy category of schemes is pseudo-abelian, one can use projectors in order to define projector symmetric powers, denoted by $\Sym^n_{\pr}$. 
As a result, we obtain that, if $-1$ is a sum of squares, then the categoric, geometric and projector symmetric powers of a quasi-projective scheme are isomorphic in rational stable $\AF^1$-homotopy category. More precisely, our result is the following 
(Theorem \ref{finth1} in the text):
\medskip
\begin{itemize}
\item[]
{\it 
Let $k$ be a field such that $-1$ is a sum of squares in it. Then, for any quasi-projective $k$-scheme $X$, we have the following isomorphisms
$$L\Sym^n_TE_{\QQ}(X)\isom  E_{\QQ}(\Sym^n X)\isom \Sym^n_{\pr}E_{\QQ}(X)\,.$$ 
}
\end{itemize}
\medskip          
The fourth type of symmetric power, that is defined as a homotopy quotient under the action of a symmetric group, is called homotopy symmetric power, and is denoted by $\Sym^n_{h,T}$. It turns out that the natural transformation from the $n$th fold homotopy symmetric power to the $n$th categoric symmetric power is an isomorphism on the stable $\AF^1$-homotopy category (see \cite{GoGu11}). Consequently, for a quasi-projective $k$-scheme $X$, the $n$th fold homotopy symmetric power $\Sym^n_{h,T}E_{\QQ}(X)$ is isomorphic to $L\Sym^n_TE_{\QQ}(X)$. Thus, we get a comparison of four types of  symmetric powers in the rational stable $\AF^1$-homotopy category.

Although in this paper we are limited to work only over a base field, our constructions might be generalized to a broader class of nice base schemes. It would be interesting to investigate how to construct categoric (resp. geometric) symmetric powers, in a more general framework, namely on premotivic categories (resp. premotivic categories with geometric sections) defined in \cite{C-D13}; but this question is beyond the author's knowledge.  
   \medskip

The paper is organized as follows. In Section \ref{Prelim}, we give a survey of admissible categories and basic properties of simplicial Nisnevich sheaves. Section \ref{geomsymu} is devoted to the study of geometric symmetric powers and K\"unneth towers associated to cofibre sequences. Section \ref{weakequiv} concerns the $\AF ^1$-localization of geometric symmetric powers. In section  \ref{lamstr} we recall the notion of a $\lambda $-structure and prove one of the main result (Theorem \ref{radmainth1}). In section \ref{MorphismL}, we construct a canonical morphism of $\lambda$-structures from categoric symmetric powers to geometric symmetric powers. In Section \ref{Prelims}, we give an outline on the category of spectra and the notion of rational stable $\AF^1$-homotopy category. Section \ref{geomsym} presents the construction of geometric symmetric powers in the category of motivic symmetric spectra. Section
\ref{lammdctt} concerns the $\lambda$-structure of geometric symmetric powers in the stable set-up. 
Section \ref{Transfers} is devoted to the formalism of transfers. In section \ref{Mainsec}, we prove our main result in the stable set-up, see Theorem \ref{finth1}. Finally, in \ref{Appen}, we show that the categoric symmetric powers and geometric symmetric powers does not coincide on the unstable $\AF^1$-homotopy category of schemes. However, on sheaves represented by Galois field extensions, the above morphism of $\lambda$-structures gives isomorphisms on sections of the form $\Spec(L)$, where $L$ is a finite Galois extension over the base field, see Proposition \ref{abs.prgx}. 

\medskip

{\sc Acknowledgements.} I am grateful to Vladimir Guletski\u \i \,  for his inspiring explanations and helpful suggestions, and to Anwar Alameddin for several discussions. The paper is written in the framework of the EPSRC grant EP/I034017/1 ``Lambda-structures in stable categories". I am grateful to EPSRC for funding my PhD project.

%

\section{Admissible categories and Nisnevich sheaves}
\label{Prelim}

 Throughout the paper $k$ will denote a field of arbitrary characteristic and we write $\AF^1$ for the affine line over $\Spec(k)$. Let $\Sch/k$ be the category of schemes over $k$.  For two $k$-schemes $X$ and $Y$, we write $X\times Y$ to mean the Cartesian product $X\times_{\Spec(k)} Y$.   We also denote by $X\amalg Y$ the disjoint union of $X$ and $Y$, as schemes. The point $\Spec(k)$ and the empty scheme $\emptyset$, as objects of $\Sch/k$, are the terminal and initial object respectively.   

 We say that a full subcategory $\bcC$ of $\Sch/k$ is {\em admissible}, if it satisfies the following five axioms: ({\it i}) the point $\Spec(k)$ and the affine line $\AF^1$ are objects of $\bcC$, ({\it ii}) for any two objects $X$ and $Y$ of $\bcC$, the product $X\times Y$ is in $\bcC$, ({\it iii}) for any two objects $X$ and $Y$ of $\bcC$, the disjoint union $X\amalg Y$ is in $\bcC$, ({\it iv}) if $U\into X$ is an \'etale morphism of  $k$-schemes such that $X$ is in $\bcC$, then $U$ is in $\bcC$, ({\it v}) If $G$ is a finite group acting on an object $X$ of $\bcC$, then the quotient $X/G$ is in $\bcC$ whenever it exists as $k$-scheme.  In \cite{Voe10},  this definition appears as $f$-admissible category.    
 
A typical example of an admissible category is the category of quasi-projective schemes over $k$. 

\begin{remark}{\em
 By definition every admissible category of schemes over a field contains the affine line $\AF^1$, but it is not true that all admissible categories contain the projective line $\PR^1$ over a field. For example, the subcategory of normal quasi-affine schemes over a perfect field is admissible, but the projective line $\PR^1$ is not quasi-affine. }
\end{remark}
%
Unless indicated otherwise, $\bcC$ will be a small, admissible category contained in the category of quasi-projective schemes over $k$. These conditions will allow us to use Voevodsky's results in \cite {Del09,Voe10-214}. The smallness condition on $\bcC$ permits to use Kan extensions, expressed in terms of small colimits, when we define geometric symmetric powers of simplicial sheaves. In the sequel, all colimits will mean small colimits.  
\bigskip 

 An {\em elementary distinguished square} in $\bcC$ is a Cartesian square of the form
\begin{displaymath}\label{papeq10}
\mathQ:\qquad\vcenter{\xymatrix@C=10ex@R=10ex{Y\ar[d]\ar[r]&V\ar[d]^p\\
U\ar[r]_j&X}
}
\end{displaymath}
where $j$ is an open embedding and $p$ is an \'etale morphism such that the induced morphism $p^{-1}(X-U)_{\red}\into (X-U)_{\red}$ of reduced schemes is an isomorphism. We recall that a family of \'etale morphisms $\{f_i:U_i\into X\}_{i\in I}$ of $\bcC$ is a {\em Nisnevich covering} if for every point $x\in X$ there exists an index $i\in I$ and a point $y\in U_i$ such that $f_i(y)=x$ and the corresponding morphism of residual fields $k(x)\into k(y)$ is an isomorphism. 
The Nisnevich topology on $\bcC$ can be described as the smallest Grothendieck topology generated by families of the form $\{j:U\into X,\,p:V\into X\}$ associated to elementary distinguished squares of the form \eqref{papeq10}, see \cite[page~1400]{Voe10-214}.  We denote by $\bcC_{\Nis}$ the site consisting of $\bcC$ and the Nisnevich topology on it.  

The category of presheaves $\Pre(\bcC)$ on $\bcC$ is the category of functors from the opposite category $\bcC^{\op}$ to the category of sets. A presheaf $F$ on $\bcC$ is a Nisnevich sheaf if and only if for each elementary distinguished square \eqref{papeq10}, the square of sets 
$$
F(\mathQ):\qquad\vcenter{\xymatrix@C=10ex@R=10ex{F(X)\ar[d]_{F(j)}\ar[r]^{F(p)}&F(V)\ar[d]\\
F(U)\ar[r]& F(Y)}
}
$$
 is Cartesian, see \cite{MoVoe99}. In the sequel, $\bcS$ will denote the 
category of sheaves on $\bcC_{\Nis}$. 

We denote by $h$ the Yoneda embedding of $\bcC$ into the category of presheaves $\Pre(\bcC)$.
 Since representable sheaves are Nisnevich sheaves, sometimes we use the same letter $h$ to denote the full embedding of $\bcC$ into $\bcS$.  We recall that the forgetful functor from $\bcS$ to $\Pre(\bcC)$ has a left adjoint which we denote by $\shf_{\Nis}$. The category $\bcS$ is complete and cocomplete, its terminal object is $h_{\Spec(k)}$, and filtered colimits of Nisnevich sheaves in the category of presheaves are Nisnevich sheaves. Let $\{F_i\}_{i\in I}$ be a family of objects in $\bcS$.  The coproduct of this family in $\bcS$ is the sheafification $\shf_{\Nis} \left(\coprod_{i\in I} F_i\right)$
of the coproduct  $\coprod_{i\in I} F_i$ in $\Pre(\bcC)$.  We abusively denote it by $\coprod_{i\in I}F_i$, if no confusion arises.

In this paper, we shall consider the injective model structure on the category of simplicial sheaves $\Delta^{\op}\bcS$, where the class of cofibrations is the class of monomorphisms, a weak equivalence is a stalk-wise weak equivalence and fibrations are morphisms having the right lifting property with respect to trivial cofibrations. 

The category $\Delta^{\op}\bcS$ is a simplicial category.  For a simplicial sheaf $\bcX$ and a simplicial set $K$, we define the product $\bcX\times K$ to be the simplicial sheaf, such that  for every  $n\in\NN$, its term $(\bcX\times K)_n$ is defined to be the coproduct $\coprod_{K_n}\bcX_{n}$ in $\bcS$.   For a couple of sheaves $(\bcX,\bcY)$, the function complex $\Map(\bcX,\bcY)$ is defined to be the simplicial set which assigns an object $[n]$ of $\Delta$ to the set 
$\Hom_{\Delta^{\op}\bcS}(\bcX\times\Delta[n],\bcY)\,.$
Then, for every pair of simplicial sheaves $(\bcX,\bcY)$ and every simplicial set $K$, one has a natural bijection, 
\begin{equation}\label{papex15}
\Hom_{\Delta^{\op}\bcS}(\bcX\times K,\bcY)\isom\Hom_{\SSet}\left(K, \Map(\bcX,\bcY)\right)\,,
\end{equation}
which is functorial in $\bcX$, $\bcY$ and $K$. 

For each object $U$ of $\bcC$, we denote by $\Delta_U[0]$ the constant functor from $\Delta^{\op}$ to $\bcS$ with value $h_U$. Sometimes, we shall simply write $h_U$ instead of $\Delta_U[0]$, if no confusion arises.
For each $n\in\NN$ and each object $U$ of $\bcC$, we denote by $\Delta_U[n]$ the simplicial sheaf $\Delta_U[0]\times\Delta[n]$. Similarly, we denote by $\partial\Delta_U[n]$ the simplicial sheaf $\Delta_U[0]\times\partial\Delta[n]$.

Notice that Yoneda lemma provides an isomorphism $\Map(\Delta_U[0],\bcY)\isom \bcY(U)$  for every object $U$ of $\bcC$ and every simplicial sheaf $\bcY$. Hence, replacing $\bcX$ by $\Delta_U[0]$ in \eqref{papex15}, we obtain an isomorphism 
\begin{equation}\label{papex16}
\Hom_{\Delta^{\op}\bcS}(\Delta_U[0]\times K,\bcY)\isom\Hom_{\SSet}\left(K, \bcY(U)\right)\,.
\end{equation}

We write $\bcC_+$ to denote the full subcategory of the pointed category $\bcC_{\ast}$ generated by objects of the form $X_+:=X\amalg \Spec(k)$. We denote by $\bcS_{\ast}$ the pointed category of $\bcS$. The symbols $\vee$ and $\wedge$ denote, respectively, the coproduct and the smash product in $\bcS_{\ast}$. An elementary square in $\bcC_+$ is a square of the form $\mathQ_+$, where $\mathQ$ is an elementary distinguished square in $\bcC$. A family $\{(f_i)_+\}_{i\in I}$ of morphisms in $\bcC_+$ is a {\em Nisnevich covering} if the family $\{f_i\}_{i\in I}$ is a Nisnevich covering in $\bcC$. Let $\bcC_{\Nis,+}$ be the site consisting of $\bcC_+$ and the Nisnevich topology on it. The category $\Shv(\bcC_{\Nis,+})$ is equivalent to the pointed category $\Shv(\bcC_{\Nis})_{\ast}$.

We denote by $\bcH(\bcC_{\Nis})$ (resp. by $\bcH(\bcC_{\Nis},\AF^1)$) the homotopy category of $\Deltaop\bcS$ localized with respect to weak equivalences (resp. $\AF^1$-weak equivalences). 
We write $\bcH_{\ast}(\bcC_{\Nis})$ (resp. by $\bcH_{\ast}(\bcC_{\Nis},\AF^1)$) for the homotopy category of $\Deltaop\bcS_{\ast}$ localized with respect to weak equivalences (resp. $\AF^1$-weak equivalences) in the pointed model structure. 

\begin{remark}\label{remffcx2}
\em The category $\bcH(\bcC_{\Nis,+})$ is equivalent to the pointed homotopy category $\bcH_{\ast}(\bcC_{\Nis})$. Similarly, the category $\bcH(\bcC_{\Nis,+},\AF^1)$ is equivalent to the pointed homotopy category $\bcH_{\ast}(\bcC_{\Nis},\AF^1)$. 
\end{remark}

\medskip

 In the next paragraph, we recall the notion of $\bar{\Delta}$-closed classes introduced by Voevodsky, see \cite{Voe10-5}. They allow us to express the class of $\AF^1$-weak equivalences as a certain class generated by local weak equivalences and by $\AF^1$-homotopies given in terms of  projections (see Example \ref{papex17}). 
 
 Let $\bcD$ be a category with  finite coproducts. A class of morphisms $E$ of $\Delta^{\op}\bcD$ is called {\em $\bar{\Delta}$-closed}, if it satisfies the following four axioms: ({\it i}) $E$ contains all $\Delta[1]$-homotopy equivalences in $\Delta^{\op}\bcD$,  ({\it ii})  $E$ has the $2$-out-of-$3$ property,  ({\it iii})  if $f$ is a morphism of bi-simplicial objects on $\bcD$ such that for every $n\in\NN$, either $f([n], -)$ or $f(-,[n])$ belongs to $E$, then the diagonal morphism of $f$ belongs to $E$, and  ({\it iv}) $E$ is closed under filtered colimits.  For any class of morphisms $E$ in $\Delta^{\op}\bcD$, we denote by $cl_{\bar{\Delta}}(E)$ the smallest $\bar{\Delta}$-closed class containing the class $E$.

\begin{example}\label{papex17}
{\em The class of $\AF^1$-weak equivalences in $\Delta^{\op}\bcS$ coincides with the $\bar{\Delta}$-class $cl^{}_{\bar{\Delta}}(\bfW_{\Nis}\cup \mathP_{\AF^1})$, where $\bfW_{\Nis}$ is the class of local equivalences with respect to the Nisnevich topology and  $\mathP_{\AF^1}$ is the class of projections from $\Delta_X[0]\times \Delta_{\AF^1}[0]$ to $\Delta_X[0]$, for $X\in\bcC$ (see \cite[Th.~4, page 378]{Del09}). 
Similarly, the class of $\AF^1$-weak equivalences in $\Delta^{\op}\bcS_{\ast}$ coincides with the class $cl^{}_{\bar{\Delta}}(\bfW_{\Nis,+}\cup \mathP_{\AF^1,+})$, where $\bfW_{\Nis,+}$ is the image of $\bfW_{\Nis}$ through the functor which sends a simplicial sheaf $\bcX$ to the pointed simplicial sheaf $\bcX_+$ and $\mathP_{\AF^1,+}$ is the image of $\mathP_{\AF^1}$ through the same functor. 
}
\end{example}

A morphism from $A$ to $X$ in $\bcD$ is called {\em coprojection}, if there exists an object $Y$ of $\bcD$ such that this morphism is isomorphic to the canonical morphism from $A$ to $A\amalg Y$. A morphism $f$ in $\Deltaop\bcD$ is called a {\em termwise coprojection}, if for each natural $n$, its term  $f_n$ is a coprojection. 

\begin{example}\label{radxcoprjcxx}
{\em Let $\bcX$ be a simplicial sheaf on $\bcC_{\Nis}$. If $K\subset L$ is an inclusion of simplicial sets, then the induced morphism from  
$\bcX\times K$ to $\bcX\times L$ is a termwise coprojection. Indeed, for each natural $n$, the $n$-simplex $(\bcX\times K)_n$ is equal to the coproduct of sheaves $\coprod_{K_n}\bcX_n$, similarly,  $(\bcX\times L)_n$ is equal to $\coprod_{L_n}\bcX_n$. In view of the inclusion $K_n\subset L_n$, we have a canonical isomorphism 
$$\coprod_{L_n}\bcX_n\isom \left(\coprod_{K_n}\bcX_n\right)\amalg\left(\coprod_{L_n\setminus K_n}\bcX_n\right)\,,$$ 
which allow us to deduce that $(\bcX\times K)_n\into (\bcX\times L)_n$ is a coprojection for all $n\in\NN$. 
}
\end{example}

We write $\omega$ for the countable ordinal. This notation will be used in Lemma \ref{radtermwise}, Corollary \ref{radtermwise1} and Corollary \ref{radcorsmll}.  
\begin{lemma}\label{radtermwise}
 Let $\bcD$ be a cocomplete category. Then, the class of termwise coprojections in $\Deltaop\bcD$ is stable under pushouts, small coproducts and countable transfinite compositions.
\end{lemma}
\begin{proof}
Since colimits in $\Deltaop\bcD$ are termwise, it is enough to prove that the class of coprojections in $\bcD$ is closed under ({\it  a}) pushouts, ({\it  b}) arbitrary coproducts and ({\it  c}) countable transfinite compositions. Indeed, ({\it  a}) follows from the fact that the pushout square of a coprojection $i_A:A\into A\amalg Y$ and a morphism $f:A\into B$ is a cocartesian square of the form
 $$\xymatrix@C=10ex@R=10ex{A\ar[d]_f\ar[rr]^-{i_{A}}&&A\amalg Y\ar[d]^{f\amalg \id_{Y}}\\
 B\ar[rr]_-{i_{B}}&&B\amalg Y}$$
where the bottom horizontal and the right vertical arrows are the canonical morphisms.  To prove ({\it  b}), we give a family of canonical coprojections $\{A_i\into A_i\amalg X_i\}_{i\in I}$, where $I$ is a set of indices. We can assume that $I$ is an ordered set. Then,  the coproduct $\coprod_{i\in I} (A_i\amalg X_i)$ is isomorphic to the coproduct $\left(\coprod_{i\in I} A_i\right)\amalg\left( \coprod_{i\in I}X_i\right)$. This isomorphism allow us to deduce ({\it  b}). 
Finally, the class of coprojections is closed under countable coprojections because one can deduce that any countable sequence $X_0\into X_1\into X_2\into\cdots$ ($n<\omega$) has terms $X_n$ of the form $\coprod_{i=0}^n{X'}_i$ for all $n<\omega$, where ${X'}_0=X_0$. Therefore, its tranfinite composition coincides with the canonical morphism from $X'_0$ to $\coprod_{i<\omega }X'_i$. This shows ({\it  c}) . 
\end{proof}

 We denote by $\bar{\bcC}$ the full subcategory of $\bcS$ generated by objects which are isomorphic to coproducts of representable functors in $\bcS$.
The full embedding of $\bar{\bcC}$ into $\bcS$ induces a full embedding of $\Delta^{\op}\bar{\bcC}$ into $ \Delta^{\op}\bcS$. 
\bigskip

\begin{corollary}\label{radtermwise1}
Let $I$ be a set of morphisms in $\Delta^{\op}\bar{\bcC}$ consisting of termwise coprojections. Then, any countable transfinite composition of pushouts of coproducts of elements of $I$ is a termwise coprojection with terms of the form $\bcY_n\into \bcY_n\amalg \bcZ_n$, where $\bcZ_n$ is in $\bar{\bcC}$, for $n\in\NN$.  
\end{corollary}
\begin{proof}
Let $\bcX:\omega\into\Delta^{\op}\bcS$
be a $\omega$-sequence such that for each $n<\omega$ the morphism $\bcX_n\into \bcX_{n+1}$ is a pushout of coproducts of elements of $I$. Since $I$ consists of termwise coprojections in $\Delta^{\op}\bar{\bcC}$, by Lemma \ref{radtermwise}, the coproduct of elements of $I$ is a termwise coprojection  in $\Delta^{\op}\bar{\bcC}$. Since termwise coprojection are closed under pushouts, we deduce that the terms of each morphism  $\bcX_n\into \bcX_{n+1}$ are canonical morphisms of the form $(\bcX_n)_i\into (\bcX_{n})_i\amalg \bcY_{n,i}$, where $\bcY_{n,i}$ is an object in $\bar{\bcC}$, for $i\in\NN$. Finally, by the same Lemma, we conclude the each term of the transfinite composition of $\bcX$ is a canonical morphism of the form $(\bcX_0)_i\into (\bcX_{0})_i\amalg \bcY_i$, where $\bcY_i$ is an object of $\bar{\bcC}$.   
\end{proof}

\bigskip

We define the following sets of morphisms of simplicial sheaves 
\begin{align}\label{lagssqq}
&I_{\proj}:=\left \lbrace \partial\Delta_U[n]\into \Delta_U[n]\,|\, U\in\bcC, n\in\NN\right\rbrace
\end{align}
\medskip

Notice that, by Example \ref{radxcoprjcxx}, the morphisms $\partial\Delta_U[n]\into \Delta_U[n]$ are termwise coprojections in $\Delta^{\op}\bar{\bcC}$ for all $U\in\bcC$ and  $n\in\NN$. 
 
\begin{lemma}\label{radfiniteob}
For any object $U\in \bcC$ and every finite simplicial set $K$, the object $\Delta_U[0]\times K$ is finite relative to $\Delta^{\op}\bcS$, in the sense of Definition~2.1.4 of \cite{Hovey0}.   
\end{lemma}
\begin{proof}
 Let us fix an object  $U\in \bcC$ and a  finite simplicial set $K$.   Since $K$ is finite, there is a finite cardinal $\kappa$ such that $K$ is $\kappa$-small relative to all morphisms of $\SSet$.   We claim that $\Delta_U[0]\times K$ is $\kappa$-small relative to all morphisms in $\Delta^{\op}\bcS$. Indeed,  let $\lambda$ be a $\kappa$ -filtered ordinal and let 
\[\bcX_0\into \bcX_1\into \cdots \into \bcX_{e}\into \cdots (\beta< \lambda)\]
be a $\lambda$-sequence of simplicial sheaves on $\bcC_{\Nis}$. Since filtered colimits  of Nisnevich sheaves (computed in the category of presheaves) are sheaves, we obtain a $\lambda$-sequence of simplicial sets,
\[\bcX_0(U)\into \bcX_1(U)\into \cdots \into \bcX_{e}(U)\into \cdots (\beta< \lambda)\]
Then, we have a commutative diagram 
\[\xymatrix@C=7ex@R=7ex{\colim_{d\in D}\Hom_{\Delta^{\op}\bcS}(\Delta_U[0]\times K, \bcX_{e})\ar[d]\ar[r]&\Hom_{\Delta^{\op}\bcS}\left(\Delta_U[0]\times K, \colim_{d\in D}\bcX_{e}\right)\ar[d]\\
\colim_{d\in D}\Hom_{\SSet}(K,\bcX_{e}(U))\ar[r]&\Hom_{\SSet}\left(K,\colim_{d\in D}(\bcX_{e}(U))\right)
}\]
where the vertical arrows are bijections.
Since  $K$ is $\kappa$-small relative to all morphisms of $\SSet$, the below arrow of the preceding diagram is bijective, hence the top arrow is so. This completes the proof.  
\end{proof}

\begin{lemma}\label{paplem74}
Every morphism in $(I_{\proj})$-$\inj$ is a sectionwise trivial fibration. 
\end{lemma}
\begin{proof}
Let $f:\bcX\into \bcY$ be a morphism in  $(I_{\proj})$-$\inj$ and let us fix an object $U$ of $\bcC$. By the naturality of the isomorphism \eqref{papex16}, a commutative diagram 
\begin{equation}\label{papex22}
\xymatrix@C=10ex@R=10ex{\partial\Delta[n]\ar[d]\ar[r]&\bcX(U)\ar[d]\\
\Delta[n]\ar[r]&\bcY(U)}
\end{equation}
in $\SSet$, corresponds biunivocally to a diagram 
  \[\xymatrix@C=10ex@R=10ex{\partial\Delta_U[n]\ar[d]\ar[r]&\bcX\ar[d]\\
\Delta_U[n]\ar[r]&\bcY}\]
in $\Delta^{\op}\bcS$. As the left vertical arrow is an element of $I_{\proj}$, the above diagram has a lifting. Therefore, the bijection \eqref{papex16} induces a lifting of \eqref{papex22}. 
\end{proof}

The following corollary is a consequence of the small object argument. It will be useful to show that the cofibrant resolution takes its values in the category $\Delta^{\op}\bar{\bcC}$. 
 
\begin{corollary}\label{radcorsmll}
There exist a functorial factorization $(\alpha,\beta)$ on $\Delta^{\op}\bcS$ such that for every morphism $f$ is factored as $f=\beta(f)\circ\alpha(f)$, where $\beta(f)$ is sectionwise trivial fibration and  $\alpha(f)$ is a termwise coprojection with terms form $\bcX_n\into \bcX_n\amalg\bcY_n$, where  $\bcY_n$ is an object of $\bar{\bcC}$.
\end{corollary}
\begin{proof}
By Lemma \ref{radfiniteob}, the objects $\partial\Delta_U[n]$ and $\Delta_U[n]$ are finite relative to $\Delta^{\op}\bcS$. Since the countable ordinal $\omega$ is $\kappa$-filtered,  the small object argument provides a factorization such that $\beta(f)$ in $(I_{\proj})$-$\inj$ and $\alpha(f)$ is a countable transfinite composition of pushouts of coproducts of elements of  $I_{\proj}$. 
By Example \ref{radxcoprjcxx}, every morphism $\partial\Delta_U[n]\into\Delta_U[n]$ of $I_{\proj}$ is a termwise coprojection in $\Delta^{\op}\bar{\bcC}$. Therefore, Corollary \ref{radtermwise1}  
provides the desired factorization.   
\end{proof}


 We denote by $Q^{\proj}$ the endofunctor of $\Delta^{\op}\bcS$ which sends a simplicial sheaf $\bcX$ to the codomain of the morphism $\alpha(\emptyset\into\bcX)$, where $\emptyset$ is the initial object of $\Delta^{\op}\bcS$.  The endofunctor $Q^{\proj}$ will be called cofibrant resolution. 
In particular, for every object $\bcX$ of $\Delta^{\op}\bcS$, the canonical morphism from $Q^{\proj}(\bcX)$ to $\bcX$ is a section-wise trivial fibration.

\begin{corollary}\label{radreplac10}
The functor  $Q^{\proj}$ takes values in $\Delta^{\op}\bar{\bcC}$. 
\end{corollary}
\begin{proof}
Let $\bcX$ be a simplicial sheaf in $ \Delta^{\op}\bcS$. By Corollary \ref{radcorsmll}, the morphism of simplicial sheaves $\emptyset\into\bcX$, where $\emptyset$ is the initial object of $ \Delta^{\op}\bcS$, factors into $\emptyset\into Q^{\proj}(\bcX)\into\bcX$ such that the terms of $Q^{\proj}(\bcX)$ are in $\bar{\bcC}$, that is, $Q^{\proj}(\bcX)$ is in $\Deltaop\bar{\bcC}$. 
\end{proof}



\begin{lemma}\label{radlzz9}
The class of $\AF^1$-weak equivalences in $\Delta^{\op}\bcS_{\ast}$ is closed under finite coproducts and smash products. 
\end{lemma}
\begin{proof}
By Example \ref{papex17}, the class of $\AF^1$-weak equivalences in $\Delta^{\op}\bcS_{\ast}$ is $\bar{\Delta}$-closed. Then, it is closed under finite coproducts. Let us consider the statement for smash products. By the cube lemma (see \cite[Lemma 5.2.6]{Hovey0}),  one reduces this problem to the unpointed case, i.e. for products in $\Delta^{\op}\bcS$. Using standard simplicial methods, the problem is reduced to show that; for every $\AF^1$-weak equivalence and  every simplicial sheaf $\bcZ$ of the form $\Delta_U[0]$, where $U$ is in $\bcC$, the product $f\times\id_{\bcZ}$ is an $\AF^1$-weak equivalence. But it follows from Example \ref{papex17} and Lemma 2.20 of \cite{Voe10-5} applied to the functor $(-)\times \id_{\bcZ}$.
\end{proof}

\section{Geometric symmetric powers of simplicial sheaves}
\label{geomsymu}

In this section, we study the geometric symmetric powers on the category of simplicial Nisnevich sheaves. Here, we prove the K\"unneth rule for geometric symmetric powers (see Corollary \ref{radccttff}).  
\bigskip

Let $\bcC\subset \Sch/k$ be an admissible category. Fix an object $X$ of $\bcC$ and an integer $n\geq1$. By definition of an admissible category,  $\bcC$ is closed under finite products and quotients by finite groups. Then $n$th fold product $X^{\times n }$ is an object of $\bcC$, hence, the quotient $X^{\times n }/\Sigma_n$ is also in $\bcC$. Denote this quotient by $\Sym^n(X)$.  Then, we have a functor $\Sym^n:\bcC\into\bcC$. It is immediate  to observe that $\Sym^n\left(\Spec(k)\right)$ is isomorphic to the point $\Spec(k)$ for $n\geq1$. By convention, $\Sym^0$ will be the constant endofunctor of $\bcC$ which sends an object $X$ of $\bcC$ to the point $\Spec(k)$.

Let us fix $n\in\NN$.  Since $\bcC$ is a small category and $\Delta^{\op}\bcS$ is cocomplete, Theorem 3.7.2 of \cite{Bor94-1} asserts the existence of  the left Kan extension of the composite 
$$\bcC\stackrel{\Sym^n}{\longrightarrow} \bcC\stackrel{h}{\longrightarrow}\bcS$$ along the Yoneda embedding $h$. We denote it by $\Sym_g^n$ and called it the $n$th-fold {\em geometric symmetric power} of Nisnevich sheaves.  Explicitly,  $\Sym_g^n$ is described as follows. For a sheaf $\bcX$ in $\bcS$, we denote by $(h\downarrow \bcX)$ the comma category whose objects are arrows of the form $h_U\into \bcX$ for $U\in\ob(\bcC)$. Let $F_{\bcX}:(h\downarrow \bcX)\into\bcS$ be the functor which sends a morphism $h_U\into \bcX$ to the representable sheaf $h_{\Sym^nU}$. Then,  $\Sym_g^n(\bcX)$ is nothing but the colimit of the functor $F_{\bcX}$.  

The endofunctor $\Sym_g^n$ of $\bcS$ induces an endofunctor of   $\Delta^{\op}\bcS$ defined termwise.  By abuse of notation, we denote this functor by the same symbol $\Sym_g^n$, if no confusion arises.

\begin{example}\label{papexxv71}
 {\em Let $k$ be an algebraically closed field and fix a natural number $n$. For each $k$-scheme $X$ in $\bcC$,  the $n$th fold geometric symmetric power  $\Sym_g^n(h_X)$ of the representable functor $h_X$ coincides with the representable functor $h_{\Sym^nX}$.
 The section $\Sym_g^n(h_X)(\Spec(k))$ is nothing but the set of effective zero cycles of degree $n$ on $X$.   
}
\end{example}

Let $X$ be an object of $\Deltaop\bcC$. The $n$th fold symmetric power $\Sym^n(X)$ is the simplicial object on $\bcC$ whose terms are $\Sym^n(X)_i:=\Sym^n(X_i)$ for all $i\in\NN$. Thus, $\Sym^n$ induces a endofunctor of $\Deltaop\bcC$.

The Yoneda embedding $h$ of $\bcC$ into $\bcS$ induces an embedding $\Deltaop h$ of $\Deltaop\bcC$ into $\Deltaop\bcS$. Sometimes, we denote it by the same letter $h$.


\begin{lemma}\label{lerffs}
For each $n\in\NN$, $\Sym^n_g$ is isomorphic to the left Kan extension of the composite 
$$\xymatrix@C=10ex@R=10ex{\Deltaop\bcC\ar[r]^{\Sym^n}&\Deltaop \bcC\ar[r]^{\Deltaop h}&\Deltaop\bcS}$$ 
along $\Deltaop h$.  
\end{lemma}
\begin{proof}
Notice that $\Deltaop\bcC$ is a small category. Let $\bcX$ be a simplicial sheaf and let $i$ be a natural number. Let us consider the functor  $F_{\bcX_i}$ such that $\Sym^n_g(\bcX_i)=\colim F_{\bcX_i}$, as defined above. Observe that the functor $$J_{\bcX,i}:(\Deltaop\bcC\downarrow\bcX)\into(\bcC\downarrow\bcX_i)\,,$$ given by $(\Deltaop h_U\into \bcX)\mapsto  (\Deltaop h _{U_i}\into \bcX_i)$, is final. This implies that there is a natural isomorphism  $\colim (J_{\bcX,i}\circ F_{\bcX_i})\isom \colim F_{\bcX_i}$. The colimit of  $J_{\bcX,i}\circ F_{\bcX_i}$ is nothing but the $i$th term of $\Lan_{\Deltaop h}(\Deltaop h\circ \Sym^n)(\bcX)$. Thus, we get a natural isomorphism $$\Lan_{\Deltaop h}(\Deltaop h\circ \Sym^n)(\bcX)\isom \Sym^n_g(\bcX)$$
 for every object $\bcX$ in $\Deltaop\bcS$. 
\end{proof}

Since the $n$th fold geometric symmetric power $\Sym^n_g$ on $\Deltaop\bcS$ preserves terminal object, it induces an endofunctor of $\Deltaop\bcS_{\ast}$, denoted by the same symbol if no confusion arises.  We denote by $h^+$ the canonical functor from $\Deltaop\bcC$ to $\Deltaop\bcS_{\ast}$. 

\begin{corollary}
For each $n\in\NN$, the $n$th fold geometric symmetric power $\Sym^n_g$ on $\Deltaop\bcS_{\ast}$ is isomorphic to the left Kan extension of the composite 
$$\xymatrix@C=10ex@R=10ex{\Deltaop\bcC\ar[r]^{\Sym^n}&\Deltaop \bcC\ar[r]^{h^+}&\Deltaop\bcS_{\ast}}$$ 
along $h^+$.  
\end{corollary}
\begin{proof}
It follows from the previous lemma in view that the canonical functor from $\Deltaop\bcS$ to $\Deltaop\bcS_{\ast}$ is left adjoint. 
\end{proof}

\bigskip

{\em\noindent K\"unneth rules.}---  The Lemmas \ref{abs.tlt4}, \ref{abszb0} and Proposition \ref{pappro44s} will be useful to prove the K\"unneth rule for symmetric for schemes (Corollary \ref{abse7}), that is, for every $n\in\NN$ and for any two schemes $X$  and $Y$ on an admissible category, one has an isomorphism 
$$\Sym^n(X\amalg Y)\isom\coprod_{i+j=n}(\Sym^iX\times \Sym^jY)\,.$$
 We recall that   for a category $\bcC$ and a finite group $G$, the category $\bcC^G$ is the category of functors $G\into \bcC$, where $G$ is viewed as a category. A functor $G\into \bcC$ is identified with a $G$-object of $\bcC$. If $H$ is a subgroup of $G$, then the {\em restriction} functor $\res^G_H:\bcC^G\into \bcC^H$ sends a functor $G\into \bcC$ to the composite $H\incl G\into \bcC$. If $\bcC$ has finite coproducts and quotients by finite groups, then $\res^G_H$ is a right adjoint functor. Its left adjoint functor is called {\em corestriction} functor, and we denoted it by $\cor^G_H$.  
 
\begin{lemma}\label{abs.tlt4}
Let $\bcC$ be a category with finite coproducts and quotients by finite groups. Let $G$ be a finite group and let $H$ be a subgroup of $G$. If $X$ is an $H$-object of $\bcC$, then
    $$\cor^G_H(X)/G\isom X/H\,.$$
\end{lemma}
\begin{proof}
Suppose $X$ is an $H$-object of $\bcC$. We recall that 
 $\cor^G_0(X)$ coincides with the coproduct of $|G|$-copies of $X$, it is usually denoted by $G\times X$ in the literature. Observe that the group $G\times H$ acts canonically on $G\times X$. By definition, $\cor^G_H(X)$ is equal to $\colim_H(G\times X)$. One can also notice that $\colim_G(G\times X)=X$. Then, we have,
\begin{align*}
\cor^G_H(X)/G&=\colim_G\,\cor^G_H(X)&&\\
&=\colim_G\,\colim_H(G\times X)&&\\
&=\colim_H\,\colim_G(G\times X)&& \text{(change of colimits)}\\
&=\colim_HX\,.&&
\end{align*}
By definition, $X/H$ is equal to $\colim_H X$, thus we obtain that  $\cor^G_H(X)/G$  is isomorphic to $X/H$. 
\end{proof}

\bigskip 

\begin{lemma}\label{abszb0}
Let $\bcC$ be a symmetric monoidal category with finite coproducts and quotients by finite groups. Let $n,i,j$ be three natural numbers such that  $i,j\leq n$ and $i+j=n$, and let $X_0,X_1$ be two objects of $\bcC$. Then, the symmetric group $\Sigma_n$ acts on the coproduct $\bigvee_{k_1+\cdots +k_n=j} X_{k_1}\wedge\cdots \wedge X_{k_n}$ by permuting the indices of the factors, and one has an isomorphism 
$$\bigslant{\left(\bigvee_{k_1+\cdots +k_n=j} X_{k_1}\wedge\cdots \wedge X_{k_n}\right)}{\Sigma_n}\isom \Sym^iX_0\wedge \Sym^j X_1$$
\end{lemma}
\begin{proof}
After reordering of factors in a suitable way, we can notice that  the coproduct $\bigvee_{k_1+\cdots +k_n=j} X_{k_1}\wedge\cdots \wedge X_{k_n}$ is isomorphic to the coproduct of  $\binom{n}{i}$-copies of the term $X_{0}^{\wedge i}\wedge X_{1}^{\wedge j}$, in other words, it is isomorphic to $\cor^{\Sigma_n}_{\Sigma_i\times\Sigma_j}(X_{0}^{\wedge i}\wedge X_{1}^{\wedge j})$ which is a $\Sigma_n$-object.  
By Lemma \ref{abs.tlt4}, we have an isomorphism 
$$\bigslant{\left(\cor^{\Sigma_n}_{\Sigma_i\times\Sigma_j}(X_{0}^{\wedge i}\wedge X_{1}^{\wedge j})\right)}{\,\Sigma_n\,}\isom 
\bigslant{(X_{0}^{\wedge i}\wedge X_{1}^{\wedge j})}{(\Sigma_i\times\Sigma_j)}$$ 
and the right-hand-side is isomorphic to $ \Sym^iX_0\wedge \Sym^j X_1$, which implies the expected isomorphism.  
\end{proof}

\begin{proposition}\label{pappro44s}
Suppose $\bcC$ is a category as in Lemma \ref{abszb0}. Let $X_0$, $X_1$ be two objects of $\bcC$. For any integer $n\geq1$, there is an isomorphism 
\begin{equation}\label{radqee2}
\Sym^n(X_0\vee X_1)\isom \bigvee_{i+j=n}(\Sym^i X_0\wedge \Sym^{j}X_1)\,.
\end{equation}
\end{proposition}
\begin{proof}
Let us fix an integer $n\geq1$. We have the following isomorphism,  
$$(X_0\vee X_1)^{\wedge n}\isom\bigvee_{0\leq j\leq n}\left( \bigvee_{ k_1+\cdots +k_n=j} X_{k_1}\wedge\cdots \wedge X_{k_n}\right)\,,$$
and for each index $0\leq j\leq n$, the symmetric group $\Sigma_n$ acts by permuting factors on the coproduct $\bigvee_{k_1+\cdots +k_n=j} X_{k_1}\wedge\cdots \wedge X_{k_n}$. Hence, we deduce that $\bigslant{(X_0\vee X_1)^{\wedge n}}{\Sigma_n}$ is isomorphic to the coproduct of the quotients $\bigslant{( \bigvee_{k_1+\cdots +k_n=j} X_{k_1}\wedge\cdots \wedge X_{k_n})}{\Sigma_n}$ for $j=0,\dots, n$. Finally,
by Lemma \ref{abszb0}, we obtain that $\Sym^n(X_0\vee X_1)$ is isomorphic to the coproduct $\bigvee_{0\leq j\leq n}(\Sym^{n-j} X_0\wedge \Sym^{j}X_1)$, thus we have the isomorphism  \eqref{radqee2}. 
\end{proof}

\begin{corollary}\label{abse7}
{\em Let $\bcC\subset \Sch/k$ be an admissible category. Then, for every integer $n\geq 1$ and for any two objects $X,Y$ of $\bcC$, we have an isomorphism 
$$\Sym^n(X\amalg Y)\isom \coprod_{i+j=n}(\Sym^i X\times \Sym^{j}Y)\,. $$
}
\end{corollary}

\begin{example}\label{paprem51}
{\em Let $\varphi: X\into X\amalg Y$ be a coprojection in $\bcC$. Since $\bcC$ is symmetric monoidal and has finite coproducts and  quotients by finite groups, the filtration of the morphism $\Sym^n(\varphi)$,   
 $$\Sym^n(X)=\tBox^n_0(\varphi)\into\tBox^n_1(\varphi)\into \cdots \into \tBox^n_n(\varphi)=\Sym^n(X\amalg Y)$$
defined in \cite{GoGu}, exists. By virtue of Corollary \ref{abse7}, for every index $0\leq i\leq n$, the object $\tBox^n_i(\varphi)$ is isomorphic to the coproduct $\coprod_{n-i\leq j\leq n}(\Sym^j X\times \Sym^{n-j}Y)$, and the morphism $\tBox^n_{i-1}(\varphi)\into\tBox^n_i(\varphi)$ is isomorphic to the canonical morphism
$${\coprod_{n-(i-1)\leq j\leq n}(\Sym^j X\times \Sym^{n-j}Y)}\into\coprod_{n-i\leq j\leq n}(\Sym^j X\times \Sym^{n-j}Y)\,.$$
 }
\end{example}

\begin{lemma}\label{paplmxq8}
Let $\mathJ$ be a category with finite coproducts and Cartesian products. Then, for every integer $n\geq 1$, the diagonal functor $\diag:\mathJ\into \mathJ^{\times n}$ is final (see \cite[page~213]{McLa71}).  
\end{lemma}
\begin{proof}
Let $A=(A_1,\dots,A_n)$ be an object of $\mathJ^{\times n}$. We shall prove that the comma category $A\downarrow \diag$, whose objects has the form $A\into\diag(B)$ for $B$ in $\mathJ$, is nonempty and connected. We set $B:=A_1\amalg\dots\amalg A_n$. For every index $0\leq i\leq n$, we have a canonical morphism $A_i\into B$, then we get a morphism from $A$ to $\diag(B)$.  Thus, the comma category $A\downarrow \diag$ is nonempty.  Let $B$ and $B'$ be two objects of $\mathJ$ and let $A=(A_1,\dots,A_n)$ be an object of $ \mathJ^{\times n}$.  Suppose that we have two morphisms: $(\varphi_1,\dots, \varphi_n)$ from $A$ to $\diag(B)$ and $(\varphi'_1,\dots, \varphi'_n)$ from $A$ to $\diag(B')$.  For every index $0\leq i\leq n$, we have a 
commutative diagram 
$$\xymatrix@C=5ex@R=5ex{&A_i\ar@/^1pc/[rdd]^{\varphi'_i}\ar@/_1pc/[ldd]_{\varphi_i}\ar@{.>}[d]^{\psi_i}&\\
&B\times B'\ar@/^0.5pc/[rd]\ar@/_0.5pc/[ld]&\\
B&&B'}$$
where the dotted arrow exists by the universal property of product. Notice that we get a morphism $(\psi_1,\dots, \psi_n)$ from $A$ to $\diag(B\times B')$ and a commutative diagram 
$$\xymatrix@C=5ex@R=5ex{&A\ar[rd]\ar[d]\ar[ld]&\\
\diag(B)&\diag(B\times B')\ar[r]\ar[l]& \diag(B')}$$ 
Thus, the comma category $A\downarrow \diag$ is connected. 
\end{proof}

\begin{corollary}\label{papcorgg18}
Let $\bcX$ be a sheaf in $\bcS$.  For every integer $n\geq 1$, we have an isomorphism   
 \begin{equation}
\bcX^{\times n}\isom \colim_{h_X\into \bcX}h_{X^n}\,,
\end{equation}  
\end{corollary}
\begin{proof}
Let $(h\downarrow \bcX)$ be a comma category and let us consider the functor $F_{\bcX,n}$ from $(h\downarrow \bcX)^{\times n}$ to $ \Deltaop\bcS$ defined by 
$$(h_{X_1}\into \bcX,\dots, h_{X_n}\into \bcX)\mapsto h_{X_1}\times\cdots\times h_{X_n}\,.$$ 
For every integer $n\geq 1$, let us consider the diagonal functor $\diag$ from $(h\downarrow \bcX)$ to $(h\downarrow \bcX)^{\times n}$.  
We recall that we have an isomorphism 
\begin{equation}
\bcX\isom \colim_{h_X\into \bcX}h_X\,,
\end{equation}  
where the colimit is taken from the comma category with objects $h_X\into \bcX$, for $X\in \bcC$ to the category of sheaves. Hence, we deduce an isomorphism  
$$\bcX^{\times n}\isom \colim F_{\bcX,n}$$   
By Lemma \ref{paplmxq8}, the functor $\diag$ is final, then  by Theorem 1 of \cite[page~213]{McLa71}, the canonical morphism $\colim (F_{\bcX,n}\circ \diag)\into\colim F_{\bcX,n}$ is an isomorphism. Notice that the composite functor $F_{\bcX,n}\circ \diag$ is given by $(h_X\into \bcX)\mapsto h_{X^n}$.  Finally, composing the following isomorphisms
 \begin{equation}
\bcX^{\times n}\isom \colim F_{\bcX,n}\isom \colim (F_{\bcX,n}\circ \diag)\isom \colim_{h_X\into \bcX}h_{X^n}\,. 
\end{equation}  
we get the required isomorphism. 
\end{proof}
 
\begin{lemma}\label{abs.lwm3n}
Let $F$, $G$ be two objects in $\bcS$. For any integer $n\geq1$, there is an isomorphism 
$$\Sym^n_g(F\amalg G)\isom \coprod_{i+j=n}(\Sym^i_g F\times\Sym^{j}_g G)\,. $$
\end{lemma} 
\begin{proof}
Let us fix an integer $n\geq 1$. By Corollary \ref{abse7}, for any two objects $X$ and $Y$ of $\bcC$, we have an isomorphism 
$$\Sym^n(X\amalg Y)\isom \coprod_{i+j=n}(\Sym^i X\times\Sym^{j}Y)\,. $$
 Since the Yoneda embedding $h:\bcC\into\bcS$ preserves finite product and coproduct, we get an isomorphism
 \begin{equation}\label{abseqjz8}
h_{\Sym^n(X\amalg Y)}\isom \coprod_{i+j=n}\left(h_{\Sym^i X}\times h_{\Sym^{j}Y}\right)\,. 
\end{equation}
By definition, we have the following three equalities,
\begin{align*}
&\Sym^n_g(h_{X})=h_{\Sym^n(X)}\,,\\&\Sym^n_g(h_{Y})=h_{\Sym^n(Y)}\,,\\&
\Sym^n_g(h_{X\amalg Y})=h_{\Sym^n(X\amalg Y)}.
\end{align*}
 $$$$ Replacing all these in \eqref{abseqjz8}, we get an isomorphism   
\begin{equation}\label{abseqjz9}
\Sym^n_g(h_{X}\amalg h_{Y})\isom \coprod_{i+j=n}\left(\Sym_g^i(h_X)\times \Sym_g^{j}(h_Y)\right)\,. 
\end{equation}
Let us consider the functor $\Phi_1:\bcC\times \bcC\into \bcS$ which sends a pair  $(X,Y)$ to the $n$th fold geometric symmetric power $\Sym^n_g(h_{X}\amalg h_{Y})$, and the functor $\Phi_2:\bcC\times \bcC\into\bcS$ which sends a pair $(X,Y)$ to $\coprod_{i+j=n}(\Sym_g^ih_X\times \Sym_g^{j}h_Y)$. Let
 $$\Lan \Phi_1,\Lan \Phi_2:\bcS\times\bcS\into \bcS$$
 be the left Kan extension of $\Phi_1$ and $\Phi_2$, respectively, along the embedding $h\times h$ from $\bcC\times \bcC$ into $\bcS\times\bcS$. By \cite[Prop.~3.4.17]{Bor94-3} $\bcS$ is a Cartesian closed, hence, using Corollary \ref{papcorgg18} one deduces that the functor $\Lan \Phi_1$ is nothing but the functor that sends a pair $(F,G)$ to $\Sym^n_g(F\amalg G)$. Analogously, $\Lan \Phi_2$ sends a pair $(F,G)$ to $\coprod_{i+j=n}(\Sym^i_g F\times\Sym^{j}_g G)$. Finally, from the isomorphism \eqref{abseqjz9}, we have $\Phi_1\isom \Phi_2$,  which implies that $\Lan \Phi_1$ is isomorphic to $\Lan \Phi_2$. This proves the lemma.   
\end{proof}

\begin{corollary}[K\"unneth rule]\label{radccttss}
Let $\bcX$, $\bcY$ be two objects in $\Delta^{\op}\bcS$. For any integer $n\geq1$, there is an isomorphism 
$$\Sym^n_g(\bcX\amalg \bcY)\isom \coprod_{i+j=n}(\Sym^i_g \bcX\times \Sym^{j}_g\bcY)\,. $$
\end{corollary}
\begin{proof}
It follows from Lemma \ref{abs.lwm3n}. 
\end{proof}

\begin{lemma}\label{potdb57}
Let $F$, $G$ be two objects in $\bcS_{\ast}$. For any integer $n\geq1$, there is an isomorphism 
$$\Sym^n_g(F\vee G)\isom \bigvee_{i+j=n}(\Sym^i_g F\wedge\Sym^{j}_g G)\,. $$
\end{lemma}
\begin{proof}
The proof is similar as in lemma \ref{abs.lwm3n}. In this case we define two functor $\Phi_1$ and $\Phi_2$ from $\bcC_{+}\times\bcC_+$ to $\bcS_{\ast}$ such that $\Phi_1$ takes a pair $(X_+,Y_+)$ to $\Sym^n_g(h_{X_+}\vee h_{Y_+})$ and $\Phi_2$ takes a pair $(X_+,Y_+)$ to $\bigvee_{i+j=n}(\Sym_g^ih_{X_+}\wedge \Sym_g^{j}h_{Y_+})$. Hence we prove that the left Kan extensions of $\Phi_1$ and $\Phi_2$, along the canonical functor $\bcC_+\times\bcC_+\into\bcS_{\ast}$, are isomorphic.  
\end{proof}

\begin{corollary}[Pointed version of K\"unneth rule]\label{radccttff}
Let $\bcX$, $\bcY$ be two objects in $\Delta^{\op}\bcS_{\ast}$. For any integer $n\geq1$, there is an isomorphism 
$$\Sym^n_{g}(\bcX\vee \bcY)\isom \bigvee_{i+j=n}(\Sym^i_{g} \bcX\wedge \Sym^{j}_{g}\bcY)\,. $$
\end{corollary}
\begin{proof}
It is a consequence of Lemma \ref{potdb57}.  
\end{proof}

\bigskip


\bigskip
 A {\em coprojection sequence} in $\Delta^{\op}\bcS_{\ast}$ is a diagram of the form 
$\bcX\into \bcY\into \bcY/\bcX$,
where $\bcX\into \bcY$ is a termwise coprojection.
Termwise coprojections are particular examples of cofibre sequences. 

\begin{proposition}\label{fertsga}
For each $n\in\NN$, the functor $\Sym^n_g$ preserves termwise coprojections. 
\end{proposition}
\begin{proof}
It follows from Lemma \ref{abs.lwm3n} for the unpointed case and from Lemma \ref{potdb57} for the pointed case. 
\end{proof}

{\em\noindent K\"unneth towers.}--- 
Let $f:\bcX\into\bcY$ be a morphism of pointed simplicial sheaves. A filtration of $\Sym^n_g(f)$ in $\Deltaop\bcS_{\ast}$, 
$$\Sym^n_g(\bcX)=\bcL^n_0(f)\into \bcL^n_1(f)\into\cdots\into \bcL^n_n(f)=\Sym^n_g(\bcY)\,,$$
 is called ({\em geometric}) {\em K\"unneth tower} of $\Sym^n_g(f)$, if for every index $1\leq i\leq n$, there is an isomorphism 
$$\cone\Big(\bcL^n_{i-1}(f)\into \bcL^n_{i}(f)\Big)\isom \Sym^{n-i}_g(\bcX)\wedge \Sym^{i}_g(\bcX)$$ 
in $\bcH_{\ast}(\bcC_{\Nis},\AF^1)$. 

\medskip

Later, we shall prove that the $n$th fold geometric symmetric power of a $I_{\proj}$-cell complex has canonical K\"unneth towers, see Proposition \ref{gsjpr5}. In the next paragraphs, $h^+$ will denote the canonical functor from $\Deltaop\bcC_+$ to $\Deltaop\bcS_{\ast}$. 
 
 \medskip

A pointed simplicial sheaf is called {\em representable}, if it is isomorphic to a simplicial sheaf of the form $h^+(X)$, where $X$ is a simplicial object on $\bcC$. For example, the simplicial sheaves $\Delta_U[n]_+$ and $\partial\Delta_U[n]_+$, where $U$ is an object of $\bcC$ and $n\in\NN$, are both representable.

A {\em directed} set is a nonempty set with a preorder such that every pair of elements has an upper bound. A {\em directed colimit} is the colimit of a directed diagram, i.e. a functor whose source is a directed set. 

We shall denote by $(\Deltaop\bcC)^{\#}_+$ the full subcategory of $\Deltaop\bcS_{\ast}$ generated by directed colimits of representable simplicial sheaves. By Theorem 6.1.8 of \cite{KS06}, $(\Deltaop\bcC)^{\#}_+$ is closed under directed colimits.

\label{lpagrttf}
Let $\bcD$ be small category with finite coproducts. The {\em coequalizer completion} of   $\bcD$ is a category $\bcD_{\coeq}$ with finite colimits together with a finite coproducts preserving functor $\Phi_{\bcD}: \bcD\into \bcD_{\coeq}$, such that for any finite coproducts preserving functor $F: \bcD\into \bcE$, there exists a unique finite colimit preserving functor $\overline{F}: \bcD_{\coeq}\into \bcE$ with $\overline{F}\circ\Phi_{\bcD}= F$. The objects of $\bcD_{\coeq}$ are reflexive pairs in $\bcD$, that is, parallel pairs $U\rightrightarrows V$ with common section (see \cite{BC95}). The functor $\Phi_{\bcD}$ sends an object $U$ of $\bcD$ to the reflexive pair
$$\xymatrix{U\ar@<0.7ex>[r]^{\id}\ar@<-0.7ex>[r]_{\id}&U}\,.$$

In the next paragraphs, we shall consider the coequalizer completion $\bcC_{\coeq}$ of an admissible category $\bcC$. For simplicity, we shall write $\Phi$ instead of $\Phi_{\bcC}$. If $G$ is a finite group acting on an object $X$ of $\bcC$, then $\Phi(X/G)$ is canonically isomorphic to $\Phi(X/G)$. In particular, for every $n\in\NN$ and every object $X$ of $\bcC$, we have a canonical isomorphism 
\begin{equation}\label{tffbb39}
\Phi(\Sym^nX)\isom \Sym^n\Phi(X)\,.
\end{equation}
where $\Sym^n\Phi(X)$ denotes the quotient $\Phi(X)^{\tensor}/\Sigma_n$. The cartesian product of $\bcC$ induces, in a natural way, a symmetric monoidal product on 
$\bcC_{\coeq}$. Let us write $\bcC_{\coeq,+}$ for the category $(\bcC_{\coeq})_+$. Thus, 
$\bcC_{\coeq,+}$ is a symmetric monoidal category; its product will be denoted by $\wedge$. The functor $\Phi$ induces a functor of $\bar{\Phi}$ from $\Deltaop\bcC_+$ to $\Deltaop\bcC_{\coeq,+}$. Notice that the category $\Deltaop\bcC_{\coeq,+}$ is also symmetric monoidal and has finite colimits. From \eqref{tffbb39}, we deduce an isomorphism 
  \begin{equation}\label{tffbb30}
\bar{\Phi}(\Sym^nX)\isom \Sym^n\bar{\Phi}(X)\,.
\end{equation}
for every simplicial object $X$ on $\bcC_+$. 
By the universal property of $\bcC_{\coeq,+}$, there exists a functor $$\bar{h}:\Deltaop\bcC_+\into \Deltaop\bcS_{\ast}$$
such that $\bar{\phi}\circ\bar{h}=h^+$. From \eqref{tffbb30}, we obtain an isomorphism   
\begin{equation}\label{tffbb31}
h^+(\Sym^nX)\isom \bar{h}(\Sym^n\bar{\Phi}(X))\,.
\end{equation}
for every object $X$ of $\Deltaop\bcC_+$.

\begin{proposition}\label{exvvg4}
For every $n\in\NN$, the $n$th fold geometric symmetric power of a morphism of representable simplicial sheaves has a canonical K\"unneth tower.
\end{proposition}
\begin{proof}
Let $\varphi:X\into Y$ be a morphism in $\Deltaop\bcC_+$. Since the category $\Delta^{\op}\bcC_{\coeq,+}$ is symmetric  monoidal and has finite colimits, we follow \cite{GoGu} to obtain a filtration of $\Sym^n(\Phi(\varphi))$, 
\begin{equation}
\Sym^n(\Phi(X))=\tilde{\Box}^n_0\longrightarrow \tilde{\Box}^n_1\longrightarrow\cdots\longrightarrow \tilde{\Box}^n_n=\Sym^n(\Phi(Y))
\end{equation} 
such that, for every index $1\leq i\leq n$, there is an isomorphism
\begin{equation}
\tilde{\Box}^n_i/\tilde{\Box}^n_{i-1}\isom\Sym^{n-i}(\Phi(X))\wedge \Sym^i\Big(\Phi(Y)/\Phi(X)\Big)
\end{equation}
In view of \eqref{tffbb31}, the above sequence induces a filtration of $\Sym^n_{g}(h^+(\varphi))$,  
$$\bar{h}(\tilde{\Box}^n_0)\longrightarrow \bar{h}(\tilde{\Box}^n_1)\longrightarrow\cdots\longrightarrow\bar{h}(\tilde{\Box}^n_n)\,.$$ 
Since $\bar{h}$ preserves finite colimits and monoidal product, the above sequence is a K\"unneth tower of $\Sym^n_{g}(h^+(\varphi))$. 
\end{proof}

\begin{proposition}\label{prfkvvu}
Let $f:\bcX\into \bcX\vee \bcY$ be a coprojection of simplicial sheaves, where $\bcX$ and $\bcY$ are in  $(\Deltaop\bcC)^{\#}_+$. Then , for every $n\in\NN$, the K\"unneth tower of $\Sym^n_g(f)$ is a sequence 
$$\bcL^n_0(f)\longrightarrow \bcL^n_1(f)\longrightarrow\cdots\longrightarrow \bcL^n_n(f)\,, $$
such that each term $\bcL^n_i(f)$ is isomorphic to the coproduct 
\begin{equation}\label{tvvxzpp2}
\bigvee_{(n-i)\leq l\leq n}(\Sym^l_g \bcX\wedge \Sym^{n-l}_g\bcY)\,.
\end{equation}
\end{proposition}
 \begin{proof}
 Let us write $\bcX:=\colim_{d\in D}h^+(X_{d})$ and $\bcY:=\colim_{e\in E}h^+(Y_{e})$, where $X_{d}$ and $Y_{e}$ are in $\Deltaop\bcC$. Then, the coproduct $\bcX\vee \bcY$ is isomorphic to the colimit 
 $$\colim_{(d,e)\in D\times E}(h^+(X_{d})\vee h^+(Y_{e}))\,,$$ and $f$ is the colimit of the coprojections $h^+(X_{d})\into h^+(X_{d})\vee h^+(Y_{e})$ over all pairs $(d,e)$ in $D\times E$. Let us write $\varphi_{d,e}$ for the coprojection $X_{d}\into X_{d}\vee Y_{e}$.  By Example \ref{paprem51}, the morphism $\Sym^n(\varphi_{d,e})$ has a K\"unneth tower whose $i$th term has the form 
$$\tBox_i(\varphi_{d,e})\isom \bigvee_{(n-i)\leq l\leq n}(\Sym^l X_{d}\wedge \Sym^{n-l}Y_{e})\,.$$
 Hence, we have an isomorphism 
 $$h^+(\tBox^n_i(\varphi_{d,e}))\isom \bigvee_{(n-i)\leq l\leq n}\Big(\Sym^l_g h^+(X_{d})\wedge \Sym^{n-l}_g h^+(Y_{e})\Big)\,.$$
Taking colimit over $D\times E$,  we get that $$\bcL^n_i(f):=\colim_{(d,e)\in D\times E}h^+(\tBox^n_i(\varphi_{d,e}))$$
 is isomorphic to the coproduct \eqref{tvvxzpp2}, as required. 
 \end{proof}

\begin{lemma}\label{tspecgq}
Let 
\begin{equation}\label{fddxzze}
\xymatrix@C=10ex@R=10ex{\bcA\ar[d]_f\ar[r]^g&\bcX\ar[d]^{f'}\\
\bcB\ar[r]_{g'}&\bcY
}
\end{equation}
be a cocartesian in $\Deltaop\bcS_{\ast}$, where $f$ is the image of a termwise coprojection in $\Deltaop\bcC$ through the functor $h^+$. One has the following assertions:
\begin{enumerate}
\item[($a$)]  If $\bcX$ is a representable simplicial sheaf, then $\bcY$ is so, and $f'$ is the image of a termwise coprojection in $\Deltaop\bcC$ through the functor $h^+$. 
\item[($b$)] Suppose that $\bcA$ and $\bcB$ are compact objects. If $\bcX$ is in $(\Deltaop\bcC)^{\#}_+$, then so is $\bcY$. Moreover, if $\bcX$ is a directed colimit of representable simplicial sheaves which are compact, then so is $\bcY$.
\end{enumerate}
\end{lemma}
\begin{proof}
($a$). By hypothesis, there are a termwise coprojection $\varphi:A\into B$ and a morphism $\psi: A\into X$ in $\Deltaop\bcC$ such that $f=h^+(\varphi)$ and $g=h^+(\psi)$. 
Since $\varphi$ is a termwise coprojection, we have a cocartesian square
$$\xymatrix@C=10ex@R=10ex{A\ar[d]_{\varphi}\ar[r]^{\psi}&X\ar[d]^{\varphi'}\\
B\ar[r]_{\psi'}&Y
}$$
in $\Deltaop\bcC$, where $\varphi'$ is a termwise coprojection. As $h$ preserves finite coproducts, we deduce that $\bcY$ is isomorphic to $h^+(Y)$ and $f'=h^+(\varphi')$. 

($b$).  Suppose that $\bcX$ is the colimit of a directed diagram $\{\bcX_{d}\}_{d\in D}$, where $\bcX_{e}$ is a representable simplicial sheaf. Since $\bcA$ is compact, there exists an element $e\in D$ such that the morphism $g$ factors through an object $\bcX_{e}$. For every ordinal $d\in D$ with $e\leq d$, we consider the following cocartesian square
$$\xymatrix@C=10ex@R=10ex{\bcA\ar[d]_{f}\ar[r]^-{}&\bcX_{d}\ar[d]\\\bcB\ar[r]&\bcB\amalg_{\bcA}\bcX_{d}}$$
By item ($a$), the simplicial sheaf $\bcB\amalg_{\bcA}\bcX_{d}$ is representable. Therefore, we get a cocartesian square
$$\xymatrix@C=10ex@R=10ex{\bcA\ar[d]_{f}\ar[r]^-{g}&\colim_{d\in D}\bcX_{d}\ar[d]\\\bcB\ar[r]&\colim_{\underset{d\in D}{e\leq d}}(\bcB\amalg_{\bcA}\bcX_{d})}$$
as required. 
\end{proof}

\begin{lemma}\label{lemff73}
Every $I_{\proj}$-cell complex of $\Deltaop\bcS_{\ast}$  is a directed colimit of representable simplicial sheaves which are compact. In particular, every $I_{\proj}$-cell complex of $\Deltaop\bcS_{\ast}$ is in $(\Deltaop\bcC)^{\#}_+$.
\end{lemma}
\begin{proof}
Since an element of $I_{\proj}$-$cell$ is a transfinite composition of pushouts of element of $I_{\proj}$, the lemma follows by transfinite induction in view of Lemma \ref{tspecgq} ($b$). 
\end{proof}

\section{Geometric symmetric powers under $\AF^1$-localization}
\label{weakequiv}

In this section, we follow the ideas of Voevodsky \cite{Voe10} in order to prove that geometric symmetric powers preserve $\AF^1$-weak equivalences between simplicial Nisnevich sheaves which termwise are coproducts of representable sheaves.   We also prove the existence of the left derived functors associated to geometric symmetric powers.    

\bigskip

Let $\bcC$ be an admissible category of schemes over a field $k$. For an integer $n\geq1$, the category $\bcC^{\Sigma_n}$ denotes the category of functors $\Sigma_n\into \bcC$, where $\Sigma_n$ is viewed as a category. We recall that $\bcC^{\Sigma_n}$ can be viewed as the category of $\Sigma_n$-objects of $\bcC$.  Let $X$ be an $\Sigma_n$-object on $\bcC$ and let $x\in X$. The {\em stabilizer} of $x$ is by definition the subgroup $\stab(x)$ consisting of elements $\sigma\in\Sigma_n$ such that $\sigma.x=x$. A family of  morphisms $\{f_i:U_i\into X\}_{i\in I}$ in $\bcC^{\Sigma_n}$ is called {\em $\Sigma_n$-equivariant Nisnevich covering} if  each morphism $f_i$, viewed as a morphism of $\bcC$, is \'etale and we have the following property: for each point $x\in X$, viewed as an object of $\bcC$, there exist an index $i\in I$ and a point $y\in U_i$ such that: $f_i(y)=x$, the canonical homomorphism of residual fields $k(x)\into k(y)$ is an isomorphism, and the induced homomorphisms of groups $\stab(y)\incl\stab(x)$ is an isomorphism.  Let $(\bcC^{\Sigma_n})_{\Nis}$ be the site consisting of $\bcC^{\Sigma_n}$ and the Grothendieck topology formed by the $\Sigma_n$-equivariant Nisnevich coverings. We denote by $\bcS^{\Sigma_n}$ the category of sheaves on $(\bcC^{\Sigma_n})_{\Nis}$.  

 \begin{remark}
 {\em 
For $n=1$, a $\Sigma_n$-equivariant Nisnevich covering is a usual Nisnevich covering in $\bcC$.}  
 \end{remark}

A Cartesian square in  $\bcC^{\Sigma_n}$ of the form \eqref{papeq10} is an {\em elementary distinguished  square} if $p$ is an \'etale morphism and $j$ is an open embedding when we forget the action of $\Sigma_n$, such that the morphism of reduced schemes $p^{-1}(X-U)_{\red}\into (X-U)_{\red}$ is an isomorphism. Notice that when $n=1$, this definition coincide with the usual definition of an elementary distinguished square.  The above square induces a diagram $$\xymatrix@C=10ex@R=10ex{\Delta_Y[0]_+\vee \Delta_Y[0]_+\ar[d]\ar[r]&\Delta_Y[0]_+\wedge \Delta[1]_+\\
\Delta_U[0]_+\vee\Delta_V[0]_+&}$$
We denote by  $K_{\mathQ}$ the pushout  in $\Delta^{\op}\bcS^{\Sigma_n}_{\ast}$ of the above diagram and denote by $\mathG_{\Sigma_n,\Nis}$ the set of morphisms in $\bcC^{\Sigma_n}$ of canonical morphisms from $K_{\mathQ}$ to $\Delta_X[0]_+$.
The set $\mathG_{\Sigma_n,\Nis}$ is called {\em set of generating Nisnevich equivalences}.
On the other hand, we denote by $\mathP_{\Sigma_n,\AF^1}$ the set of morphisms in $\bcC^{\Sigma_n}$ which is isomorphic to  the projection from $\Delta_X[0]_+\wedge\Delta_{\AF^1}[0]_+$ to $\Delta_X[0]_+$, for $X$ in $\bcC^{\Sigma_n}$.  By Lemma 13 \cite[page~392]{Del09}, the class of $\AF^1$-weak equivalences in $\Deltaop\bcS^{\Sigma_n}$ coincides with the class 
\begin{equation}\label{papcllg}
\cl_{\bar{\Delta}}(\mathG_{\Sigma_n,\Nis}\cup \mathP_{\Sigma_n,\AF^1})\,.
\end{equation} 

\bigskip

We denote by $\Const:\bcC\into \bcC^{\Sigma_n}$ the functor which sends $X$ to the $\Sigma_n$-object $X$, where $\Sigma_n$ acts on $X$ trivially. Let $\colim_{\Sigma_n}:\bcC^{\Sigma_n}\into \bcC$ be the functor which sends $X$ to $\colim_{\Sigma_n} X=X/\Sigma_n$. By definition of colimit, the functor $\colim_{\Sigma_n}$ is left adjoint to 
the functor $\Const$.  It turns out that the functor $\Const$ preserves finite limits and it sends  Nisnevich coverings to  $\Sigma_n$-equivariant Nisnevich coverings. In consequence, the functor $\Const$ is continuous and the functor $\colim_{\Sigma_n}$ is cocontinuous. 

Let $\Lambda_n:\bcC\into \bcC^{\Sigma_n}$ be the functor which sends $X$ to the $n$th fold product $X^{\times n}$.   
 Then, the endofunctor $\Sym^n$ of $\bcC$ is nothing but the composition of $\colim_{\Sigma_n}$ with $\Lambda_n$.

  \begin{proposition}
 The cocontinuous functor $\colim_{\Sigma_n}:(\bcC^{\Sigma_n})_{\Nis}\into \bcC_{\Nis}$ is also continuous. In consequence, it is a morphism of sites.  
 \end{proposition}
 \begin{proof}
 See \cite[Prop.~43]{Del09} 
 \end{proof}
 
The previous proposition says that the functor $\colim_{\Sigma_n}$ is a morphism of sites, then it induces an adjunction between the inverse and direct image functors,
  $$(\colim_{\Sigma_n})_*:\bcS\rightleftarrows \bcS^{\Sigma_n}:(\colim_{\Sigma_n})^*\,.$$
Hence, one has a commutative diagram up to isomorphisms
  \begin{equation}\label{papeq330}
  \xymatrix@C=10ex@R=10ex{(\bcC^{\Sigma_n})_{\Nis}\ar[d]_{h}\ar[rr]^{\colim_{\Sigma_n}}&&\bcC_{\Nis}\ar[d]^{h}\\
\bcS^{\Sigma_n}\ar[rr]_{(\colim_{\Sigma_n})^*}&&\bcS}
  \end{equation}
where $h$ is the Yoneda embedding. We denote  by 
 $$\gamma_n: \Delta^{\op}\bcS^{\Sigma_n}_{\ast} \longrightarrow \Delta^{\op}\bcS$$
 the functor induced by $(\colim_{\Sigma_n})^*$  defined termwise.   From the diagram \eqref{papeq330}, we deduce that $\gamma_n$ preserve terminal object, then it induces a functor $$\gamma_{n,+}: \Delta^{\op}\bcS^{\Sigma_n}_{\ast}\longrightarrow \Delta^{\op}\bcS_{\ast}\,.$$
 
We write $\tilde{\Lambda}_n$ for the left Kan extension of the composite
 $\bcC\stackrel{\Lambda_n}{ \longrightarrow}\bcC^{\Sigma_n}\stackrel{h}{ \longrightarrow}\bcS^{\Sigma_n}$
 along the Yoneda embedding $h:\bcC \longrightarrow\bcS$. Denote by 
$$\lambda_n:\Delta^{\op}\bcS \longrightarrow\Delta^{\op}\bcS^{\Sigma_n}_{\ast}$$ 
the functor induced by $\tilde{\Lambda}_n$ defined termwise. Since $\tilde{\Lambda}_n$ preserves terminal objects,  the functor  $\lambda_n$ does so, hence it induces a functor 
$$\lambda_{n,+}:\Delta^{\op}\bcS_{\ast} \longrightarrow\Delta^{\op}\bcS^{\Sigma_n}_{\ast}\,.$$ 

\medskip

 The following lemmas will be used in Theorem \ref{radthind107}. 
   
\begin{lemma}\label{papll12}
Left adjoint functors preserves left Kan extensions, in the following sense. Let $L:\bcE\into \bcE'$ be a left adjoint functor. If $\Lan_G F$ is the left Kan extension of a functor $F:\bcC\into \bcE$ along a functor $G:\bcC\into \bcD$, then the composite $L\circ \Lan_G F$ is the left Kan extension of the composite $L\circ F$ along $G$.  
\end{lemma} 
\begin{proof}
See \cite[Lemma~1.3.3]{Riehl14}.
\end{proof}

\begin{lemma}\label{paple811}
 For every natural $n$, the endofunctor $\Sym_g^n$ of $\Delta^{\op}\bcS$ is isomorphic to the composition $\gamma_n\circ \lambda_n$. Similarly, $\Sym_g^n$ as an endofunctor of $\Delta^{\op}\bcS_{\ast}$ is isomorphic to  the composition $\gamma_{n,+}\circ \lambda_{n,+}$. 
\end{lemma}
\begin{proof}
Since the functors $\Sym_g^n$, $\gamma_n$ and $\lambda_n$ are termwise, it is enough to show that $\Sym_g^n$, as a endofunctor of $\bcS$, is isomorphic to the composition 
of $\tilde{\Lambda}_n$ with $(\colim_{\Sigma_n})^*$.  Indeed, as the functor $(\colim_{\Sigma_n})^*$ is left adjoint, Lemma \ref{papll12} implies that the composite 
 \begin{equation}\label{papeq55}
 \xymatrix@C=10ex@R=10ex{\bcS\ar[r]^-{\tilde{\Lambda}_n}&\bcS^{\Sigma_n}\ar[rr]^-{(\colim_{\Sigma_n})^*}&&\bcS}
 \end{equation}
 is the left Kan extension of the composite 
$$\xymatrix@C=10ex@R=10ex{\bcC\ar[r]^-{\Lambda_n}&\bcC^{\Sigma_n}\ar[r]^-{h}&\bcS^{\Sigma_n}\ar[rr]^-{(\colim_{\Sigma_n})^*}&&\bcS}$$
 along the embedding $h:\bcC\into\bcS$. 
Now, in view of the commutativity of diagram \eqref{papeq330}, the preceding composite is  isomorphic to the composite 
 $$\xymatrix@C=10ex@R=10ex{ \bcC\ar[r]^-{\Lambda_n}&\bcC^{\Sigma_n}\ar[rr]^-{\colim_{\Sigma_n}}&&\bcC\ar[r]^-{h}&\bcS}
$$
but it is isomorphic to the composite $\bcC\stackrel{\Sym^n}{\longrightarrow}\bcC\stackrel{h}{\into}\bcS$,
which implies that the composite \eqref{papeq55} is isomorphic to $\Sym_g^n$, as required.   
\end{proof}

 
We denote by $\bar{\bcC}_+$ the full subcategory of coproducts  of pointed objects of the form $(h_X)_+$  in $\bcS_{\ast}$ for objects $X$ in $\bcC$.  For every object $X$ in $\bcC$, the pointed sheaf $(h_X)_+$ is 
isomorphic to $h_{(X_+)}$. Indeed, $(h_X)_+$ is by definition equal to the coproduct $h_X\amalg h_{\Spec(k)}$ and this coproduct is isomorphic to the representable functor $h_{X\amalg \Spec(k)}$ which is equal to $h_{(X_+)}$. 

Similarly, we denote by $\bar{\bcC^{\Sigma_n}_+}$ the full subcategory of coproducts  of pointed objects $(h_X)_+$ in $\bcS^{\Sigma_n}_{\ast}$ for objects $X$ in $\bcC^{\Sigma_n}$. 

\vspace{0.5cm} 

\begin{theorem}[Voevodsky]\label{radthind107}
 Let $f:\bcX\into \bcY$ be a morphism in $\Delta^{\op}\bar{\bcC}_{+}$. If $f$ is an $\AF^1$-weak equivalence in $\Delta^{\op}\bcS_{\ast}$, then $\Sym^n_g(f)$ is an $\AF^1$-weak equivalence.
\end{theorem}
\begin{proof}
By Lemma \ref{paple811}, $\Sym_g^n$ is the composition $\gamma_{n,+}\circ \lambda_{n,+}$. The idea of the proof is to show that $\gamma_{n,+}$ and $\lambda_{n,+}$ preserve $\AF^1$-weak equivalences between objects which termwise are coproducts of representable sheaves.  The functor $\lambda_{n,+}$ sends morphisms of $\bfW_{\Nis,+}\cup\mathP_{\Nis,+}$ between objects in $\Delta^{\op}\bar{\bcC_{+}}$ to $\AF^1$-weak equivalences between objects in $\Delta^{\op}\bar{\bcC^{\Sigma_n}_{+}}$. Since
$\lambda_{n,+}$ preserves filtered colimits,  Lemma  2.20 of \cite{Voe10-5} implies that $\lambda_{n,+}$ preserves $\AF^1$-weak equivalence as claimed.  Similarly, in view of the class given in \eqref{papcllg}, we use again Lemma 2.20 of {\em loc.cit.} to prove that $\gamma_{n,+}$ 
sends $\AF^1$-weak equivalences between objects in $\Delta^{\op}\bar{\bcC^{\Sigma_n}_+}$ to $\AF^1$-weak equivalences, as required 
\end{proof}


We define the functor $\Phi:\Delta^{\op}\bar{\bcC}_{+}\into\bcH_{\ast}(\bcC_{\Nis}, \AF^1)$ as the composite 
$$\Delta^{\op}\bar{\bcC}_{+}\incl \Delta^{\op}\bcS_{\ast}\into \bcH_{\ast}(\bcC_{\Nis}, \AF^1)\,.$$ 
 where the first arrow is the inclusion functor and the second arrow is the localization functor with respect to the $\AF^1$-weak equivalences.  

\begin{lemma}\label{radpozhh}
Let $\bcC$ be an admissible category. The functor
$$\Phi:\Delta^{\op}\bar{\bcC}_{+}\into\bcH_{\ast}(\bcC_{\Nis}, \AF^1)$$
is a strict localization, that is,  for every morphism $f$ in $\bcH_{\ast}(\bcC_{\Nis}, \AF^1)$, there is a morphism $g$ of $\Delta^{\op}\bar{\bcC}_{+}$ such that the image $\Phi(g)$ is isomorphic to $f$. 
\end{lemma}
\begin{proof}
By Theorem 2.5 of \cite[page~71]{MoVoe99}, the category $\bcH_{\ast}(\bcC_{\Nis}, \AF^1)$ is the localization of the category  $\bcH_{\ast}(\bcC_{\Nis})$ with respect to the image of $\AF^1$-weak equivalences trough the canonical functor. Then, it is enough to prove that the canonical functor $\Delta^{\op}\bar{\bcC}_{+}\into \bcH_{\ast}(\bcC_{\Nis})$  is a strict localization. Indeed, let $f:\bcX\into\bcY$ be a morphism of pointed simplicial sheaves on the site $\bcC_{\Nis}$ representing a morphism in $\bcH_{\ast}(\bcC_{\Nis})$.  The functorial resolution $Q^{\proj}$ gives a commutative square
$$\xymatrix@C=5ex@R=10ex{
Q^{\proj}(\bcX)\ar[rrr]^{Q^{\proj}(f)}\ar[d]&&&Q^{\proj}(\bcY)\ar[d]\\
\bcX\ar[rrr]_f&&&\bcY}$$ 
 where the vertical arrows are object-wise weak equivalences. 
Since the object-wise weak equivalences are local weak equivalences,  the vertical arrows of the above diagram are weak equivalences. This implies that $f$ is isomorphic to $Q^{\proj}(f)$ in $\bcH_{\ast}(\bcC_{\Nis})$. Moreover, by Corollary \ref{radreplac10}, the morphism $Q^{\proj}(f)$ is in $\Delta^{\op}\bar{\bcC}_{+}$. 
\end{proof}

\begin{corollary}\label{radcozzml}
For each integer $n\geq 1$, there exists the left derived functor $L\Sym^n_g$ from $\bcH_{\ast}(\bcC_{\Nis}, \AF^1)$ to itself such that we have a commutative diagram up to isomorphism 
\begin{equation}\label{rad5tgf}
\xymatrix@C=10ex@R=10ex{\Delta^{\op}\bar{\bcC}_{+}\ar[d]_{\Phi}\ar[rr]^{\Sym^n_g}&&\Delta^{\op}\bcS_{\ast}\ar[d]^{}\\
\bcH_{\ast}(\bcC_{\Nis}, \AF^1)\ar[rr]_{L\Sym^n_g}&&\bcH_{\ast}(\bcC_{\Nis}, \AF^1)}
\end{equation}
where the right arrow is the localization functor. 
\end{corollary}
\begin{proof}
By theorem \ref{radthind107}, the functor $\Sym_g^n$ preserves $\AF^1$-weak equivalences between objects in $\Delta^{\op}\bar{\bcC}_{+}$. Hence, the composite 
$$\Delta^{\op}\bar{\bcC}_{+}\stackrel{\Sym_g^n}{\longrightarrow}\Delta^{\op}\bcS_{\ast}\longrightarrow \bcH_{\ast}(\bcC_{\Nis}, \AF^1)$$
sends $\AF^1$-weak equivalences to isomorphisms.  Then, by Lemma \ref{radpozhh} there exists a functor  $L\Sym^n_g$ such the diagram \eqref{rad5tgf} commutes and for every simplicial sheaf $\bcX$, the object $L\Sym^n_g(\bcX)$ is isomorphic to $\Sym^n_g(Q^{\proj}(\bcX))$ in $\bcH_{\ast}(\bcC_{\Nis}, \AF^1)$. 
\end{proof}

\bigskip

\section{Lambda structures}
 \label{lamstr}
Our goal in this section is to prove the main Theorem \ref{radmainth1} which asserts that the left derived geometric symmetric powers $L\Sym_g^n$, for $n\in\NN$ (see Corollary \ref{radcozzml}), induce a $\lambda$-structure on the pointed motivic homotopy category $\bcH_{\ast}(\bcC_{\Nis},\AF^1)$.  We start by giving the definition of $\lambda$-structure on the homotopy category of a symmetric monoidal model category as in \cite{GoGu}.

 Let $\bcC$ be a closed symmetric monoidal model category with unit $\one$. A {\em $\lambda$-structure}  on $\Ho(\bcC)$ is a sequence $\Lambda^*=(\Lambda^0,\Lambda^1,\Lambda^2,\dots)$ consisting of endofunctors  $\Lambda^n$ of $\Ho(\bcC)$ for $n\in\NN$, satisfying the following axioms:
\begin{enumerate}
\item[(i)] $\Lambda^0=\one$, $\Lambda^1=\id$
\item[(ii)] {\em (K\"unneth tower axiom)}. For any cofibre sequence $X\stackrel{f}{\into}Y \into Z$ in $\Ho(\bcC)$,  and any $n\in\NN$, there is a unique sequence
$$\Lambda^n(X)= L^n_0 \into  L^n_1\into\cdots \into L^n_i\into\cdots \into L^n_n=\Lambda^n(Y)$$
called {\em K\"unneth tower}, such that for any index $0\leq i\leq n$, the quotient $L^n_i/L^n_{i-1}$ in $\bcC$ is weak equivalent to the product  
$\Lambda^{n-i}(X)\wedge \Lambda^{i}(Z)$.
 \item[(iii)]  {\em (Functoriality axiom)}. For any morphism of cofibre sequences
 $$\xymatrix@C=10ex@R=10ex{X\ar[r]\ar[d]&Y\ar[r]\ar[d]&Z\ar[d]\\
 X'\ar[r]&Y'\ar[r]&Z'}$$
in $\Ho(\bcC)$, there is a commutative diagram 
  $$\xymatrix@C=4ex@R=10ex{\Lambda^n(X)= L^n_0\ar[r]\ar@<+25pt>[d]& L^n_1\ar[r]\ar[d] &L^n_2\ar[r]\ar[d] &\cdots&  \cdots\ar[r]&\ar[r]\ar[d] L^n_{n-1}\ar[r]\ar[d] & L^n_n=\Lambda^n(Y)\ar@<-25pt>[d]\\
\Lambda^n(X')= {L'}^n_0\ar[r]& {L'}^n_1\ar[r]\ar[r]&{L'}^n_2\ar[r] &\cdots&\cdots \ar[r]&\ar[r] {L'}^n_{n-1}\ar[r] &{ L'}^n_n=\Lambda^n(Y')
}$$
in $\Ho(\bcC)$, were the horizontal sequences are their respective K\"unneth towers. 
\end{enumerate} 
\medskip 

Let $\Sym^n$ be the abstract $n$th fold symmetric power defined on $\Delta^{\op}\bcS_{\ast}$, for $n\in\NN$, which is defined for every pointed simplicial sheaf $\bcX$ as the quotient $\Sym^n(\bcX):=(\bcX^{\wedge n})/\Sigma_n$. 

\begin{example}
{\em  The left derived functors $L\Sym^n$, for $n\in\NN$, provide a $\lambda$-structure on $\bcH_{\ast}(\bcC_{\Nis},\AF^1)$  (see  \cite[Theorem~57]{GoGu} for the proof in the context Nisnevich sheaves on the category of smooth schemes). Indeed,  the morphism $\Delta_{\AF^1}[0]\into\Delta_{\Spec(k)}[0]$ is a diagonalizable interval, meaning that  $\Delta_{\AF^1}[0]$ has a structure of symmetric co-algebra in the category $\Delta^{\op}\bcS$.   We claim that the class of cofibrations and the class of trivial cofibrations in $\Delta^{\op}\bcS$ are symmetrizable. Since cofibrations in $\Delta^{\op}\bcS$ are section-wise cofibrations of simplicial sets, it follows from Proposition 55 of \cite{GoGu} that cofibrations are symmetrizable. 
Let $f$ be a trivial cofibration in $\Delta^{\op}\bcS$. As $f$ is a cofibration, it is a symmetrizable cofibration. For every point $P$ of the site $\bcC_{\Nis}$, the induced morphism $f_P$ is a weak equivalence of simplicial sets.  By \cite[Lemma 54]{GoGu}, the $n$th fold symmetric power $\Sym^n(f_P)$ is also a weak equivalence. Since the morphism $\Sym^n(f)_P$ coincide with  $\Sym^n(f_P)$, we deduce that the $n$th fold symmetric power $\Sym^n(f)$ is a weak equivalence too. Hence, by  \cite[Corollary 54]{GoGu}, $f$ is a symmetrizable trivial cofibration. Finally, Theorem 38 and Theorem 22 of \cite{GoGu} imply the existence of left derived functors $L\Sym^n$, for $n\in\NN$, and they  provide a $\lambda$-structure on $\bcH_{\ast}(\bcC_{\Nis},\AF^1)$. 

}
\end{example}

\medskip

\begin{proposition}\label{radthcofibre7}
Let $\bcC$ be an admissible category. Every cofibre sequence in $ \bcH_{\ast}(\bcC_{\Nis},\AF^1)$ is isomorphic to a coprojection sequence of the form $$\bcA\into \bcB\into \bcB/\bcA\,,$$ where $\bcA\into \bcB$ is in $I_{\proj}$-$\cell$ and $\bcA$ is an $I_{\proj}$-cell complex. In particular, $\bcA\into\bcB$ is a morphism in $\Deltaop\bar{\bcC}_+$.   
\end{proposition}
\begin{proof}
 Let $\bcX\into\bcY \into \bcZ$ be a cofibre sequence in $ \bcH_{\ast}(\bcC_{\Nis},\AF^1)$, where $f$ is a cofibration from $\bcX$ to $\bcY$ in $\Delta^{\op}\bcS_{\ast}$, such that $\bcZ=\bcY/\bcX$. Let $Q^{\proj}$ be the cofibrant resolution given in section \ref{Prelim}. We write  $\bcA:=Q^{\proj}(\bcX)$  and consider the induced morphism $\bcA\into \bcX$.  By Corollary \ref{radcorsmll}, the composition of $\bcA\into \bcX$ with $f$ induces a commutative diagram 
$$\xymatrix@C=10ex@R=10ex{\bcA\ar[r]^{\alpha(f)}\ar[d]&\bcB\ar[d]^{\beta(f)}\\\bcX\ar[r]_f&\bcY}$$ 
where $\beta(f)$ is a sectionwise trivial fibration and $\alpha(f)$ is in $I_{\proj}$-$\cell$. By \cite[Prop. 6.2.5]{Hovey0}, the cofibre sequence $\bcA\stackrel{}{\rightarrow}\bcB\rightarrow \bcB/\bcA$
  is isomorphic to the cofibre sequence $\bcX\stackrel{[f]}{\into}\bcY \into \bcZ$ in $\bcH_{\ast}(\bcC_{\Nis}, \AF^1)$. 

\end{proof}

\begin{proposition}\label{gsjpr5}
Let $f:\bcX\into\bcY$ be a morphism in $I_{\proj}$-$\cell$, where $\bcX$ is an $I_{\proj}$-cell complex. Then, for each $n\in\NN$, $\Sym^n_g(f)$ has a functorial K\"unneth tower.
\end{proposition}
\begin{proof}
By virtue of Lemma \ref{lemff73}, $\bcX$ and $\bcY$ are directed colimit of representable simplicial sheaves. By Proposition 6.1.13 of \cite{KS06}, the morphism $f$ can be expressed as the colimit of a directed diagram $\{f_d\}_{d\in D}$ of morphisms of representable simplicial sheaves. Hence, by Proposition \ref{exvvg4}, the $n$th fold geometric symmetric power  $\Sym^n_g(f_{d})$ has a K\"unneth tower 

\begin{equation}\label{eqpp4231}\xymatrix{\bcL^n_0(f_{d})\ar[r]&\bcL^n_1(f_{d})\ar[r]&\cdots \ar[r]&\bcL^n_n(f_{d})}\,.
\end{equation}
For each index $0\leq i\leq n$, we define 
$$\bcL^n_i(f):=\colim_{d\in D}\bcL^n_i(f_{d})\,.$$ 
 Thus, we get a sequence 
\begin{equation}\label{eqpp4241}
\bcL^n_0(f)\longrightarrow \bcL^n_1(f)\longrightarrow\cdots\longrightarrow \bcL^n_n(f)\,. 
\end{equation}
Let us show that this gives a K\"unneth tower of $\Sym^n_g(f)$ that is functorial in $f$. 
Since the sequence \eqref{eqpp4231} is a K\"unneth tower of $\Sym^n_g(f_{d})$, we have an isomorphism
$$\bcL^n_i(f_{d})/\bcL^n_{i-1}(f_{d})\isom \Sym^{n-i}_g(h^+(X_{d}))\wedge \Sym^{i}_g(h^+(Y_{d})/h^+(X_{d}))\,.$$ Hence, taking the colimit, for $d\in D$, we get the an isomorphism 
\begin{equation}
\bcL^n_i(f)/\bcL^n_{i-1}(f)\isom \Sym^{n-i}_g(\bcX)\wedge \Sym^{i}_g(\bcY/\bcX)\,.
\end{equation}   
\end{proof}


\begin{lemma}\label{radlaz04}
The endofunctor $L\Sym^0_g$ of $\bcH_{\ast}(\bcC_{\Nis}, \AF^1)$ is the constant functor with value $\one$, where $\one$ is the object $\Delta_{\Spec(k)}[0]$ in $\bcH_{\ast}(\bcC_{\Nis}, \AF^1)$, and the endofunctor $L\Sym^1_g$ is the identity functor on $\bcH_{\ast}(\bcC_{\Nis}, \AF^1)$. 
\end{lemma}
\begin{proof}
Since $\Sym^0X=\Spec(k)$ for every scheme $X$ in $\bcC$, the endofunctor $\Sym^0$ of $\bcC$  is constant with value $\Spec(k)$. By the lef Kan extension, we deduce that $\Sym^0$ extends to an endofunctor $\Sym^0_g$ of $\Delta^{\op}\bcS$ given by $\bcX\mapsto\Delta_{\Spec(k)}[0]$. Hence, we deduce that $L\Sym^0_g$ is the endofunctor of $ \bcH_{\ast}(\bcC_{\Nis},\AF^1)$ given by $\bcX\mapsto\one$. On the other hand, for every scheme $X$ in $\bcC$, we have $\Sym^1X=X$. By the left Kan extension, we deduce that the endofunctor $\Sym^1_g$ of $\Delta^{\op}\bcS$ is the identity functor, then  $L\Sym^0_g$ is the identity functor on  $ \bcH_{\ast}(\bcC_{\Nis},\AF^1)$. 
\end{proof}

\medskip

Now, we are ready to state and prove our main theorem. 

\begin{theorem}\label{radmainth1}
 The endofunctors $L\Sym^n_g$, for $n\in \NN$, 
 provides a $\lambda$-structure on $\bcH_{\ast}(\bcC_{\Nis},\AF^1)$.    
\end{theorem}
\begin{proof}
By Lemma \ref{radlaz04}, $L\Sym^0_g$ is the constant functor with value $\one$, and $L\Sym^1_g$ is the identity functor on $\bcH_{\ast}(\bcC_{\Nis},\AF^1)$. Let $\bcX\into\bcY\into \bcZ$ be a cofibre sequence in $\bcH_{\ast}(\bcC_{\Nis},\AF^1)$ induced by a cofibration $f:\bcX\into\bcY$ in the injective model structure of $\Delta^{\op}\bcS_{\ast}$. 
By Proposition \ref{radthcofibre7}, we can assume that $f$ is in $I_{\proj}$-$\cell$ and $\bcX$ is an $I_{\proj}$-cell complex. 
Hence, by Proposition \ref{gsjpr5}, for each index $n\in\NN$, $\Sym^n_g(f)$ has a K\"unneth tower,
\begin{equation}\label{edd23}
\Sym^n_g(\bcX)=\bcL^n_0(f)\into\bcL^n_1(f)\into\cdots\into \bcL^n_n(f)=\Sym^n_g(\bcY)\,,
\end{equation}
which induces a K\"unneth tower, 
$$L\Sym^n_g(\bcX)=L\bcL^n_0(f)\into L\bcL^n_1(f)\into\cdots\into L\bcL^n_n(f)=L\Sym^n_g(\bcY)\,,$$ of $L\Sym^n_g(f)$ in $\bcH_{\ast}(\bcC_{\Nis},\AF^1)$. Finally, the functoriality axiom follows from the functionality of K\"unneth towers of the form \eqref{edd23}, see  Proposition \ref{gsjpr5}.
\end{proof}

\bigskip

\section{Morphisms of Lambda structures}
\label{MorphismL}

In this section we define a morphism between two $\lambda$-structures as a sequence of natural transformations which are compatible with their K\"unneth towers. We show the existence of a natural morphism of $\lambda$-structures from left derived categoric symmetric powers to the left derived geometric symmetric powers, see Theorem \ref{nattrath}. 
\medskip

Let us consider the monoidal product $\wedge:\Deltaop\bcS_{\ast}\times\Deltaop\bcS_{\ast}\into \Deltaop\bcS_{\ast}$.
We recall that, for any two morphisms $f:\bcX\into \bcY$ and $f':\bcX'\into \bcY'$ in $\Deltaop\bcS_{\ast}$, the {\em box operation} of $f$ and $f'$ is the object
$$\Box(f,f')=(\bcX\wedge \bcY')\vee_{\bcX\wedge \bcX'}(\bcY\wedge \bcX')$$
together with a universal morphism
$f\Box f':\Box(f,f')\into \bcY\wedge \bcY'$. The box operation $\Box$ is associative and commutative, therefore for a finite collection of morphisms of pointed simplicial sheaves $\{f_i:\bcX_i\into \bcY_i\,|\, i=1,\dots,n\}$, we have an induced morphism
$$f_1\Box\cdots\Box f_n:\Box(f_1,\dots,f_n)\into \bcY_1\wedge\cdots \wedge \bcY_n\,.$$
For each $n\geq2$, we write $\Box^n(f)=\Box(f,\cdots,f)$ and $f^{\Box n}=f\Box\cdots\Box f$. By convention, we set $\Box^1(f)=X$ and $f^{\Box 1}=f$.
Let $\{0,1\}$ be the category with two objects with only one non-identity morphism $0\into 1$. The $n$-fold cartesian product $\{0,1\}^n$ of $\{0,1\}$ is a category whose objects are $n$-tuples $(a_1,\dots,a_n)$ where each $a_i$ is $0$ or $1$, and a morphism $(a_1,\dots,a_n)\into (a'_1,\dots,a'_n)$ is characterized by the condition $a_i\leq a'_i$ for all $i=1,\dots,n$.
Notice that the giving of a functor $K:\{0,1\}\into \bcC$ is the same as giving two objects $K(0)=\bcX$, $K(1)=\bcY$ and a morphism $K(0\into 1)=f$ from $\bcX$ to $\bcY$. For any morphism $f:\bcX\into\bcY$ in $\Deltaop\bcS_{\ast}$ and any integer $n\geq 1$, let $K^n(f):\{0,1\}^n\into \Deltaop\bcS_{\ast}$ be the composition
of the $n$-fold cartesian product $K(f)^n$ of the functor $K(f):\{0,1\}\into \Deltaop\bcS_{\ast}$ with the functor $\wedge:(\Deltaop\bcS_{\ast})^n\into\Deltaop\bcS_{\ast}$ which sends an object $(\bcX_1,\dots,\bcX_n)$ to the product $\bcX_1\wedge\cdots\wedge\bcX_n$.
For instance, the functor $K^2(f)$ is the commutative diagram of the form
\[ \xymatrix@C=10ex@R=10ex{\bcX\wedge \bcX\ar[r]\ar[d]&\bcY\wedge \bcX\ar[d]\\
 \bcX\wedge \bcY\ar[r]&\bcY\wedge \bcY
  }
\]
For any $0\leq i\leq n$, we denote by $\{0,1\}^n_i$ the full subcategory of $\{0,1\}^n$ generated by $n$-tuples $(a_1,\dots,a_n)$ with $a_1+\cdots+a_n\leq i$.
We shall denote by $K^n_i(f)$ the restriction of $K^n(f)$ on $\{0,1\}^n_i$. For instance, if $n=2$, we have that $K^2_0(f)$ is $\bcX\wedge \bcX\wedge \bcX$, $K^2_1(f)$ is the diagram
 \[ \xymatrix{\bcX\wedge\bcX\ar[r]\ar[d]&\bcY\wedge\bcX\\
 \bcX\wedge\bcY&
  }
\]
and $K^2_2(f)=K^2(f)$.
We set
\[\Box^n_i(f)=\colim K^n_i(f).\]
Let $f:\bcX\into \bcY$ be a morphism in $\Deltaop\bcS_{\ast}$ as before. Notice that, as $K^n_0(f)=\bcX^{\wedge n}$, we have $\Box^n_0(f)=\bcX^{\wedge n}$, and since $K^n_n(f)=\bcY^{\wedge n}$, we have $\Box^n_n(f)=\bcY^{\wedge n}$.
In view of the sequence of sub-diagrams
$$K^n_0(f)\subset K^n_1(f)\subset\cdots \subset K^n_n(f)\,,$$
we have a sequence of morphisms in $\Deltaop\bcS_{\ast}$,
$$\bcX^{\wedge n}=\Box^n_0(f)\into\Box^n_1(f)\into \cdots \into \Box^n_n(f)=\bcY^{\wedge n}$$
and its composite is nothing but the $n$-fold product $f^{\wedge n}:\bcX^{\wedge n}\into \bcY^{\wedge n}$ of $f$.

\begin{lemma}
Let $f:\bcX\into \bcY$ be a morphism in $\Deltaop\bcS_{\ast}$ and fix a natural number $n$. The symmetric group $\Sigma_n$ acts naturally on each object $\Box^n_i(f)$ for all $i=1,\dots,n$.
\end{lemma}
\begin{proof}
 We observe that the symmetric group $\Sigma_n$ acts naturally on the category $\{0,1\}^n$. Any permutation $\sigma\in\Sigma_n$ induces an automorphism $\sigma:\{0,1\}^n\into \{0,1\}^n$
 taking a $n$-tuple $(a_1,\dots,a_n)$ to $(a_{\sigma(1)},\dots,a_{\sigma(1)})$.
 Notice that if $a_1+\cdots+a_n\leq i$ then $a_{\sigma(1)}+\cdots+a_{\sigma(n)}=a_1+\cdots+a_n\leq i$, so the subcategory $\{0,1\}^n_i$ is invariant under the action of $\Sigma_n$. Thus, every automorphism $\sigma:\{0,1\}^n\into \{0,1\}^n$ induces an automorphism $\sigma:\{0,1\}^n_i\into \{0,1\}^n_i$ for $1\leq i\leq n$.
 For any morphism $(a_1,\dots,a_n)\into (a'_1,\dots,a'_n)$ in $\{0,1\}^n$, we have a commutative square
 \[ \xymatrix@C=10ex@R=10ex{K_i^n(f)(a_1,\dots,a_n)\ar[r]^-{\sigma}\ar[d]&K^n_i(f)(a_{\sigma(1)},\dots,a_{\sigma(1)})\ar[d]\\
 K_i^n(f)(a'_1,\dots,a'_n)\ar[r]_-{\sigma}&K^n_i(f)(a'_{\sigma(1)},\dots,a'_{\sigma(1)})
  }
\]
Then, by the universal property of colimit, there is a unique automorphism $\phi_{\sigma}$ of $\Box^n_i$ such that we have a commutative diagram
  \[ \xymatrix@C=10ex@R=10ex{K_i^n(f)(a_1,\dots,a_n)\ar[r]^-{\sigma}\ar[d]&K^n_i(f)(a_{\sigma(1)},\dots,a_{\sigma(1)})\ar[d]\\
 \Box^n_i(f)\ar[r]_-{\phi_{\sigma}}&\Box^n_i(f)
  }
\]
where the vertical morphisms are the canonical morphism. Moreover, the map $\phi: \Sigma_n\into\Aut(\Box^n_i(f))$ given by $\sigma\mapsto\phi_{\sigma}$ is a homomorphism of groups. This gives an action of $\Sigma_n$ on $\Box^n_i(f)$. 
\end{proof}

In view of the previous lemma, for every $0\leq i\leq n$, we denote
$$L^n_i(f)=\Box^n_i(f)/\Sigma_n\,.$$
In particular, we have $L^n_0=\bcX^{\wedge n}/\Sigma_n=\Sym^n(\bcX)$ and $L^n_n=\bcY^{\wedge n}/\Sigma_n=\Sym^n(\bcY)$.
 We have the following commutative diagram,
 \[ \xymatrix@C=5ex@R=10ex{\bcX^{\wedge n}=\Box^n_0(f)\ar@/^1.5pc/[rrrr]^{f^{\wedge n}}\ar[r]\ar[d]&\Box^n_1(f)\ar[r]\ar[d]&\cdots\ar[r]&\Box^n_{n-1}(f)\ar[r]\ar[d]&\Box^n_n(f)=\bcY^{\wedge n}\ar[d]\\
  \Sym^n \bcX=L^n_0(f)\ar[r]\ar@/_1.5pc/[rrrr]_{\Sym^nf}&L^n_1(f)\ar[r]&\cdots\ar[r]&L^n_{n-1}(f)\ar[r]&L^n_n(f)=\Sym^n \bcY}
\]

%
%

\medskip

{\it \noindent Morphisms of $\lambda$-structures.}--- Let $\bcC$ be a closed symmetric monoidal model category with unit $\one$ and let $\Lambda^*$ and $\Lambda'^*$ be two $\lambda$-structures on $\Ho(\bcC)$. A morphism of $\lambda$-structures from $\Lambda^*$ to $\Lambda'^*$ consists of a sequence $\Phi^*=(\Phi^0, \Phi^1,\Phi^2,\dots)$ of natural transformations $\Phi^n$ from $\Lambda^n$ to $\Lambda'^n$ for $n\in\NN$, such that for any cofibre sequence $X\into Y \into Z$ in $\Ho(\bcC)$ and any $n\in\NN$, there a commutative diagram   
  $$\xymatrix@C=5ex@R=10ex{\Lambda^n(X)= L^n_0\ar[r]\ar@<+25pt>[d]_{\Phi^n(X)}& L^n_1\ar[r]\ar[d] &L^n_2\ar[r]\ar[d] &\cdots&  \cdots\ar[r]&\ar[r]\ar[d] L^n_{n-1}\ar[r]\ar[d] & L^n_n=\Lambda^n(Y)\ar@<-25pt>[d]^{\Phi^n(Y)}\\
\Lambda'^n(X)= {L'}^n_0\ar[r]& {L'}^n_1\ar[r]\ar[r]&{L'}^n_2\ar[r] &\cdots&\cdots \ar[r]&\ar[r] {L'}^n_{n-1}\ar[r] &{ L'}^n_n=\Lambda'^n(Y)
}$$

For every simplicial sheaf $\bcX$, we want to construct a natural morphism $\vartheta^n_{\bcX}$ from $\Sym^n(\bcX)$ to $\Sym_g^n(\bcX)\,.$
First of all, let us consider the case when $\bcX$ is a representable simplicial sheaf $h_X$ for $X$ in $\bcC$.  
In this case, $\Sym_g^n(h_X)$ is nothing but $h_{\Sym^nX}$. In view  of the isomorphism $(h_X)^{\times n}\isom h_{X^n}$, the canonical morphism 
$h_{X^n}\into h_{\Sym^n}$ induces a morphism $(h_{X})^{\times n}/\Sigma_n\into h_{\Sym^n}$, that is, a morphism $\Sym^n(h_X)\into\Sym_g^n(h_X)$. We denote this morphism by $\vartheta^n_{h_X}$ or simply by  $\vartheta^n_{X}$. 
    
\begin{proposition}\label{papogxm}
For every simplicial sheaf $\bcX$, there is a functorial morphism $$\vartheta^n_{\bcX}:\Sym^n(\bcX)\into\Sym_g^n(\bcX)\,.$$ 
\end{proposition}
\begin{proof}
It is enough to show for a sheaf $\bcX$. Indeed, in view of Corollary \ref{papcorgg18}, we have an isomorphism $\bcX^{\times n}\isom \colim_{h_X\into \bcX}h_{X^n}$. Hence, one has
\begin{align*}
\Sym^n(\bcX)&=(\bcX^{\times n})/\Sigma_n\\&\isom (\colim_{h_X\into \bcX}h_{X^n})/\Sigma_n\\&\isom \colim_{h_X\into \bcX}\left(h_{X^n}/\Sigma_n\right)\\&= \colim_{h_X\into \bcX}\Sym^n(h_{X})\,.
\end{align*}
Taking colimit to the canonical morphisms $\vartheta^n_X: \Sym^nh_X\into\Sym_g^nh_X$, for $X$ in $\bcC$,  we get a morphism 
$$\colim_{h_X\into \bcX}\vartheta^n_X: \colim_{h_X\into \bcX}\Sym^nh_X\into\colim_{h_X\into \bcX}\Sym_g^nh_X\,.$$ 
On the one hand, we have seen above that $ \colim_{h_X\into \bcX}\Sym^nh_X$ is isomorphic to $\Sym^n(\bcX)$, and on the other hand, $\colim_{h_X\into \bcX}\Sym_g^nh_X$ is by definition equal to $\Sym_g^n\bcX$. Thus, we get a functorial morphism  from $\Sym^n(\bcX)$ to $\Sym_g^n(\bcX)$ which we denote by $\vartheta^n_{\bcX}$. 
\end{proof}

\begin{corollary}\label{cpapogxm}
For every pointed simplicial sheaf $\bcX$, there is a functorial morphism $$\vartheta^n_{\bcX}:\Sym^n(\bcX)\into\Sym_g^n(\bcX)\,.$$ 
\end{corollary}
\begin{proof}
It follows from the previous Proposition \ref{papogxm}. 
\end{proof}

For each $n\in\NN$,  we denote by $\vartheta^n:\Sym^n\into\Sym^n_g$ the natural transformation defined for every pointed simplicial sheaf $\bcX$ to be the functorial morphism $\vartheta^n(\bcX):=\vartheta^n_{\bcX}$ of Corollary \ref{cpapogxm}.   

\medskip

\begin{lemma}\label{paplzva11}
Let $\varphi:X\into Y$ be termwise coprojection in $\Deltaop\bcC_+$ and let us write $f:=h^+(\varphi)$. Then, for every pair of numbers $(n,i)\in\NN^2$ with $0\leq i\leq n$, there exists a canonical morphism 
$$\vartheta^n_i(f):L^n_i(f)\into\bcL^n_i(f)\,,$$
 such that one has a commutative diagram 
 \begin{equation}\label{papess1y}
      \xymatrix@C=4ex@R=10ex{ L^n_0(f)\ar[r]\ar[d]_{\vartheta^n_0(f)}& L^n_1(f)\ar[r]\ar[d]^{\vartheta^n_1(f)} &\cdots&  \cdots\ar[r]&\ar[r]\ar[d] L^n_{n-1}(f)\ar[r]\ar[d]^{\vartheta^n_{n-1}(f)} & L^n_n(f)\ar[d]^{\vartheta^n_n(f)}\\
\bcL^n_0(f)\ar[r]& \bcL^n_1(f)\ar[r]\ar[r]&\cdots&\cdots \ar[r]&\ar[r] \bcL^n_{n-1}(f)\ar[r] &\bcL^n_n(f)
}
   \end{equation}
\end{lemma}
\begin{proof}
Let us fix a natural number $n$. For each index $0\leq i\leq n$, $\bcL^n_i(f)$ is nothing but the object $h^+(\tilde{\Box}^n_i(\varphi))$, see Proposition \ref{exvvg4}. Since the functor $h^+:\bcC_+\into\bcS_{\ast}$ is monoidal, $\Box^n_i(f)$ is canonically isomorphic to $h^+(\Box^n_i(\varphi))$. Thus, we have a canonical morphism $\Box^n_i(f)\into \bcL^n_i(f)$, and this morphism induces a morphism $\vartheta^n_i(f):L^n_i(f)\into\bcL^n_i(f)$. Since $\vartheta^n_i(f)$ is constructed canonically, we get a commutative diagram \eqref{papess1y}. 
\end{proof}

\begin{example}
\em Let us consider a coprojection $X\into X\vee Y$ in $\Deltaop\bcC_+$ and let $f$ be the morphism $h^+(\varphi)$. We have a commutative diagram 
\begin{equation}\label{papeqgd72}
\xymatrix{h_{X}\wedge h_{X}\ar[r]\ar[d]& (h_X\vee h_Y)\wedge h_X\ar[d]\\
h_X \wedge  (h_X\vee h_Y)\ar[r]&(h_X\vee h_Y)\wedge ( h_X\vee h_Y)
  }
\end{equation}
which is induced by a diagram 
\[ \xymatrix{X\wedge X\ar[r]\ar[d]& (X\vee Y)\wedge X\ar[d]\\
X \wedge  (X\vee Y)\ar[r]&(X\vee Y)\wedge(X\vee Y)
  }
\]
Then, one gets canonical morphisms 
\begin{align*}
&\vartheta^2_0(f):L^2_0(f)\longrightarrow\bcL^2_0(f)\,,\\&
\vartheta^2_1(f):L^2_1(f)\longrightarrow \bcL^2_1(f)\,,\\&
\vartheta^2_2(f):L^2_2(f)\longrightarrow \bcL^2_2(f)\,.
\end{align*}
where their domains have the form
\begin{align*}
\hspace{1cm}&\Box^2_0(f)= h_{X\wedge X}\,,\\&
\Box^2_1(f)=h_{X \wedge  (X\vee Y)}\wedge_{h_{X\wedge X}}h_{ (X\vee Y)\wedge X}\,, \\
&\Box^2_2(f)=h_{X\wedge Y}\wedge h_{X\wedge Y}
\end{align*}
and their codomains have the shape
\begin{align*}
\hspace{3cm}& \bcL^2_0(f)=h_{\Sym^2X}\,,\\&
 \bcL^2_1(f)=h_{\Sym^2X}\wedge (h_{\Sym^1X}\vee h_{\Sym^1Y}) \,, \\
& \bcL^2_2(f)=h_{\Sym^2X}\wedge (h_{\Sym^1X}\vee h_{\Sym^1Y})\wedge h_{\Sym^2Y}\,.
\end{align*}
\end{example}

\begin{proposition}\label{papffj331}
Let $f:\bcX\into\bcY$ be a morphism of pointed simplicial sheaves in $I_{\proj}$-$cell$ such that $\bcX$ is an $I_{\proj}$-cell complex. Then, for every index $0\leq i\leq n$, there exists a canonical morphism 
$$\vartheta^n_i(f):L^n_i(f)\into\bcL^n_i(f)\,,$$
 such that one has a commutative diagram 
   \begin{equation}\label{papess2y}
   \xymatrix@C=4ex@R=10ex{ L^n_0(f)\ar[r]\ar[d]_{\vartheta^n_0(f)}& L^n_1(f)\ar[r]\ar[d]^{\vartheta^n_1(f)} &\cdots&  \cdots\ar[r]&\ar[r]\ar[d] L^n_{n-1}(f)\ar[r]\ar[d]^{\vartheta^n_{n-1}(f)} & L^n_n(f)\ar[d]^{\vartheta^n_n(f)}\\
\bcL^n_0(f)\ar[r]& \bcL^n_1(f)\ar[r]\ar[r]&\cdots&\cdots \ar[r]&\ar[r] \bcL^n_{n-1}(f)\ar[r] &\bcL^n_n(f)
}
 \end{equation}
where $\vartheta^n_0(f)=\vartheta^n_{\bcX}$ and $\vartheta^n_n(f)=\vartheta^n_{\bcY}$. 
\end{proposition}
\begin{proof}
As in Proposition \ref{gsjpr5}, the morphism $f$ can be expressed as the colimit of a directed diagram $\{f_d\}_{d\in D}$ of morphisms of representable simplicial sheaves.
Now, by Lemma \ref{paplzva11} we have canonical morphisms $\vartheta^n_i(f_{d}):L^n_i(f_{d})\into\bcL^n_i(f_{d})$. Hence, taking colimit we get a morphism 
$$\colim_{d\in D}\vartheta^n_i(f_{d}):\colim_{d\in D}L^n_i(f_{d})\into\colim_{d\in D}\bcL^n_i(f_{d})\,,$$ 
 Since this colimit is filtered,  $\colim_{d\in D}L^n_i(f_{d})$ is isomorphic to $L^n_i(f)$. Thus, the above morphism is isomorphic to a morphism from $L^n_i(f)$ to $\bcL^n_i(f)$, denoted by $\vartheta^n_i(f)$. Finally, the diagrams of the form \eqref{papess1y} induce a commutative diagram \eqref{papess2y}.    
\end{proof}
By virtue of Proposition \ref{papogxm}, for each $n\in\NN$, we get a natural transformation $$\vartheta^n: \Sym^n\into\Sym^n_g$$ as functors on the category $(\Deltaop\bcC)^{\#}_+$.

\begin{theorem}\label{nattrath}
The natural transformations $\vartheta^n: \Sym^n\into\Sym^n_g$ on $(\Deltaop\bcC)^{\#}_+$, for $n\in\NN$, induce a morphism of $\lambda$-structures from the left derived categoric symmetric powers to the left derived geometric powers on $\bcH_{\ast}(\bcC_{\Nis},\AF^1)$.  
\end{theorem}
\begin{proof}
The natural transformations $\vartheta^n: \Sym^n\into\Sym^n_g$ for $n\in\NN$ induce a natural transformation of derived functors $L\vartheta^n: L\Sym^n\into L\Sym^n_g$ on $\bcH_{\ast}(\bcC_{\Nis},\AF^1)$. Hence we apply Proposition 
\ref{papffj331} to get morphisms of K\"unneth towers. 
\end{proof}

%

\section{Generalities of symmetric spectra}
\label{Prelims}

In this section $\bcC$ will denote a small admissible category contained in the category of quasi-projective schemes over a field $k$ of arbitrary characteristic. The letter $\bcS$ to denote the category of simplicial Nisnevish sheaves and the category $\Deltaop\bcS_{\ast}$ is the category of pointed simplicial sheaves studied in the previous sections. We write $S^1$ for pointed simplicial circle, i.e. the cokernel of the morphism $\partial\Delta[1]_+\into \Delta[1]_+$ in $\SSet_{\ast}$. We shall denote by $T$ the smash product $S^1\wedge(\GG_m,1)$. There is an isomorphism $T\isom (\PR^1, \infty)$ in $\bcH_{\ast}(\bcC_{\Nis},\AF^1)$, cf. \cite[Lemma~3.2.15]{MoVoe99}.

 \medskip

{\em\noindent Symmetric spectra.}--- Let $\Sigma$ be the category, whose objects are natural numbers, and for any two natural numbers $m$ and $n$, $\Hom_{\Sigma}(m,n)$ is the group of permutations $\Sigma_n$ if $m=n$, and it is $0$ otherwise.
If $\bcX=(\bcX_0,\bcX_1,\dots )$ and $\bcY=(\bcY_0,\bcY_1,\dots )$ are two symmetric sequences on $\Delta^{\op}\bcS_{\ast}$, then there is a product $\bcX\wedge\bcY$ given by the formula
 $$(\bcX\wedge\bcY)_n=\bigvee_{i+j=n}\cor^{\Sigma_n}_{\Sigma_i\wedge\Sigma_j}(\bcX_i\wedge\bcY_j)\,.$$
 A symmetric $T$-spectrum on $\Deltaop\bcS_{\ast}$  is a sequence of $\Sigma_n$-objects $\bcX_n$ in $\Deltaop\bcS_{\ast}$ together with $\Sigma_n$-equivariant morphisms $\bcX_n\wedge T\into \bcX_{n+1}$ for $n\in\NN$, such that the composite
 $$\bcX_m\wedge T^{\wedge n}\into \bcX_{m+1}\wedge T^{\wedge (n-1)}\into\cdots \into \bcX_{m+n}$$ 
 is $\Sigma_{m+n}$-equivariant for all couples $(m,n)\in\NN^2$.    
 We denote by $\Sp_T(k)$ the category of symmetric $T$-spectra on the category $\Deltaop\bcS_{\ast}$. The category $\Sp_T(k)$ is naturally equivalent to the category of left modules over the commutative monoid $\sym(T):=(\Spec(k)_+,T,T^{\wedge 2},T^{\wedge 3},\dots)$, see \cite[Prop.~2.2.1]{HSBS-01}. For each $n\in\NN$, there is an evaluation functor $\Ev_n$ from $\Sp_T(k)$ to $\Deltaop\bcS_{\ast}$ which takes a symmetric $T$-spectrum $\bcX$ to its $n$th slice $\bcX_n$. The evaluation functor $\Ev_n$ has a left adjoint functor denoted by $F_n$. The functor $F_0$ is called {\em suspension functor}, and it is usually denoted by $\Sigma^{\infty}_T$. This functor takes simplicial sheaf $\bcX$ to the symmetric $T$-spectrum $(\bcX,\bcX\wedge T, \bcX\wedge T^{\wedge 2},\dots)$. 
For a scheme $X$ in $\bcC$, we write $\Sigma^{\infty}_T(X_+)$ instead of $\Sigma^{\infty}_T(\Delta_{X}[0]_+)$. A morphism of $T$-spectra $f:\bcX\into\bcY$ is a level $\AF^1$-weak equivalence (a level fibration) if each term $f_n$ is an $\AF^1$-weak equivalence (a fibration) in $\Deltaop\bcS_{\ast}$ for all $n\in\NN$. We say that $f$ is a projective cofibration if it has the left lifting property with respect to both level $\AF^1$-equivalences and level fibrations.   
 The class of level $\AF^1$-weak equivalences, the class of the level fibrations and the class of projective cofibrations define a left proper cellular model structure on $\Sp_T(k)$ called projective model structure, see \cite{Hov-01}. Let $I$ (resp. $J$) be the set of generating (resp. trivial) cofibrations of the injective model structure of $\Deltaop\bcS_{\ast}$.
The set $I_T:=\bigcup_{n\geq 0}F_n(I)$ (resp. $J_T:=\bigcup_{n\geq 0}F_n(J)$) is the set of generating cofibrations (resp. trivial cofibrations) of the projective model structure of $\Sp_T(k)$, cf. \cite{Hov-01}.  
 
 In order to define the stable model structure on $\Sp_T(k)$, one uses the Bousfield localization of its projective model structure  with respect to a certain set of morphisms of symmetric $T$-spectra, so that the functor $-\wedge T:\Sp_T(k)\into \Sp_T(k)$ becomes a Quillen equivalence. We shall define this set as follows. For every simplicial sheaf  $\bcX$ in $\Deltaop\bcS_{\ast}$ and every $n\in\NN$, we denote by
$\zeta^{\bcX}_n:F_{n+1}(\bcX\wedge T)\into F_n(\bcX)$ the morphism which is adjoint to the morphism
$$\bcX\wedge T\into\Ev_{n+1}(F_n(\bcX))=\Sigma_{n+1}\times_{\Sigma_1}(\bcX\wedge T)$$
induced by the canonical embedding of $\Sigma_1$ into $\Sigma_n$.
We set
$$S:=\{\zeta^{\bcX}_n\,|\,\bcX\in\dom(I)\cup\codom(I), n\in\NN\}$$
The {\em stable model structure} on $\Sp_T(k)$ is the Bousfield localization of the projective model structure on $\Sp_T(k)$ with respect to $S$, cf. \cite{Hov-01}. A $S$-local weak equivalence will be called a {\em stable equivalence}. The stable model structure on $\Sp_T(k)$ is left proper and cellular.
The functor $\Sigma^{\infty}_T:\Deltaop\bcS_{\ast}\into \Sp_T(k)$ is a left Quillen functor, see {\it loc.cit}. 
For any two symmetric $T$-spectra $\bcX$ and $\bcY$, its {\em smash} product $\bcX\wedge_{\sym(T)}\bcY$ is defined to be the coequalizer of the diagram 
$$\bcX\wedge\sym(T)\wedge\bcY\rightrightarrows \bcX\wedge\bcY$$
induced by the canonical morphisms $\bcX\wedge\sym(T)\into\bcX$ and $\sym(T)\wedge\bcY\into\bcY$. The smash product of spectra defines a symmetric monoidal structure on $\Sp_T(k)$.
We denote by $\SH_{T}(k)$ the homotopy category of the category $\Sp_T(k)$ with respect to stable $\AF^1$-weak equivalences.

\bigskip 
{\it \noindent Symmetric spectra of chain complexes.}--- Let $\Ab$ be the category of Abelian groups. The classical Dold-Kan correspondence establishes a Quillen equivalence
\[N:\Delta^{\op}\Ab\rightleftarrows\ch_{+}(\Ab):\Gamma\]
 between the category of simplicial Abelian groups and the category of $\NN$-graded chain complexes of Abelian groups. 
Let $\bcA$ be an Abelian Grothendieck category. We write $\ch_{+}(\bcA)$ for the category of $\NN$-graded chain complexes on $\bcA$. 
The above adjunction induces an adjunction 
\begin{equation}\label{equdkj}
N:\Delta^{\op}\bcA\rightleftarrows\ch_{+}(\bcA):\Gamma
\end{equation}
The category $\ch_{+}(\bcA)$ has a monoidal proper closed simplicial model category such that the class of weak equivalences are quasi-isomorphisms and such that the adjunction \ref{equdkj} becomes a Quillen equivalence \cite[Lemma~2.5]{Jard03}.  For any $n\in\ZZ$, we have the translation functor $\ch_{+}(\bcA)\into \ch_{+}(\bcA)$ which sends a chain complex $C$ to $C[n]$ defined by $(C[n])_i:= C_{n+i}$ for $i\geq 0$. For each $n\geq0$, we denote by $\ZZ[n]$ the chain complex
\[\cdots\into 0\into \ZZ\into 0 \into  \cdots\into 0\]
concentrated in degree $n$. If the symbol $\tensor$ denotes the tensor product of $\NN$-graded chain complexes of Abelian groups, then, for $n\in\NN$, we have $\ZZ[n]=\ZZ[1]^{\tensor n}$.  Hence, the symmetric group $\Sigma_n$ acts naturally on $\ZZ[n]$, and we have the symmetric sequence
\[\sym(\ZZ[1])=\left(\ZZ[0],\ZZ[1],\ZZ[2],\cdots\right)\]
in $\ch_{+}(\Ab)$. For any chain complex $C_*$ in $\ch_-(\bcA)$, we have
\[C_*\tensor\ZZ[n]=C_*[-n]\,.\]

Let $\Sp_{\ZZ[1]}(\ch_{+}(\mathA))$ be the category of symmetric $\ZZ[1]$-spectra. Its objects are symmetric sequences $(C_0,C_1,\dots,C_n\dots)$
where each $C_n$ is a chain complex in $\ch_{+}(\mathA)$ together with an action of the symmetric group $\Sigma_n$ on it.
For a symmetric $\ZZ[1]$-spectrum $C_*$, we have structural morphisms of the form
 $C_n\tensor \ZZ[1]\into C_{n+1}$ for $n\in\NN$.
 
\bigskip 
{\it \noindent Rational stable homotopy category of schemes.}--- In the next paragraphs, we shall recall some results on rational stable homotopy categories of schemes over a field. 
Here, $\SH_T(k)$ will be the stable $\AF^1$-homotopy category of smooth schemes over a field $k$ constructed in \cite{Jar00}.  
 One result that is very important is a theorem due to Morel which asserts an equivalence of categories between the rational stable homotopy category $\SH_T(k)_{\QQ}$ and the rational big Voevodsky's category $\DM(k)_{\QQ}$. This will allows us to show the existence of transfers of some morphisms in $\SH_T(k)_{\QQ}$ that will be studied in Section \ref{Transfers} and \ref{Mainsec}. 
   
\medskip

 Let $\bcT$ be a triangulated  category with small sums and with a small set of compact generators \cite{Nee01}. An object $T$ in $\bcT$ is said to be {\em torsion} (resp. {\em uniquely divisible}) if for every compact generator $X$ in $\bcT$, the canonical morphism from $\Hom_{\bcT}(X,T)$ to $\Hom_{\bcT}(X,T)\tensor_{\ZZ}\QQ$ is the zero morphism (resp. an isomorphism).   Let $\bcT_{\tor}$ (resp. $\bcT_{\QQ}$) be the triangulated subcategory of $\bcT$ generated  by the torsion objects (resp.  uniquely divisible objects). 
The full embedding functor $\bcT_{\QQ}\incl \bcT$ has a left adjoint $L_{\QQ}:\bcT\into \bcT_{\QQ}$ and its kernel is nothing but $\bcT_{\tor}$. Then, $\bcT_{\QQ}$ is equivalent to the Verdier quotient $\bcT/\bcT_{\tor}$ (see \cite[Annexe~A]{Riou06}). We denote by $\SH_T(k)_{\QQ}$ the Verdier quotient of $\SH_T(k)$ by the full-subcategory $\SH_T(k)_{\tor}$ generated by compact torsion objects. We recall that a morphism of symmetric $T$-spectra $f:\bcX\into \bcY$ is a stable $\AF^1$-weak equivalence if and only if the induced morphism
$$f_*:\Hom_{\SH_T(k)}(\Sigma^{\infty}_T(S^r\wedge \GG_m^s\wedge U_+), \bcX)\into\Hom_{\SH_T(k)}(\Sigma^{\infty}_T(S^r\wedge \GG_m^s\wedge U_+), \bcY)$$
is an isomorphism of Abelian groups for all couples $(r,s)\in\NN^2$ and all smooth schemes $U$ over $k$ (see \cite[Th.~1.2.10(iv)]{Hovey0} and \cite[Cor. 3.3.8]{Pel11}). A morphism of $T$-spectra $f:\bcX\into \bcY$ is called  
{\em  rational stable $\AF^1$-weak equivalence} if the induced morphism $f_*\tensor \QQ$
is an isomorphism of $\QQ$-vector spaces for all couples $(r,s)\in\NN^2$ and all smooth schemes $U$ over $k$. The localization of $\SH_T(k)$ with respect to the rational stable $\AF^1$-weak equivalences coincides with $\SH_T(k)_{\QQ}$. 

\bigskip 

{\it \noindent Motivic categories.}--- Let $\Ab^{\tr}_{\Nis}$ be the category of Nisnevich Abelian sheaves with transfers on the category of smooth schemes $\Sm/k$ over a field $k$ c.f. \cite{MVW06,Deg07}. 
Let $\tau$ be either the $h$-topology or $\qfh$-topology on the category of $k$-schemes of finite type. We write $\underline{\Ab}^{\tr}_{\tau}$ for the category of $\tau$-Abelian sheaves with transfers on the category of $k$-schemes of finite type. 
We consider the $\AF^1$-localized model category of the projective model structure on $\ch_+(\Ab^{\tr}_{\Nis})$ and in $\ch_+(\underline{\Ab}^{\tr}_{\tau})$. 
Let $\DM(k)$ be the homotopy category of the category of symmetric $T$-spectra  $\Sp_{T}(\ch_+(\Ab^{\tr}_{\Nis}))$ with respect to stable $\AF^1$-weak equivalences.  If the characteristic of $k$ is zero, then $\DM(k)$ is equivalent to the homotopy category of the category of modules over the motivic spectrum \cite{RO06}. 
We denote by $\underline{\DM}(k)_{\tau}$ the homotopy category of the category of symmetric $T$-spectra  $\Sp_{T}(\ch_+(\underline{\Ab}^{\tr}_{\tau}))$ with respect to stable $\AF^1$-weak equivalences. We write $\DM_{\tau}(k)$ for the localizing subcategory of  $\underline{\DM}_{\tau}(k)$ generated by the objects of the form $\Sigma^{\infty}_T\ZZ_{\tau}(X)(m)[n]$ for $k$-smooth schemes of finite type $X$ and for all couples $(m,n)\in\ZZ$, see \cite{C-D13}. 
 One has an adjunction of triangulated categories 
$$\Hu: \SH_T(k)\rightleftarrows \DM(k): H$$ 
where $\Hu$ is the Hurewicz functor and $H$ is the Eilenberg McLane functor \cite{Mor04,C-D13}. 
This adjunction induces an adjunction of triangulated categories with rational coefficients
$$\Hu_{\QQ}: \SH_T(k)_{\QQ}\rightleftarrows \DM(k)_{\QQ}: H_{\QQ}$$ 
We write $\mathbb{S}^0$ for the sphere $T$-spectrum. Let $\epsilon:\mathbb{S}^0\into\mathbb{S}^0$ be the morphism of spectra induced by the morphism $\mathbb{G}_m\into \mathbb{G}_m$ which comes from  the homomorphism of $k$-algebras $k[x,x^{-1}]\into k[x,x^{-1}]$ given by $x\mapsto x^{-1}$. 
Notice that $\epsilon$ is idempotent i.e. $\epsilon^2=\epsilon$.
 Since $\SH_T(k)_{\QQ}$ has small coproducts (see \cite{Nee01}), the triangulated category $\SH_T(k)_{\QQ}$ is pseudo-abelian, then the morphism $\epsilon$ splits. Hence, $\mathbb{S}^0_{\QQ}$ splits into a direct sum $\mathbb{S}^0_{\QQ}=\mathbb{S}^0_{\QQ,+}\oplus \mathbb{S}^0_{\QQ,-}$ in $\SH_T(k)_{\QQ}$. This decomposition induces a decomposition $$\SH_T(k)_{\QQ}=\SH_T(k)_{\QQ,+}\times \SH_T(k)_{\QQ,-}\,.$$

\begin{theorem}[Morel]\label{thhf74}
Suppose that $-1$ is a sum of squares in $k$. Then we have an equivalence of categories 
$\SH_T(k)_{\QQ}\isom \DM(k)_{\QQ}$.
\end{theorem}
\begin{proof}

The fact that $-1$ is a sum of squares in $k$ implies that the category $\SH_T(k)_{\QQ,+}$ coincides with $\SH_T(k)_{\QQ}$. Hence, the theorem follows from Theorem 16.1.4 and Theorem 16.2.13  in \cite{C-D13}. 
\end{proof}

Let $D_{\AF^1}(k)$ be the homotopy category of the category of symmetric $T$-spectra  $\Sp_{T}(\ch_+(\Ab_{\Nis}))$ with respect to stable $\AF^1$-weak equivalences. The category of Beilinson motives $\DM_{\Bei}(k)$ is the Verdier quotient of $D_{\AF^1}(k)_{\QQ}$ by the localizing subcategory generated by $H_{\Bei}$-acyclic objects, where $H_{\Bei}$ is the Beilison motivic spectrum, see \cite{C-D13, Riou06}.  If $-1$ is a sum of squares in $k$, then we have a diagram of equivalences of categories:
$$\xymatrix{\SH_T(k)_{\QQ}\ar@{=}[r]&D_{\AF^1}(k)_{\QQ}\ar@{=}[r]&\DM_{\Bei}(k)\ar@{=}[r]\ar@{=}[d]&\DM(k)_{\QQ}\\
&&\DM_{h}(k)_{\QQ}\ar@{=}[d]&\\&&\DM_{\qfh}(k)_{\QQ}&}$$
For the proof of these equivalences see \cite{C-D13,Mor04}. As a consequence, we obtain the following corollary.   

\begin{corollary}\label{tfddj}
 If $-1$ is a sum of squares in $k$, then we have an equivalence of categories $\SH_T(k)_{\QQ}\isom \DM_{\qfh}(k)_{\QQ}$. 
\end{corollary}
\begin{proof}
It is a consequence of the preceding considerations, see \cite{C-D13}.   
\end{proof}
\medskip

\section{Geometric symmetric powers of motivic spectra}
\label{geomsym}

In this section, we shall define geometric symmetric powers $\Sym^n_{g,T}$ on the category of symmetric $T$-spectra of motivic spaces on admissible categories of schemes over a field. We show that they extend, in a natural way, the geometric symmetric powers $\Sym^n_{g}$  of motivic spaces  discussed in Section \ref{geomsymu}, see Proposition \ref{pevvxs}.  
\bigskip

Let $\bcC\subset \Sch/k$ be an admissible category and let $\Sp_T(k)$ be the category of $T$-spectra on the category $\Deltaop\bcS_{\ast}$. We denote by $\Sym^n_T$ the categoric symmetric power in $\Sp_T(k)$, that is, for a symmetric $T$-spectrum $\bcX$, $\Sym^n_T(\bcX)$ is the quotient of the $n$th smash fold product $\bcX^{\wedge n}$ by the symmetric group $\Sigma_n$. 

\begin{lemma}\label{javvc5}
We have  a commutative diagram
$$\xymatrix@C=10ex@R=10ex{\Deltaop\bcS_{\ast}\ar[r]^{\Sym^n}\ar[d]_{\Sigma^{\infty}_T}&\Deltaop\bcS_{\ast}\ar[d]^{\Sigma^{\infty}_T}\\
\Sp_T(k)\ar[r]_{\Sym_{T}^n}&\Sp_T(k)}$$
\end{lemma}
\begin{proof}
Let $\bcX$ be a pointed simplicial sheaf in $\Deltaop\bcS_{\ast}$. By \cite[Th.~6.3]{Hovey0}, the functor $\Sigma^{\infty}_T: \Deltaop\bcS_{\ast}\into \Sp_T(k)$ is a monoidal Quillen functor. Hence, for $n\in\NN$, the suspension $\Sigma^{\infty}_T(\bcX^{\wedge n})$ is isomorphic to the product $\Sigma^{\infty}_T(\bcX)^{\wedge n}$. Since $\Sigma^{\infty}_T$ is a left adjoint functor, we have 
\begin{align*}
\Sigma^{\infty}_T(\Sym^n\bcX)&= \Sigma^{\infty}_T(\bcX^{\wedge n}/\Sigma_n)\\&
\isom  \Sigma^{\infty}_T(\bcX^{\wedge n})/\Sigma_n\\&
\isom  \Sigma^{\infty}_T(\bcX)^{\wedge n}/\Sigma_n\\&
\isom\Sym^n_T(\Sigma^{\infty}_T\bcX)\,.
\end{align*} 
This proves the lemma. 
\end{proof}

\begin{corollary}\label{sscor4}
For every simplicial sheaf $\bcX$ in $\Deltaop\bcS_{\ast}$, we have an isomorphism 
$$(\Ev_0\circ\Sym^n_{T}\circ \Sigma^{\infty}_T)(\bcX)\isom\Sym^n(\bcX)\,.$$
\end{corollary}
\begin{proof}
It follows from the previous lemma in view that $\Ev_0(\Sigma^{\infty}_T\bcY)=\bcY$ for a pointed simplicial sheaf $\bcY$. 
\end{proof}
\medskip 

The category $\Delta^{\op}\bcC$ is symmetric monoidal. For two simplicial objects $X$ and $Y$, the product $X\times Y$ is the simplicial object such that each term $(X\times Y)_n$ is given by the product $X_n\times Y_n$.
If $X=(X_0,X_1,X_2,\dots )$ and $Y=(Y_0,Y_1,Y_2,\dots )$ are two symmetric sequences on $\Delta^{\op}\bcC$, then we have a product $X\tensor Y$ which given by the formula
 $$(X\tensor Y)_n=\coprod_{i+j=n}\cor^{\Sigma_n}_{\Sigma_i\times\Sigma_j}(X_i\times Y_j)\,.$$
 For every  symmetric sequence $X=(X_0,X_1, X_2,\dots)$ on the category $\Deltaop\bcC$ and for every $n\in\NN$, there exists the quotient $X^{\tensor n}/ \Sigma_n$ in the category of symmetric sequences on $\Deltaop\bcC$. 
For every $p\in\NN$, we have 
$$(X^{\tensor n})_p=\coprod_{i_1+\cdots +i_n=p}\cor^{\Sigma_p}_{\Sigma_{i_1}\times\cdots \times\Sigma_{i_n}}(X_{i_1}\times\cdots\times X_{i_n})\,$$
and the symmetric group $\Sigma_n$ acts on $(X^{\tensor n})_p$ by permutation of factors. As $\bcC$ allows quotients by finite groups, the quotient $(X^{\tensor n})_p/\Sigma_n$ is an object of $\bcC$ for all $p\in\NN$. Notice that the $0$th slice of $X^{\tensor n}/\Sigma_n$ is nothing but the usual $n$th symmetric power $\Sym^n(X_0)=X_0^{\tensor n}/\Sigma_n$ in $\bcC$.

 Let us fix  an object $S$ of $\Deltaop\bcC$. A  symmetric $S$-spectrum on $\Deltaop\bcC$  is a sequence of $\Sigma_n$-objects $X_n$ in $\Deltaop\bcC$ together with $\Sigma_n$-equivariant morphisms $X_n\times S\into X_{n+1}$ for $n\in\NN$, such that the composite
 $$X_m\times S^{\times n}\into X_{m+1}\times S^{\times (n-1)}\into\cdots \into X_{m+n}$$ is $\Sigma_{m+n}$-equivariant for all couples $(m,n)\in\NN^2$.  We denote by $\Sp_S(\Delta^{\op}\bcC)$ the category of symmetric $S$-spectra on $\Deltaop\bcC$. 
We have a suspension functor $F_0$ from $\Delta^{\op}\bcC$ to $\Sp_S(\Delta^{\op}\bcC)$ that takes an object $X$ of $\Delta^{\op}\bcC$ to the symmetric $S$-spectrum of the form $(X, X\times S, X\times S^{\times 2}, \dots)$. 

Let $T$ be the pointed simplicial sheaf $(\PR^1,\infty)$ and let $T'$ be the pointed simplicial sheaf $\PR^1_+$ in $\Deltaop\bcS_{\ast}$.
 The canonical functor $h^+:\Deltaop\bcC\into\Deltaop\bcS_{\ast}$ induces a functor $$H':\Sp_{\PR^1}(\Delta^{\op}\bcC)\into \Sp_{T'}(k)$$
  which takes a symmetric $\PR^1$-spectrum $(X_0,X_1,\dots)$ to the symmetric $T'$-spectrum $(h^+(X_0),h^+(X_1),\dots)$. Since $\bcC$ is a small category, the category $\Deltaop\bcC$ is also small. Hence, the category $\Sp_{\PR^1}(\Delta^{\op}\bcC)$ is so.

Let $f:T'\into T$ be the canonical morphism of simplicial sheaves.  This morphism induces a morphism of commutative monoids $\sym(T')\into \sym(T)$.  In particular, $\sym(T)$ can be seen as symmetric $T'$-spectrum. For any two symmetric $T'$-spectra $\bcX$ and $\bcY$, we write $\bcX\wedge_{\sym(T')}\bcY$ for the coequalizer of the diagram 
$$\xymatrix{\bcX\wedge\sym(T')\wedge\bcY\ar@<0.7ex>[r]\ar@<-0.7ex>[r]&\bcX\wedge\bcY}
$$
induced by the canonical morphisms $\bcX\wedge\sym(T')\into\bcX$ and $\sym(T')\wedge\bcY\into\bcY$. For every symmetric $T'$-spectrum $\bcX$, the symmetric sequence $\bcX\wedge_{\sym(T')}\sym(T)$ is a symmetric $T$-spectrum. We have a functor 
$$(-)\wedge_{\sym(T')}\sym(T): \Sp_{T'}(k)\longrightarrow \Sp_{T}(k)\,.$$
which is left adjoint. Its right adjoint functor is the restriction functor $\res_{T/T'}$ that sends a symmetric $T$-spectrum  $\bcX$ to $\bcX$ itself thought as a symmetric $T'$-spectrum via the morphism $f:T'\into T$.  Let 
$$H:\Sp_{\PR^1}(\Delta^{\op}\bcC)\into \Sp_{T}(k)$$
be the composition of $H'$ with the functor $(-)\wedge_{\sym(T')}\sym(T)$. 

Let $U$ be an object in $\Sp_{\PR^1}(\Delta^{\op}\bcC)$ and let $n$ be a positive integer. The canonical morphisms $U\tensor\sym(\PR^1)\into U$ and $U\tensor\sym(\PR^1)\into U$ induce a diagram of the form
$$
\xymatrix@C=2.4ex{U\tensor \sym(\PR^1)\tensor U\tensor \cdots \tensor \sym(\PR^1)\tensor U\ar@<1.8ex>[rr]\ar@<-1.8ex>[rr]\ar@<2.5ex>[rr]\ar@<-2.5ex>[rr]\ar@<1.1ex>[rr]\ar@<-1.1ex>[rr]&\cdots\cdots&U^{\tensor n}}\,.
$$
in which $U$ appears $n$ times in the product of the left-hand-side. This diagram can be seen as a functor from a category, with two objects $a,b$ and with $n$ non-trivial arrows $a\into b$, to the category of symmetric sequences $(\Delta^{\op}\bcC)^{\Sigma}$. For instance, when $n=2$, this diagram is nothing but a coequalizer diagram.

The symmetric group acts on this product by permuting of factors of $U$. Hence, we obtain a diagram 
\begin{equation}\label{equu32}
\xymatrix@C=2.4ex{H\Big(\big(U\tensor \sym(\PR^1)\tensor U\tensor \cdots \tensor \sym(\PR^1)\tensor U\big)/\Sigma_n\Big)\ar@<1.8ex>[rr]\ar@<-1.8ex>[rr]\ar@<2.5ex>[rr]\ar@<-2.5ex>[rr]\ar@<1.1ex>[rr]\ar@<-1.1ex>[rr]&\cdots\cdots&H(U^{\tensor n}/\Sigma_n)}\,.
\end{equation}

\medskip
{\it \noindent Geometric symmetric powers.}--- 
For each symmetric $T$-spectrum $\bcX$, we denote by $(H\downarrow \bcX)$ the comma category whose objects are arrows of the form $H(U)\into \bcX$ for all $U$ in $\Sp_{\PR^1}(\Delta^{\op}\bcC)$. Let $F_{\bcX}:(H\downarrow \bcX)\into\Sp_T(k)$ be the functor which sends a morphism $H(U)\into \bcX$ to the colimit of the diagram \eqref{equu32}. 
We define $\Sym_{g,T}^n(\bcX)$ to be the colimit of the functor $F_{\bcX}$.   
 The functor  $\Sym_{g,T}^n$ is called $n$th-fold {\em geometric symmetric power} of symmetric $T$-spectra.

\medskip
All we have constructed in the previous paragraphs are summarized in the following diagram: 

\begin{equation}
\vcenter{
  \xymatrix@C=5ex@R=5ex{ \bcC\ar[rr]^{\Sym^n}
  \ar[dd]_{\Const}& &\bcC\ar[dd]^{\Const}
   & & \\
  & & & & \\ 
  \Deltaop\bcC\ar[dd]_{F_0}
  \ar[rr]^-{\Sym^n_{\PR^1}}
  \ar[dr]^-{h^+} & &
   \Deltaop\bcC \ar[dd]_(0.25){F_0}|[d]\hole
  \ar[dr]^-{h^+} & & \\
  & \Deltaop\bcS_{\ast}\ar[dd]_(0.35){\Sigma^{\infty}_T}
  \ar[rr]_(0.3){\Sym^n_{g}}
  & &\Deltaop\bcS_{\ast}\ar[dd]^-{\Sigma^{\infty}_T} & \\
  \Sp_{\PR^1}(\Delta^{\op}\bcC)\ar[rr]
  \ar[dr]_-{H} & &\Sp_{\PR^1}(\Delta^{\op}\bcC)
  \ar[dr]^-{H} & & \\
  & \Sp_T(k)\ar[rr]_-{\Sym^n_{g,T}}
  & & \Sp_T(k) &}
  }
\end{equation}

\begin{lemma}\label{yhd192}
Let $U$ be an object in $\Sp_{\PR^1}(\Delta^{\op}\bcC)$ and let $n$ be a positive integer. We have a canonical morphism $\vartheta^n_{H(U)}:\Sym^n_TH(U)\into \Sym^n_{g,T}H(U)$. 
\end{lemma}
\begin{proof}
The diagram \eqref{equu32} yields into a commutative diagram
\begin{equation}\label{eggd91}
\xymatrix@C=2ex@R=7ex{\bigslant{\Big(H(U)\wedge \sym(T)\wedge H(U)\wedge \cdots \wedge \sym(T)\wedge H(U)\Big)}{\Sigma_n}\ar@<1.8ex>[rr]\ar@<-1.8ex>[rr]\ar@<2.5ex>[rr]\ar@<-2.5ex>[rr]\ar@<1.1ex>[rr]\ar@<-1.1ex>[rr]\ar[d]&\cdots\cdots&H(U)^{\wedge n}/{\Sigma_n}\ar[d]\\
H\Big(\big(U\tensor \sym(\PR^1)\tensor U\tensor \cdots \tensor \sym(\PR^1)\tensor U\big)/\Sigma_n\Big)\ar@<1.8ex>[rr]\ar@<-1.8ex>[rr]\ar@<2.5ex>[rr]\ar@<-2.5ex>[rr]\ar@<1.1ex>[rr]\ar@<-1.1ex>[rr]&\cdots\cdots&H(U^{\tensor n}/\Sigma_n)}
\end{equation} 
where the vertical morphisms are the canonical morphisms. After taking colimit on the above  diagram, we obtain a morphism from $\Sym^n_TH(U)$ to $\Sym^n_{g,T}H(U)$.
\end{proof}

 We recall that $h^+$ denotes the canonical functor from $\Deltaop\bcC$ to $\Deltaop\bcS_{\ast}$. 

\begin{lemma}\label{lem721}
Let $\bcX=(\bcX_0,\bcX_1,\dots)$ be a symmetric $T$-spectrum in $\Sp_{T}(k)$. Then, the functor $\Ev_n:(H\downarrow\bcX)\into (h^+\downarrow\bcX_n)$  is final. 
\end{lemma}
\begin{proof}
Suppose that it is given a morphism $h^+(U)\into \bcX_n$, where $U$ is an object of $\Deltaop\bcC$. By adjunction, this morphism corresponds to a morphism of $T$-spectra $F_n(h^+(U))\into\bcX$. Since $F_n\circ h^+=H\circ F_n$, we have a morphism $H(F_n(U))\into\bcX$. The unit morphism $h^+(U)\into (\Ev_n\circ F_n)(h^+(U))$ gives a commutative diagram 
   $$\xymatrix{h^+(U)\ar[dr]\ar[rr]&& (\Ev_n\circ F_n)(h^+(U))\ar[dl]\\
   &\bcX_n&}$$
   where $ (\Ev_n\circ F_n)(h^+(U))=h^+(\cor^{\Sigma_n}_{\Sigma_0}(U))$.
Now, suppose that there are two morphisms $H(X)\into \bcX$ and $H(X')\into \bcX$,  where $X$ and $X'$ are in $\Sp_{\PR^1}(\Delta^{\op}\bcC)$. Then, we have a commutative diagram 
$$\xymatrix{&\bcX\ar@/^1pc/[rdd]^{}\ar@/_1pc/[ldd]_{}\ar@{.>}[d]^{}&\\
& H(X)\vee H(X')\ar@/^0.5pc/[rd]\ar@/_0.5pc/[ld]&\\
 H(X)&& H(X')}$$
where the dotted arrow exists by the universal property of coproduct. As $H(X\amalg X')$ is isomorphic to $H(X)\vee H(X')$, the above diagram induces a commutative diagram 
$$\xymatrix{&\bcX_n\ar@/^1pc/[rdd]^{}\ar@/_1pc/[ldd]_{}\ar@{.>}[d]^{}&\\
& h^+(X_n\amalg X'_n)\ar@/^0.5pc/[rd]\ar@/_0.5pc/[ld]&\\
h^+(X_n)&& h^+(X'_n)}$$
This proves that the required functor is final, see \cite[page~213]{McLa71}. 
\end{proof}

\begin{proposition}\label{pevvxs}
Let $n$ be a natural number. For every symmetric $T$-spectrum $\bcX$ in $\Sp_T(k)$, we have a canonical isomorphism 
$$\Ev_0\circ\Sym^n_{g,T}(\bcX)\isom\Sym^n_{g}\circ \Ev_0(\bcX)\,.$$
\end{proposition}
\begin{proof}
Let $U$ be an object in $\Sp_{\PR^1}(\Delta^{\op}\bcC)$. Since $\Ev_0(\sym(\PR^1))=\Spec(k)$. Applying  functor $\Ev_0$ to the diagram \eqref{equu32}, we obtain diagram consisting of identity morphisms $h^+(U^n/\Sigma)\into h^+(U^n/\Sigma)$. Hence, the colimit of this diagram is $h^+(U^n/\Sigma)$. Thus, we have 
$$\Ev_0\circ\Sym^n_{g,T}(\bcX)=\colim_{H(U)\into \bcX}h^+(U_0^n/\Sigma_n)\,.$$
By Lemma \ref{lem721}, the right-hand-side is isomorphic to $\colim_{h^+(U_0)\into\Ev_0(\bcX)}h^+(U_0^n/\Sigma_n)$, and by Lemma \ref{lerffs}, the latter is isomorphic to $\Sym^n_{g}\circ \Ev_0(\bcX)$. 
\end{proof}

\begin{corollary}
For every simplicial sheaf $\bcX$ in $\Deltaop\bcS_{\ast}$, there is a canonical isomorphism 
$$\Ev_0(\Sym^n_{g,T}(\Sigma^{\infty}_T\bcX))\isom\Sym^n_{g}(\bcX)\,.$$
\end{corollary}
\begin{proof}
It follows from the preceding proposition in view that $\Ev_0(\Sigma^{\infty}_T\bcX)$ is equal to $\bcX$.  
\end{proof}

\begin{proposition}
For every $k$-scheme $X$ in $\bcC$, one has an isomorphism 
$$\Sym^n_{g,T}\Sigma^{\infty}_TX_+\isom\Sigma^{\infty}_T(\Sym^nX)_+\,.$$
\end{proposition}
\begin{proof}
We have $$\Sigma^{\infty}_TX_+=\Sigma^{\infty}_T(h^+(X))=H(F_0(X))\,.$$
Hence, $\Sym^n_{g,T}\Sigma^{\infty}_TX_+=\Sym^n_{g,T}H(F_0(X))$. By definition, $\Sym^n_{g,T}H(F_0(X))$ is the coequalizer of the diagram \eqref{equu32} for $U=F_0(\Const(X))$. Let us write $F_0(X)$ for $F_0(\Const(X))$. We have
 $$H\big(F_0(X)^{\tensor n}/\Sigma_n\big)=H\big(F_0(X^n/\Sigma)\big)=\Sigma^{\infty}_T(h^+(\Sym^nX))=\Sigma^{\infty}_T(\Sym^nX)_+\,.$$ 
Since $\sym(\PR^1)=F_0(\Spec(k))$, the object on left-hand-side of diagram \eqref{equu32} is nothing but $H\big(F_0(X)^{\tensor n}/\Sigma_n\big)$ and the arrows are the identities. Therefore, the colimit of this diagram is $H\big(F_0(X)^{\tensor n}/\Sigma_n\big)=\Sigma^{\infty}_T(\Sym^nX)_+$.  
\end{proof}

For a symmetric $T$-spectrum $\bcX$, we shall write $\vartheta^n(\bcX)$ for $\vartheta^n_{\bcX}$.
 
 \begin{corollary}\label{gstabp31}
Let $\bcX$ be a pointed simplicial sheaf in $\Deltaop\bcS_{\ast}$. If the natural morphism $\vartheta^n_T(\Sigma^{\infty}_T\bcX):\Sym^{n}_{T}(\Sigma^{\infty}_{T}\bcX)\into\Sym^{n}_{g,T}(\Sigma^{\infty}_T\bcX)$ is a stable $\AF^1$-weak equivalence, then the natural morphism $\Sym^{n}(\bcX)\into\Sym^{n}_g(\bcX)$ is an $\AF^1$-weak equivalence. 
\end{corollary}
\begin{proof}
In virtue of Corollary \ref{sscor4} and Proposition \ref{pevvxs}, we have a commutative diagram 
\begin{equation}\label{eqst72}
\xymatrix@C=20ex@R=10ex{\Sym^{n}(\bcX)\ar[r]^{\vartheta^n(\bcX)}\ar[d]&\Sym^{n}_g(\bcX)\ar[d]\\
\Ev_0(\Sym^{n}_{T}(\Sigma^{\infty}_{T}\bcX))\ar[r]_{\Ev_0(\vartheta^n_T(\Sigma^{\infty}_T\bcX))}&\Ev_0(\Sym^{n}_{g,T}(\Sigma^{\infty}_T\bcX))}
\end{equation}
 where the vertical  morphisms are isomorphisms.
Since $\vartheta^n_T(\Sigma^{\infty}_T\bcX)$ is a stable $\AF^1$-weak equivalence, the morphism $\Ev_0(\vartheta^n_T(\Sigma^{\infty}_T\bcX))$ is an $\AF^1$-weak equivalence. Therefore, $\vartheta^n_{\bcX}$ is an $\AF^n$-weak equivalence.
\end{proof}

 \medskip
Let $n\in\NN$. For a symmetric sequence $\bcX=(\bcX_0,\bcX_1,\dots)$, we define
$\Sym_{\ell,T'}^n(\bcX)$ to be the symmetric sequence $\Big(\Sym_{g}^n(\bcX_0),\Sym_{g}^n(\bcX_1),\dots\Big)\,,$ 
and called it the $n$th fold {\em level geometric symmetric powers} of $\bcX$. From the definition, we have $$\Ev_i(\Sym_{\ell,T}^n(\bcX))=\Sym_{g}^n(\Ev_i(\bcX))$$ for all $i\in\NN$.  
 
 \begin{lemma}
 For any symmetric $T'$-spectrum $\bcX=(\bcX_0,\bcX_1,\dots)$, the $n$th level geometric symmetric power of $\bcX$ is a symmetric $T'$-spectrum. 
 \end{lemma}
 \begin{proof}
Let us consider a symmetric $T$-spectrum $\bcX=(\bcX_0,\bcX_1,\bcX_2,\dots)$. For a $k$-scheme $U$ in $\bcC$, we define a morphism of $k$-schemes from $U^n\times\PR^1$ to $(U\times \PR^1)^n$ as the composite 
 $$\xymatrix{ U^n\times \PR^1\ar[rr]^{\id\times\Delta_{\PR^1}}&&U^n\times ({\PR^1})^{n}\ar[r]& (U\times \PR^1)^n} 
$$
 where $\Delta_{\PR^1}$ is the diagonal morphism and the second arrow is the canonical isomorphism. This morphism induces a morphism $\Sym^n(U)\times \PR^1$ to $\Sym^n(U\times\PR^1)$, which we denote by $\varphi_n$. Let us fix a natural number $i$. To construct a natural morphism from $\Sym_{g}^n(\bcX_i)\wedge\PR^1_+\into \Sym_{g}^n(\bcX_{i+1})$, it is enough to construct a morphism from $\Sym_{g}^n(\bcX_i)\times\PR^1\into \Sym_{g}^n(\bcX_{i+1})$ considered as unpointed sheaves. Any morphism $h_{U}\into \bcX_i$ induces a morphism $h_{U\times\PR^1}\into \bcX_i\times h_{\PR^1}$. Composing with the preceding morphism, we obtain a morphism $h_{U\times\PR^1}\into \bcX_{i+1}$. Hence, using the morphism $\varphi_n$, we obtain a morphism 
 $$\colim_{h_{U}\into \bcX_i} h_{\Sym^n(U)\times \PR^1}\longrightarrow\colim_{h_{V}\into \bcX_{i+1}} h_{\Sym^n(V)}\,.$$ This gives a morphism from $\Sym_{g}^n(\bcX_i)\times\PR^1$ to $ \Sym_{g}^n(\bcX_{i+1})$. Since this morphism was constructed in a natural way for all index $i$, we get structural morphisms for $\Sym_{\ell,T}^n(\bcX)$.  
 \end{proof}

 \begin{proposition}\label{pexsz}
For every $n\in\NN$, the functor $\Sym_{\ell,T'}^n$ preserves levelwise $\AF^1$-weak equivalences between symmetric $T'$-spectra whose slices are objects in $\Deltaop\bar{\bcC}_{+}$ (see Section \ref{Prelim}). 
\end{proposition}
\begin{proof}
Let $f$ be a morphism of symmetric $T'$-spectra. From the definition we have an equality $\Ev_i(\Sym_{\ell,T'}^n(f))=\Sym_{g}^n(\Ev_i(f))$ for every $i\in\NN$. Hence the proposition follows from Theorem \ref{radthind107}. 
\end{proof}

\begin{remark}
{\em  The left Kan extension of the composite 
$$\Deltaop\bcC\stackrel{\Sym^n}{\longrightarrow} \Deltaop\bcC\stackrel{\Sigma^{\infty}_T}{\longrightarrow}\Sp_T(k)\,,$$
 along the suspension functor $\Sigma^{\infty}_T$, is not a good candidate to be a ``geometric" symmetric power on $\Sp_T(k)$, as this Kan extension is not isomorphic to the identity functor of $\Sp_T(k)$ when $n=1$.
 }
\end{remark}

\begin{remark}\label{Remm74}
{\em For a symmetric $T$-spectrum $\bcX$, the canonical morphism $\vartheta^n_{\bcX}$ from the categoric symmetric power $\Sym^{n}_T(\bcX)$ to geometric symmetric power $\Sym^{n}_g(\bcX)$ is not always a stable $\AF^1$-weak equivalence. For instance, if $\bcX$ is a constant simplicial sheaf whose value is the sheaf represented by the affine space $\AF^2$, then Proposition \ref{papxz301} and Corollary \ref{gstabp31} imply that the canonical morphism from $\Sym^{n}_{T}(\Sigma^{\infty}_{T}\AF^2_+)$ to $\Sym^{n}_{g,T}(\Sigma^{\infty}_T\AF^2_+)$ is not a stable $\AF^1$-weak equivalence.  
}
\end{remark}

\section{Lambda structures and motivic spectra}
\label{lammdctt}

The first result in this section is Theorem\ref{radmainths}, which asserts that, under the assumption that $\Sym^n_{g,T}$  preserves stable $\AF^1$-weak equivalences, the $n$th fold geometric symmetric powers $\Sym^n_{g,T}$, for $n\in\NN$, induce a $\lambda$-structure on the stable $\AF^1$-homotopy category. The second result is Theorem \ref{nattraths}, which establishes, under the same assumption, a morphism of $\lambda$-structures from categoric symmetric powers to geometric symmetric powers on the stable $\AF^1$-homotopy category.  
 
  \medskip

{\em\noindent K\"unneth towers.}--- 
Let $f:\bcX\into\bcY$ be a morphism of symmetric spectra in $\Sp_T(k)$. A filtration of $\Sym^n_{g,T}(f)$ of the form 
$$\Sym^n_{g,T}(\bcX)=\bcL^n_0(f)\into \bcL^n_1(f)\into\cdots\into \bcL^n_n(f)=\Sym^n_{g,T}(\bcY)$$
is called ({\em geometric}) {\em K\"unneth tower} of $\Sym^n_{g,T}(f)$, if for each index $1\leq i\leq n$, there is an isomorphism
$$\cone\Big(\bcL^n_{i-1}(f)\into \bcL^n_{i}(f)\Big)\isom \Sym^{n-i}_{g,T}(\bcX)\wedge \Sym^{i}_{g,T}(\bcX)$$ 
in $\SH_T(k)$.

A symmetric $T'$-spectrum (resp. $T$-spectrum) is called {\em representable}, if it is isomorphic to a $T'$-spectrum (resp. $T$-spectrum) of the form $H'(U)$ (resp. $H(U)$), where $U$ is an object on $\Sp_{\PR^1}(\Deltaop\bcC)$. Denote by $\Sp_{T'}(\Deltaop\bcC)^{\#}$ the full subcategory of $\Sp_{T'}(k)$ generated by directed colimits of representable $T'$-spectra. 
Similarly, we write $\Sp_{T}(\Deltaop\bcC)^{\#}$ for the full subcategory of $\Sp_{T}(k)$ generated by directed colimits of representable $T$-spectra.

Let us consider the category of symmetric sequences $(\Deltaop\bcC)^{\Sigma}$ and its coequalizer completion $(\Deltaop\bcC)^{\Sigma}_{\coeq}$, see page \pageref{lpagrttf}. We recall that the objects of $(\Deltaop\bcC)^{\Sigma}_{\coeq}$ are reflexive pairs in  $(\Deltaop\bcC)^{\Sigma}$. Denote by $\Phi$ the canonical functor from $(\Deltaop\bcC)^{\Sigma}$ to $(\Deltaop\bcC)^{\Sigma}_{\coeq}$.
The symmetric monoidal product on $(\Deltaop\bcC)^{\Sigma}$ extends to a symmetric  monoidal product on 
$(\Deltaop\bcC)^{\Sigma}_{\coeq}$, so that the universal functor $\Phi$ is monoidal. Since $\sym(\PR^1)$ is a commutative monoid on $(\Deltaop\bcC)^{\Sigma}$, the object $\Phi(\sym(\PR^1))$ is a commutative monoid in $(\Deltaop\bcC)^{\Sigma}_{\coeq}$. 
We denote by $\Sp_{\PR^1}(\Deltaop\bcC)_{\coeq}$ the category of left modules in $(\Deltaop\bcC)^{\Sigma}$ over $\Phi(\sym(\PR^1))$. The functor $\Phi$
induces a restriction functor $\Sp_{\PR^1}(\Deltaop\bcC)\into\Sp_{\PR^1}(\Deltaop\bcC)_{\coeq}$, which will be denoted by the same letter  $\Phi$, if no confusion arises.
For any two left modules $X$ and $Y$ over $\Phi(\sym(\PR^1))$, we define the product $X\wedge Y$ as the coequalizer of the canonical diagram 
$$\xymatrix{X\tensor \Phi(\sym(\PR^1))\tensor Y\ar@<0.7ex>[r]^-{}\ar@<-0.7ex>[r]_-{}& X\tensor Y}\,.$$
This product is symmetric monoidal on $\Sp_{\PR^1}(\Deltaop\bcC)_{\coeq}$.  
Thus, the restriction functor $\Phi:\Sp_{\PR^1}(\Deltaop\bcC)\into\Sp_{\PR^1}(\Deltaop\bcC)_{\coeq}$ is monoidal.

\begin{remark}
{\em If $\bcC$ has reflexive coequalizers, then   
$(\Deltaop\bcC)^{\Sigma}_{\coeq}$ is isomorphic to the category $(\Deltaop\bcC)^{\Sigma}$, and $\Sp_{\PR^1}(\Deltaop\bcC)_{\coeq}$ is isomorphic to $\Sp_{\PR^1}(\Deltaop\bcC)$. 
}
\end{remark}

By the universal property of $(\Deltaop\bcC)^{\Sigma}_{\coeq}$, the following commutative square
$$\xymatrix@C=20ex@R=10ex{(\Deltaop\bcC)^{\Sigma}\ar[r]^{(h^+)^{\Sigma}}&(\Delta^{\op}\bcS_{\ast})^{\Sigma}\\
\Sp_{\PR^1}(\Deltaop\bcC)\ar@{^(->}[u]\ar[r]_-{H'}&\Sp_{T'}(k)\ar@{^(->}[u]}$$
where $(h^+)^{\Sigma}$ is the canonical functor induced by $h^+$,
gives a commutative diagram 
$$\xymatrix@C=20ex@R=5ex{(\Deltaop\bcC)^{\Sigma}\ar[rr]^{}\ar[rd]_{}&&(\Delta^{\op}\bcS_{\ast})^{\Sigma}\ar[dd]\\
&(\Deltaop\bcC)^{\Sigma}_{\coeq}\ar[ru]&\\
\Sp_{\PR^1}(\Deltaop\bcC)\ar[rr]_(.4){H'}|(.51)[r]\hole\ar@{^(->}[uu]\ar[rd]_{\Phi}&&\Sp_{T'}(k)\ar@{^(->}[uu]
\\&\Sp_{\PR^1}(\Deltaop\bcC)_{\coeq}\ar@{^(->}[uu]\ar[ru]_{\bar{H'}}&}$$

\begin{lemma}\label{lemdkaa}
For any object $U$ in $\Sp_{\PR^1}(\Deltaop\bcC)$, we have an isomorphism 
$$\Sym^n_{g,T'}H'(U)\isom \bar{H'}\big(\Sym^n_{\PR^1}\Phi(U)\big)\,.$$
\end{lemma}
\begin{proof}
It follows from the definitions.
\end{proof}

A morphism $\varphi:X\into Y$ in $\Sp_{\PR^1}(\Delta^{\op}\bcC)$ is called {\em level-termwise coprojection}, if each level $\varphi_n$ is a termwise coprojection. 

\begin{proposition}\label{profllqo}
For every $n\in\NN$, the $n$th fold geometric symmetric symmetric power of a morphism of representable $T'$-spectra (resp. $T$-spectra), induced by a level-termwise coprojection in $\Sp_{\PR^1}(\Delta^{\op}\bcC)$,  has a canonical K\"unneth tower.
\end{proposition}
\begin{proof}
Let $\varphi:U\into V$ be a level-termwise coprojection in $\Sp_{\PR^1}(\Delta^{\op}\bcC)$. By \cite{GoGu}, we obtain a sequence of level-termwise coprojections
$$\Sym^n_{\PR^1}(\Phi(U))=\tilde{\Box}^n_0(\Phi(\varphi))\into \tilde{\Box}^n_1(\Phi(\varphi))\into\cdots\into \tilde{\Box}^n_n(\Phi(\varphi))=\Sym^n_{\PR^1}(\Phi(V))\,.$$ 
such that $\tilde{\Box}^n_i(\Phi(\varphi))/\tilde{\Box}^n_{i-1}(\Phi(\varphi))$ is isomorphic to $ \Sym^{n-i}_{\PR^1}(\Phi(U))\wedge \Sym^i_{\PR^1}(\Phi(V)/\Phi(U))$. In fact, this sequence exists because $\Sp_{\PR^1}(\Delta^{\op}\bcC)_{\coeq}$ has finite colimits.
Then, by Lemma \ref{lemdkaa}, the above sequence induces a sequence    
$$\bar{H'}\Big(\tilde{\Box}^n_0(\Phi(\varphi))\Big)\into \bar{H'}\Big(\tilde{\Box}^n_1(\Phi(\varphi))\Big)\into\cdots\into \bar{H'}\Big(\tilde{\Box}^n_n(\Phi(\varphi))\Big)\,.$$ 
Since $\bar{H'}$ respects pushouts and monoidal product, the above sequence is a K\"unneth tower of the morphism $\Sym^n_{g,T'}(H'(\varphi))$. Finally, since $H(\varphi)=H'(\varphi)\wedge_{\sym(T')}\sym(T)$, the K\"unneth tower of $\Sym^n_{g,T'}(H'(\varphi))$ induces a K\"unneth tower of $\Sym^n_{g,T}(H(\varphi))$.
\end{proof}

\medskip

We set $I_{T,\proj}:=\bigcup_{n\geq 0}F_n(I_{\proj})$, where $I_{\proj}$ is the set of   morphisms defined in \eqref{lagssqq}, see page \pageref{lagssqq}.
Our next goal is to study K\"unneth towers associated to relative $I_{T, \proj}$-cell complexes, see Proposition \ref{gsjps5s}.

\begin{lemma}\label{tspecgqs}
One has the following assertions:
\begin{enumerate}
\item[ ($a$)] A morphism of representable $T'$-spectra is isomorphic to the image of a morphism of $\PR^1$-spectra through the functor $H'$. 
\item[($b$)] Let 
\begin{equation}
\xymatrix@C=10ex@R=10ex{\bcA\ar[d]\ar[r]&\bcX\ar[d]\\
\bcB\ar[r]&\bcY
}
\end{equation}
be a cocartesian square of $T'$-spectra, such that the morphism $\bcA\into\bcB$ is the image of a level-termwise coprojection in $\Sp_{\PR^1}(\Deltaop\bcC)$ through the functor $H'$. Then, if $\bcX$ is a representable $T'$-spectra, then so is $\bcY$. 
\item[ ($c$)] Suppose that $\bcA$ and $\bcB$ are compact objects. If $\bcX$ is in  $\Sp_{T'}(\Deltaop\bcC)^{\#}$, then so is $\bcY$. Moreover, if $\bcX$ is a directed colimit of representable $T'$-spectra that are compact, then so is $\bcY$.
\end{enumerate}
\end{lemma}
\begin{proof}
($a$). It is a termwise verification.

($b$). Let us write $\bcA=H'(A)$, $\bcB=H'(B)$ and $\bcX=H'(X)$, where $A$, $B$ and $X$ are objects of $\Sp_{\PR^1}(\Deltaop\bcC)$.  Suppose that $\bcA\into\bcB$ is a morphism of the form $H'(\varphi)$, where $\varphi:A\into B$ is a level-termwise coprojection in $\Sp_{\PR^1}(\Deltaop\bcC)$. By item ($a$), the morphism $\bcA\into \bcX$ is is canonically isomorphic to a morphism of the form $H'(\psi)$, where $\psi:A\into X$ is a morphism in $\Sp_{\PR^1}(\Deltaop\bcC)$. 
Since $\varphi$ is a level-termwise coprojection, there exists an object $Y$ in $\Sp_{\PR^1}(\Deltaop\bcC)$ such that there is a cocartesian square
$$\xymatrix@C=10ex@R=10ex{A\ar[d]_{\varphi}\ar[r]^{\psi}&X\ar[d]\\
B\ar[r]&Y
}$$
Hence, $\bcY$ is isomorphic to $H'(Y)$. This proves ($b$). 

($c$). It is immediate from item ($b$) and the fact that finite colimits of compact objects are compact.
\end{proof}

\begin{lemma}\label{lemff73s}
Every $I_{T, \proj}$-cell complex of $\Sp_{T}(k)$ is in $\Sp_{T}(\Deltaop\bcC)^{\#}$. 
\end{lemma}
\begin{proof}
We reduce the problem to show that every $I_{T', \proj}$-cell complex of $\Sp_{T'}(k)$ is in $\Sp_{T'}(\Deltaop\bcC)^{\#}$. 
Since an element of $I_{T',\proj}$-$\cell$ is a transfinite composition of pushouts of element of $I_{T',\proj}$, this follows by transfinite induction in view of Lemma \ref{tspecgqs} ($c$). 
\end{proof}

\begin{proposition}\label{gsjps5s}
Let $f:\bcX\into\bcY$ be a morphism in $I_{T,\proj}$-$\cell$, where $\bcX$ is an $I_{T,\proj}$-cell complex. Then, for each $n\in\NN$, $\Sym^n_{g,T}(f)$ has a functorial K\"unneth tower.
\end{proposition}
\begin{proof}
By Lemma \ref{lemff73s} and Proposition 6.1.13 of \cite{KS06}, we deduce that the morphism $f$ can be expressed as the colimit of a directed diagram $\{f_d\}_{d\in D}$ of morphisms of representable $T$-spectra. Hence, by Proposition \ref{profllqo}, the $n$th fold geometric symmetric power  $\Sym^n_{g,T}(f_{d})$ has a canonical K\"unneth tower 
\begin{equation}\label{eqpp4231s}\xymatrix{\bcL^n_0(f_{d})\ar[r]&\bcL^n_1(f_{d})\ar[r]&\cdots \ar[r]&\bcL^n_n(f_{d})}\,.
\end{equation}
For each index $0\leq i\leq n$, we define 
$$\bcL^n_i(f):=\colim_{d\in D}\bcL^n_i(f_{d})\,.$$ 
 Then, we get a sequence 
\begin{equation}\label{eqpp4241s}
\bcL^n_0(f)\longrightarrow \bcL^n_1(f)\longrightarrow\cdots\longrightarrow \bcL^n_n(f)\,. 
\end{equation}
which is a K\"unneth tower of $\Sym^n_{g,T}(f)$. 
\end{proof}

\begin{corollary}\label{radcorsmlm}
There exists a functorial factorization $(\alpha,\beta)$ on $\Sp_T(k)$ such that for every morphism $f$ is factored as $f=\beta(f)\circ\alpha(f)$, where $\alpha(f)$ is in $I_{T,\proj}$-$\cell$ and $\beta(f)$ is in $I_{T,\proj}$-$\inj$.
\end{corollary}
\begin{proof}
It follows since the set $I_{T,\proj}$ permits the small object argument, see \cite[Prop. A.8]{Hov-01}.  
\end{proof}

\begin{proposition}\label{radthcofibre8}
Every cofibre sequence in $\SH_T(k)$ is isomorphic to a coprojection sequence of the form $$\bcA\into \bcB\into \bcB/\bcA\,,$$ where $\bcA\into \bcB$ is in $I_{T,\proj}$-$\cell$ and $\bcA$ is an $I_{T,\proj}$-cell complex.  
\end{proposition}
\begin{proof}
 Let $\bcX\into\bcY \into \bcZ$ be a cofibre sequence in $\SH_T(k)$, where $f$ is a projective cofibration from $\bcX$ to $\bcY$ in $\Sp_T(k)$, such that $\bcZ=\bcY/\bcX$. By Corollary \ref{radcorsmlm}, the morphism $\emptyset\into \bcX$ factors into $\emptyset\into \bcA\into\bcX$.  Again, by Corollary \ref{radcorsmlm}, the composition of $\bcA\into \bcX$ with $f$ induces a commutative diagram 
$$\xymatrix@C=10ex@R=10ex{\bcA\ar[r]^{\alpha(f)}\ar[d]&\bcB\ar[d]^{\beta(f)}\\\bcX\ar[r]_f&\bcY}$$ 
where $\beta(f)$ is a sectionwise trivial fibration and $\alpha(f)$ is in $I_{T,\proj}$-$\cell$. By \cite[Prop. 6.2.5]{Hovey0}, the cofibre sequence $\bcA\stackrel{}{\rightarrow}\bcB\rightarrow \bcB/\bcA$
  is isomorphic to the cofibre sequence $\bcX\stackrel{[f]}{\into}\bcY \into \bcZ$ in $\SH_T(k)$. 

\end{proof}

\begin{lemma}\label{gfta12}
For any $T$-spectrum $\bcX$, there is an isomorphism $$\colim_{H(U)\into\bcX}H(U)\isom \bcX\,.$$ 
\end{lemma}
\begin{proof}
Notice that for a symmetric $\PR^1$-spectrum, we have that $\Ev_n(H(U))$ coincides with $h^+(U_n)$. By virtue of Lemma \ref{lem721}, we get canonical isomorphisms $$\Ev_n\Big(\colim_{H(U)\into\bcX}H(U)\Big)= \colim_{H(U)\into\bcX}h^+(U_n)\isom \colim_{h^+(V)\into\bcX_n}h^+(V)=\bcX_n\,,$$
which us to deduce the expected isomorphism. 
\end{proof}

\begin{corollary}\label{tffiios}
For any $T$-spectrum $\bcX$, there is an isomorphism $\Sym^1_{g,T}(\bcX)\isom\bcX$. 
\end{corollary}
\begin{proof}
For $n=1$, the equalizer of diagram \eqref{equu32} is $H(U)$. Hence, we are in the case of Lemma \ref{gfta12}.     
\end{proof}

\medskip

\begin{theorem}\label{radmainths}
 Suppose that, for every $n\in \NN$, the left derived functor $L\Sym^n_{T,g}$  exists on $\SH_T(k)$. Then, the endofunctors $L\Sym^n_{T,g}$, for $n\in \NN$, 
 provides a $\lambda$-structure on $\SH_T(k)$.    
\end{theorem}
\begin{proof}
We have evidently that $L\Sym^0_g$ is the constant functor with value $\one$.  By Corollary \ref{tffiios}, $L\Sym^1_g$ is the identity functor on $\SH_T(k)$. Let $\bcX\into\bcY\into \bcZ$ be a cofibre sequence in $\SH_T(k)$ induced by a cofibration $f:\bcX\into\bcY$ in $\Sp_T(k)$. 
By Proposition \ref{radthcofibre8}, we can assume that $f$ is in $I_{T,\proj}$-$\cell$ and $\bcX$ is an $I_{T,\proj}$-cell complex. 
Hence, by Proposition \ref{gsjpr5}, for each index $n\in\NN$, $\Sym^n_{g,T}(f)$ has a K\"unneth tower,
\begin{equation}\label{edd23s}
\Sym^n_{g,T}(\bcX)=\bcL^n_0(f)\into\bcL^n_1(f)\into\cdots\into \bcL^n_n(f)=\Sym^n_{g,T}(\bcY)\,,
\end{equation}
which induces a K\"unneth tower, 
$$L\Sym^n_{g,T}(\bcX)=L\bcL^n_0(f)\into L\bcL^n_1(f)\into\cdots\into L\bcL^n_n(f)=L\Sym^n_{g,T}(\bcY)\,,$$ of $L\Sym^n_{g,T}(f)$. The functoriality axiom follows from the functionality of K\"unneth towers of the form \eqref{edd23s}..
\end{proof}

\bigskip

{\em\noindent Morphisms of Lambda structures.}---  For a symmetric $T$-spectra $\bcX$, we shall construct a natural morphism $\vartheta^n_{\bcX}$ from $\Sym^n_T(\bcX)$ to $\Sym_{g,T}^n(\bcX)\,.$

\begin{proposition}\label{yhd1925}
Let $\bcX$ be an object in $\Sp_{T}(k)$ and let $n\in\NN$. Then, we have a canonical morphism $\vartheta^n_{\bcX}:\Sym^n_T(\bcX)\into \Sym^n_{g,T}(\bcX)$. 
\end{proposition}
\begin{proof}
We define $\vartheta^n_{\bcX}$ to be the colimit of the morphisms $\vartheta^n_{H(U)}$ of Lemma \ref{yhd192}, where $H(U)\into \bcX$ runs on the objects of  the comma category $(H\downarrow\bcX)$. By definition $\Sym^n_{g,T}\bcX=\colim_{H(U)\into \bcX}\Sym^n_{g,T}H(U)$. It remains to show that there is a canonical isomorphism $\Sym^n_{T}\bcX=\colim_{H(U)\into \bcX}\Sym^n_{T}H(U)$. Notice the Cartesian product of $\Deltaop\bcC$ induces a Cartesian product on category $(H\downarrow\bcX)$. By Lemma \ref{paplmxq8} and Lemma \ref{gfta12}, we deduce an isomorphism $\bcX^{\wedge n}\isom \colim_{H(U)\into \bcX}H(U)^{\wedge n}$. By the same argument, we deduce that the product $\bcX\wedge\sym(T)\wedge \bcX\wedge \cdots\wedge\sym(T)\wedge \bcX$, in which the object $\bcX$ appears $n$ times, is isomorphic to the colimit 
$$\colim_{H(U)\into\bcX}\Big(H(U)\wedge\sym(T)\wedge H(U)\wedge \cdots \wedge\sym(T)\wedge H(U)\Big)\,.$$
By change of colimits and by the above considerations, we deduce that the colimit of the diagram 
$$\xymatrix{\bigslant{\Big(\bcX\wedge \sym(T)\wedge \bcX\wedge \cdots \wedge \sym(T)\wedge \bcX\Big)}{\Sigma_n}\ar@<1.8ex>[rr]\ar@<-1.8ex>[rr]\ar@<2.5ex>[rr]\ar@<-2.5ex>[rr]\ar@<1.1ex>[rr]\ar@<-1.1ex>[rr]&\cdots\cdots&\bcX^{\wedge n}/{\Sigma_n}}$$
is a double colimit, that is, the colimit of the colimits of diagrams of the form 
$$\xymatrix{\bigslant{\Big(H(U)\wedge \sym(T)\wedge H(U)\wedge \cdots \wedge \sym(T)\wedge H(U)\Big)}{\Sigma_n}\ar@<1.8ex>[rr]\ar@<-1.8ex>[rr]\ar@<2.5ex>[rr]\ar@<-2.5ex>[rr]\ar@<1.1ex>[rr]\ar@<-1.1ex>[rr]&\cdots\cdots&H(U)^{\wedge n}/{\Sigma_n}}$$
where $H(U)\into \bcX$ runs on the objects of   $(H\downarrow\bcX)$. This implies that $\Sym^n_{T}\bcX$ is isomorphic to $\colim_{H(U)\into \bcX}\Sym^n_{T}H(U)$. 
\end{proof}

For each $n\in\NN$,  we denote by $\vartheta^n:\Sym^n_T\into\Sym^n_{g,T}$ the natural transformation defined for every pointed simplicial sheaf $\bcX$ to be the functorial morphism $\vartheta^n(\bcX):=\vartheta^n_{\bcX}$.   

\medskip

\begin{lemma}\label{paplzva11s}
Let $\varphi:X\into Y$ be level-termwise coprojection in $\Sp_{\PR^1}(\Deltaop\bcC)$ and let us write $f:=\Deltaop H(\varphi)$. Then, for every pair of numbers $(n,i)\in\NN^2$ with $0\leq i\leq n$, there exists a canonical morphism 
$$\vartheta^n_i(f):L^n_i(f)\into\bcL^n_i(f)\,,$$
 such that one has a commutative diagram 
 \begin{equation}\label{papess1ys}
      \xymatrix@C=4ex@R=10ex{ L^n_0(f)\ar[r]\ar[d]_{\vartheta^n_0(f)}& L^n_1(f)\ar[r]\ar[d]^{\vartheta^n_1(f)} &\cdots&  \cdots\ar[r]&\ar[r]\ar[d] L^n_{n-1}(f)\ar[r]\ar[d]^{\vartheta^n_{n-1}(f)} & L^n_n(f)\ar[d]^{\vartheta^n_n(f)}\\
\bcL^n_0(f)\ar[r]& \bcL^n_1(f)\ar[r]\ar[r]&\cdots&\cdots \ar[r]&\ar[r] \bcL^n_{n-1}(f)\ar[r] &\bcL^n_n(f)
}
   \end{equation}
\end{lemma}
\begin{proof}
Let us fix a natural number $n$. For each index $0\leq i\leq n$, $\bcL^n_i(f)$ is nothing but the object $H(\tilde{\Box}^n_i(\varphi))$, see Proposition \ref{profllqo}.  Since the functor $H$ is monoidal, $\Box^n_i(f)$ is canonically isomorphic to $H(\Box^n_i(\varphi))$. Thus, we have a canonical morphism $\Box^n_i(f)\into \bcL^n_i(f)$, and this morphism induces a morphism $\vartheta^n_i(f):L^n_i(f)\into\bcL^n_i(f)$. Since $\vartheta^n_i(f)$ is constructed canonically, we get a commutative diagram \eqref{papess1ys}. 
\end{proof}

\begin{proposition}\label{papffj331s}
Let $f:\bcX\into\bcY$ be a morphism of $T$-spectra in $I_{T,\proj}$ such that $\bcX$ is a $I_{T,\proj}$-cell complex. Then, for every index $0\leq i\leq n$, there exists a canonical morphism 
$$\vartheta^n_i(f):L^n_i(f)\into\bcL^n_i(f)\,,$$
 such that one has a commutative diagram 
   \begin{equation}\label{papess2ys}
   \xymatrix@C=4ex@R=10ex{ L^n_0(f)\ar[r]\ar[d]_{\vartheta^n_0(f)}& L^n_1(f)\ar[r]\ar[d]^{\vartheta^n_1(f)} &\cdots&  \cdots\ar[r]&\ar[r]\ar[d] L^n_{n-1}(f)\ar[r]\ar[d]^{\vartheta^n_{n-1}(f)} & L^n_n(f)\ar[d]^{\vartheta^n_n(f)}\\
\bcL^n_0(f)\ar[r]& \bcL^n_1(f)\ar[r]\ar[r]&\cdots&\cdots \ar[r]&\ar[r] \bcL^n_{n-1}(f)\ar[r] &\bcL^n_n(f)
}
 \end{equation}
where $\vartheta^n_0(f)=\vartheta^n_{\bcX}$ and $\vartheta^n_n(f)=\vartheta^n_{\bcY}$. 
\end{proposition}
\begin{proof}
As in Proposition \ref{gsjps5s}, the morphism $f$ can be expressed as the colimit of a directed diagram $\{f_d\}_{d\in D}$ of morphisms of representable $T$-spectra.
Now, by Lemma \ref{paplzva11s}, we have canonical morphisms $\vartheta^n_i(f_{d}):L^n_i(f_{d})\into\bcL^n_i(f_{d})$. Hence, taking colimit we get a morphism 
$$\colim_{d\in D}\vartheta^n_i(f_{d}):\colim_{d\in D}L^n_i(f_{d})\into\colim_{d\in D}\bcL^n_i(f_{d})\,,$$ 
This morphism gives a morphism from $L^n_i(f)$ to $\bcL^n_i(f)$, and we denote it by $\vartheta^n_i(f)$. Finally, the diagrams of the form \eqref{papess1ys} induce a commutative diagram \eqref{papess2ys}.    
\end{proof}

\begin{theorem}\label{nattraths}
Suppose that, for every $n\in\NN$, the left derived functor of $\Sym^n_{g,T}$ exists on $\SH_T(k)$. Then, the natural transformations $\vartheta^n: \Sym^n_T\into\Sym^n_{g,T}$ on $\Sp_{\PR^1}(\Deltaop\bcC)^{\#}$, for $n\in\NN$, induce a morphism of $\lambda$-structures from the left derived categoric symmetric powers to the left derived geometric powers on $\SH_T(k)$.  
\end{theorem}
\begin{proof}
The natural transformations $\vartheta^n: \Sym^n\into\Sym^n_g$ for $n\in\NN$ induce a natural transformation of derived functors $L\vartheta^n: L\Sym^n\into L\Sym^n_g$ on $\SH_T(k)$. Then, we apply Proposition 
\ref{papffj331s} to get morphisms of K\"unneth towers, as required. 
\end{proof}

\section{Formalism of transfers}
\label{Transfers}

The purpose of this section is to study the notion of transfer of morphisms, in a categorical context, involving transfers in stable $\AF^1$-homotopy category and in the category of Voedvosky's motives. 

\medskip

In the next paragraphs $(\bcD,\wedge)$ and $(\bcE,\tensor)$ will be two symmetric monoidal categories. Assume that $\bcE$ is also an additive category. Let $E$ be a monoidal functor from $(\bcD,\wedge)$ to $(\bcE,\tensor)$.
Assume $G$ a finite group. Let $X$ be an $G$-object in $\bcD$ and let $\rho_X:G\into\Aut(X)$ be a representation of $G$ on $X$. The functor $E$ induces an homomorphism of groups $\Aut(X)\into \Aut(E(X))$. Hence, the composition of this homomorphism with $\rho_X$ gives an homomorphism of groups $G\into \Aut(E(X))$, thus $G$ acts on $E(X)$. This homomorphism induces an homomorphism of Abelian groups $\ZZ[G]\into\End(E(X))$. The norm $\Nm(E(X))$ of $E(X)$ is the image of the element $\sum_{g\in G}g$ under this map. Explicitly, it is given by the formula
 \[\Nm^E(\pi)=\sum_{g\in G}E(\rho_X(g))\,.\] 
Now, suppose that the quotient $X/G$ exists in $\bcD$ and let $\pi:X\into X/G$ be the canonical morphism. The {\em transfer morphism}, or simply, the {\em transfer} of $E(\pi)$ is a morphism
\[\tr^E(\pi):E(X/G)\into E(X)\]
such that $E(\pi)\circ \tr^E(\pi)=n.\id_{E(X/G)}$ and $\tr^E(\pi)\circ E(\pi)=\Nm(E(X))$. 

\begin{example}
{\em Consider $(\bcD,\wedge)$ to be the category of quasi-projective schemes over a field $k$ together with the Cartesian product of schemes, and consider $(\bcE,\tensor)$ to be the category of $\qfh$-sheaves together with the Cartesian product of sheaves. For every $n\in\NN$ and for every quasi-projective $k$-scheme $X$, the canonical morphism $\ZZ_{\qfh}(X^n)\into\ZZ_{\qfh}(\Sym^nX)$ has transfer, see Proposition \ref{Voetranf}. 
}
\end{example}
The following example is a consequence of the previous one. 
\begin{example}
{\em Suppose that $(\bcD,\wedge)$ is the same category as in the previous example and $(\bcE,\tensor)$ is the category of $\qfh$-motives together with the monoidal product of motives \cite{Voe96}. Then, the canonical morphism of $\qfh$-motives $M_{\qfh}(X^n)\into  M_{\qfh}(\Sym^nX)$ has transfer. 
}
\end{example}

Let us study the case when $G$ is the symmetric group $\Sigma_n$ acting of the $n$th fold product $X^{\wedge n}$ of an object $X$ of $\bcD$. Since $E$ is monoidal we have an isomorphism $E(X^{\wedge n})\isom E(X)^{\tensor n}$. 
Let $\varrho: E(X^{\wedge n})\into  E(X)^{\tensor n}/\Sigma_n$ be the composite of the isomorphism $E(X^{\wedge n})\isom E(X)^{\tensor n}$ with the canonical morphism $E(X)^{\tensor n}\into E(X)^{\tensor n}/\Sigma_n$.  One has a commutative diagram
\begin{equation}\label{eqhhs38}
\xymatrix{E(X^{\wedge n})\ar[dd]_{\sigma}\ar[rd]^{\varrho}\ar@/^1pc/[rrd]^{E(\pi)}&&\\
&E(X)^{\tensor n}/\Sigma_n\ar@{.>}[r]^u&E(X^{\wedge n}/\Sigma_n)\\
E(X^{\wedge n})\ar[ru]_{\varrho}\ar@/_1pc/[rru]_{E(\pi)}&&
}
\end{equation}
where the dotted arrow exists by the universal property of colimit. Let us keep these considerations for the proof of Proposition \ref{gstab45}.

A $\QQ$-linear category is a category enriched over the category of $\QQ$-vector spaces. 

\begin{proposition}\label{gstab45}
Suppose $E:(\bcD,\wedge)\into (\bcE,\tensor)$ 
is a monoidal functor of monoidal symmetric categories. Assume that $\bcE$ is a $\QQ$-linear  and closed under finite colimits. Let $X$ be an object of $\bcD$  and let $\pi:X^{\wedge n}\into X^{\wedge n}/\Sigma_n$ be the canonical morphism. Suppose that  $E(\pi)$ is an epimorphism and has a transfer $\tr^{E}(\pi)$. Then, the universal morphism 
\[u:E(X)^{\tensor n}/\Sigma_n\into E(X^{\wedge n}/\Sigma_n) \]
is an isomorphism. 
\end{proposition}
\begin{proof}
Let consider the above notations and diagram \eqref{eqhhs38}. Set $\xi:=\varrho\circ \tr^E(\pi)$. We have
\begin{equation}\label{tr.eq5}
\begin{split}
\xi\circ u\circ \varrho&=\varrho\circ\tr^E(\pi)\circ u\circ \varrho\\&
 =\varrho\circ\tr^E(\pi)\circ E(\pi)\\&
 =\varrho\circ \Nm(E(X))\\&
 = n!\cdot\varrho
\end{split}
\end{equation}
Hence, $\xi\circ u\circ \varrho= n!\cdot\varrho$. Notice that $\varrho$ is an epimorphism. This implies the equality
$\xi\circ u=n!\cdot\id$.
On the other hand, we have
\begin{equation}\label{tr.eq6}
\begin{split}
 u\circ (\frac{1}{n!}\cdot\xi)\circ E(\pi)&=u\circ \Big(\frac{1}{n!}\cdot\varrho\circ \tr^E(\pi)\Big)\circ E(\pi)\\&
 =\frac{1}{n!}\cdot\Big( E(\pi)\circ\tr^E(\pi)\circ E(\pi) \Big)\\&
 =\frac{1}{n!}\cdot\left( n!\cdot E(\pi) \right)\\&
 =  E(\pi) 
\end{split}
\end{equation}
It follows that $ u\circ (1/n!\cdot\xi)\circ E(\pi)
 = E(\pi)$. By assumption $E(\pi)$ is an epimorphism. Therefore, we get $u\circ (1/n!\cdot\xi)=\id$.
 We conclude that $u$ is an isomorphism. 
\end{proof}

 \medskip
 
{\em\noindent Projector symmetric powers}--- Let $\bcD$ be a stable model category \cite{HPSN97}. We say that $\bcD$ is $\QQ$-linear if $\Ho(\bcD)$ is a $\QQ$-linear triangulated category \cite{C-D13}. Let $\bcD$ be a symmetric monoidal $\QQ$-linear stable model category. Let $X$ be an object of $\Ho(\bcD)$. Suppose that $\tensor$ is the corresponding symmetric monoidal product of $\Ho(\bcD)$. For a positive integer $n$, we have a representation $\rho_{X^{\tensor n}}:\Sigma_n\into\Aut(X^{\tensor n})$ of $\Sigma_n$ on $X^{\tensor n}$ induced by permutation of factors. Set
$$d_n:=\frac{1}{n!}\cdot \Nm(X^{\tensor n})=\frac{1}{n!}\cdot \sum_{\sigma\in\Sigma_n}\rho_{A^{\tensor n}}(\sigma)\,.$$
This endomorphism is nothing but that the image of the symmetrization projector $1/n!\cdot\sum_{\sigma\in\Sigma_n}\sigma$ under the induced $\QQ$-linear map $\QQ[\Sigma_n]\into \End(X^{\tensor n})$. 
Since the category $\Ho(\bcD)$ is a $\QQ$-linear triangulated category with small coproducts, it is a pseudo-abelian category, see \cite{Nee01}. As $d_n$ is idempotent, i.e. $d_n\circ d_n=d_n$,  it splits in $\Ho(\bcD)$. This implies that $p$ has image in $\Ho(\bcD)$. We write 
$$\Sym^n_{\pr}(X):=\im d_n\,,$$
 and called it the $n$th {\em fold projector symmetric powers} of $X$. By convention, for $n=0$,  $\Sym^n_{\pr}(X)$ is the unit object $\Ho(\bcC)$. 

\begin{example}
{\em Let $\DM^{-}(k,\QQ)$ be the Voevodsky's category with rational coefficients over a field $k$ \cite{MVW06}. A $k$-rational point of a smooth projective curve $C$ induces a decomposition of the motive $M(C)$ into $\QQ\oplus M^1(C)\oplus\QQ(1)[2]$ in $\DM^{-}(k,\QQ)$. The $n$th fold projector symmetric power $\Sym^n_{\pr}(M^1(C))$ vanishes for $n$ sufficiently bigger that $2g$, where $g$ is the genus of $C$. }
\end{example}

\begin{proposition}
Let $\bcD$ be a symmetric monoidal $\QQ$-linear stable model category. The projector symmetric powers $\Sym^n_{\pr}$, for all $n\in\NN$, induce a $\lambda$-structure on $\Ho(\bcD)$.   
\end{proposition}
\begin{proof}
By convention, $\Sym^0_{\pr}$ is the constant endofunctor whose value is the unit object of $\Ho(\bcD)$. From the definition, the endofunctor $\Sym^1_{\pr}$ is the identity on $\Ho(\bcD)$. Let $X\into Y\into Z$ be a cofibre sequence in $\Ho(\bcD)$. By \cite[Proposition 15]{Go06}, there exists a sequence
\begin{equation}\label{fscr9}
\Sym^n_{\pr}(X)=A_0\into A_1\into \cdots\into A_n=\Sym^n_{\pr}(Y)
\end{equation} 
in $\Ho(\bcD)$, such that for each $1\leq i\leq n$, we have
$$\cone(A_{i-1}\into A_i)=\Sym^{n-i}_{\pr}(X)\tensor \Sym^i_{\pr}(Z)\,.$$
Thus, the K\"unneth tower axiom is satisfied. The functorial axiom on cofibre sequences follows from the functorial construction of the sequences of the form\eqref{fscr9}, see {\it loc.cit.} 
\end{proof}

%

\bigskip

We recall that an $h$-{\em covering} of a scheme $X$ is a finite family $\{p_i:X_i\into X\}_{i\in I}$ of morphisms of finite type such that the induced morphism $\amalg_{i\in I}p_i:\coprod_{i\in I}X_i\into X$ is a universal topological epimorphism. A $\qfh$-{\em covering} of $X$ is a  $h$-covering $\{p_i\}_{i\in I}$ such that $p_i$ is quasi-finite for all $i\in I$ (see \cite{Voe96}).
In the next paragraphs, all $\qfh$-sheaves are defined on the category of schemes of finite type over a field $k$. 

\begin{lemma}[Voevodsky]\label{tr471}
Let $X$ be a quasi-projective $k$-scheme and let $\pi$ be the canonical morphism from $X^n$ onto $\Sym^n(X)$.  Suppose that $F$ is a $\qfh$-sheaf of Abelian monoids on the category of $k$-schemes of finite type, and let us denote by $\pi^*$ the restriction morphism $F(\Sym^n(X))\into F(X^n)$ induced by $\pi$. Then, the image of $\pi^*$ coincides with $F(X^n)^{\Sigma_n}$. 
\end{lemma}
\begin{proof}
 As the morphism $\pi$ forms a $\qfh$-covering of $\Sym^n(X)$, we can follow the arguments of the proof of \cite[Prop.~3.3.2]{Voe96} or \cite[Lemma~5.16]{SV96}. 
\end{proof}

\begin{proposition}\label{Voetranf}
Let $X$ be a quasi-projective $k$-scheme. For every $n\in\NN$, the morphism $\ZZ_{\qfh}(\pi):\ZZ_{\qfh}(X^n)\into\ZZ_{\qfh}(\Sym^nX)$ has transfer, i.e. there exists a morphism $\tr_n$ such that 
\begin{equation}\label{speq3c}
\ZZ_{\qfh}(\pi)\circ\tr_n=\sum_{\sigma}\ZZ_{\qfh}(\sigma), \qquad \text{and}
\end{equation}
\begin{equation}\label{eqg57}
\ZZ_{\qfh}(\pi)\circ\tr_n
=n!\cdot\id_{\ZZ_{\qfh}(\Sym^nX)}\,.
\end{equation}
\end{proposition}
\begin{proof}
Let us consider the representable $\qfh$-sheaf $F=\ZZ_{\qfh}(X^n)$. Every permutation $\sigma$ in $\Sigma_n$ induces an automorphism $\sigma:X^n\into X^n$ by permuting factors, $\sigma$ corresponds to an element of $F(X^n)$, denoted by the same letter. Notice that the element $\theta_n:=\sum_{\sigma\in\Sigma_n}\sigma$ is an element of $F(X^n)$ which is $\Sigma_n$-invariant, i.e. $\sigma(\theta_n)=\theta$ for all permutation $\sigma\in\Sigma_n$. By Lemma \ref{tr471}, there exists an element $t_n$ of $F(\Sym^nX)$ such that $t_n\circ \pi^*=\theta$. We denote by $\tr_n:\ZZ_{\qfh}(\Sym^n)\into\ZZ_{\qfh}(X^n)$ the morphism of $\qfh$-sheaves corresponding to the section $t_n$. Then the equality $t_n\circ \pi^*=\theta$ gives the equality \eqref{speq3c}. Now, from \eqref{speq3c}, we have 
 \begin{align*}
 \ZZ_{\qfh}(\pi)\circ\tr_n\circ \ZZ_{\qfh}(\pi)&=\left(\sum_{\sigma}\ZZ_{\qfh}(\sigma)\right)\circ \ZZ_{\qfh}(\pi)\\&
=\sum_{\sigma}\ZZ_{\qfh}(\sigma)\circ \ZZ_{\qfh}(\pi)\\&
=\sum_{\sigma}\ZZ_{\qfh}(\pi)\\&
=n!\cdot\ZZ_{\qfh}(\pi)\,.
 \end{align*}
hence, $ \ZZ_{\qfh}(\pi)\circ\tr_n\circ \ZZ_{\qfh}(\pi)
=n!\cdot\ZZ_{\qfh}(\pi)$. This induces the equality \eqref{eqg57}.
\end{proof}

\bigskip

\section{Comparison of symmetric powers}
\label{Mainsec}

In this section $\SH_T(k)$ will be the stable $\AF^1$-homotopy category of schemes over a field $k$ constructed in \cite{Jar00}. The main result in this section is Theorem \ref{finth1}, which asserts that if $-1$ is a sum of squares, then the categoric, geometric and projector symmetric powers of a quasi-projective scheme are isomorphic in $\SH_T(k)_{\QQ}$.

\medskip

  \begin{lemma}\label{tscor15}
Let $X$ be a quasi-projective $k$-scheme. Then, the canonical morphism $\QQ_{\qfh}(\Sym^nh_X)\into \QQ_{\qfh}(\Sym^n_gh_X)$ is an isomorphism of $\qfh$-sheaves of $\QQ$-vector spaces.  
 \end{lemma}
 \begin{proof}
Let $\pi:X^n\into\Sym^n(X)$ be the canonical morphism. Since the functor $\QQ_{\qfh}(-)$ is a left adjoint, we have the following isomorphisms
 $$\Sym^n\QQ_{\qfh}(X)=\QQ_{\qfh}(X)^{\tensor n}/\Sigma_n\isom \QQ_{\qfh}(X^{ n})/\Sigma_n\,.$$
 By Proposition \ref{Voetranf}, the morphism $\QQ_{\qfh}(\pi)$ has transfer. Hence, the corollary follows from Proposition \ref{gstab45}.  
 \end{proof}

  \begin{corollary}\label{tscor21}
 Let $X$ be a quasi-projective $k$-scheme. Then, the morphism  from $\QQ_{\qfh}(\Sym^nh_X)$ to $\QQ_{\qfh}(\Sym^n_gh_X)$ is an isomorphism in  $\DM_{\qfh}(k)_{\QQ}$.  
 \end{corollary}
\begin{proof}
It follows from Corollary \ref{tscor15} and \cite[Prop.~5.3.37]{C-D13}.  
\end{proof}
Write $M_{\qfh,\QQ}$ for the canonical functor from the category of $k$-schemes of finite type to $\DM_{\qfh}(k)_{\QQ}$.
 \begin{corollary}\label{tscor2}
 Let $X$ be a quasi-projective $k$-scheme and let $\pi:X^n\into\Sym^n(X)$ be the canonical morphism. Then the morphism $M_{\qfh,\QQ}(\pi)$ has transfer.  
 \end{corollary}
\begin{proof}
It follows from Proposition \ref{Voetranf}.
\end{proof}
    
Let $E_{\QQ}$ be the canonical functor from the category of $k$-schemes of finite type to $\SH_T(k)_{\QQ}$. 

\begin{corollary}\label{cog7ww}
Suppose that $-1$ is a sum of squares in a field $k$. For a quasi-projective $k$-scheme  $X$, the induced morphism $E_{\QQ}(\pi)$ from $E_{\QQ}(X^n)$ to $ E_{\QQ}(\Sym^nX)$ has transfer. 
\end{corollary} 
\begin{proof}
It follows from Corollary \ref{tscor2} and Corollary \ref{tfddj}.   

\end{proof}

\begin{proposition}\label{prof3401}
Assume $-1$ is a sum of squares in a field $k$. For a  quasi-projective $k$-scheme  $X$, one has an isomorphism
$$\Sym^n_{\pr}E_{\QQ}(X)\isom E_{\QQ}(\Sym^n X)\,.$$
\end{proposition}
\begin{proof}
By Corollary \ref{cog7ww}, the morphism $E_{\QQ}(\pi)$ has transfer, say $\tr_{\QQ}(\pi)$. From the equality  $\tr_{\QQ}(\pi)\circ E_{\QQ}(\pi)=\Nm(E_{\QQ}(X))$, we obtain that the projector $d_n$ is equal to $1/n!\cdot\tr_{\QQ}(\pi)\circ E_{\QQ}(\pi)$. 
Hence, from the equality $E(\pi)\circ \tr_{\QQ}(\pi)=n!.\id$, we deduce that $\im d_n\isom E_{\QQ}(\Sym^nX)$, as required. 
\end{proof}

\begin{proposition}\label{geostth74}
Suppose that $-1$ is a sum of squares in $k$. For every  quasi-projective $k$-scheme $X$, the canonical morphism $$\Sym^n_{T}(\Sigma^{\infty}_TX_+)\into \Sigma^{\infty}_T(\Sym^nX)_+$$ 
is a rational stable $\AF^1$-weak equivalence. 
\end{proposition}
\begin{proof}
By Lemma \ref{javvc5}, the morphism $\Sym^n_{T}(\Sigma^{\infty}_TX_+)\into \Sigma^{\infty}_T(\Sym^nX)_+$ is isomorphic to the $T$-suspension of the canonical morphism $\Sym^n_{T}(h_{X_+})\into \Sym^n_g(h_{X_+})$ of pointed simplicial sheaves.  Hence the proposition follows from Corollary \ref{tscor21} and Corolary \ref{tfddj}. 
\end{proof}

Now, we are ready to state and prove our main theorem in this section. We recall that $E_{\QQ}$ is the canonical functor from the category of $k$-schemes of finite type to $\SH_T(k)_{\QQ}$.  

\begin{theorem}\label{finth1}
Suppose that $-1$ is a sum of squares in a field $k$. For any  quasi-projective $k$-scheme $X$, we have the following isomorphisms
$$L\Sym^n_TE_{\QQ}(X)\isom  E_{\QQ}(\Sym^n X)\isom \Sym^n_{\pr}E_{\QQ}(X)\,.$$
\end{theorem}
\begin{proof}
 The isomorphism on the left-hand side follows from Proposition \ref{geostth74}. The second isomorphism follows from Proposition \ref{prof3401}.
\end{proof}
\bigskip


Let us consider the sets $I^+_T=\bigcup_{n>0}F_n(I)$, $J^+_T=\bigcup_{n>0}F_n(J)$, where $I$ (resp. $J$) is the class of generating (resp. trivial) cofibrations of the injective model structure of $\Deltaop\bcS_{\ast}$. 
Denote by $W^+_T$ the class of morphisms of symmetric $T$-spectra $f:\bcX\into\bcY$ such that each term $f_n:\bcX_n\into\bcY_n$ is an $\AF^1$-weak equivalence for $n>0$.  The sets $I^+_T$, $J^+_T$ and the class $W^+_T$ define on $\Sp_T(k)$ a cofibrantly generated model structure called {\em positive projective model structure}, see \cite{GoGu}. 
The positive projective cofibrations are projective cofibrations that are isomorphisms in the level zero.      
%

For a $T$-spectra $\bcX$ in $\Sp_T(k)$,  the  $n$th fold {\em homotopy symmetric power} $\Sym^n_{h,T}(\bcX)$ is defined as the homotopy colimit $\hocolim_{\Sigma_n}\bcX^{\wedge n}$. The Borel construction allows one to express $\Sym^n_{h,T}(\bcX)$ as the homotopy quotient $(E\Sigma_n)_+\wedge_{\Sigma_n}\bcX^{\wedge n}$, where $E\Sigma_n$ is the $\Sigma_n$-universal principal bundle,
 see \cite[Example~4.5.5]{Riehl14}. 
The canonical morphism from $(E\Sigma_n)_+\wedge\bcX^{\wedge n}$ to $\bcX^{\wedge n}$ induces a morphism 
$$
\Sym^n_{h,T}(\bcX)\into \Sym^n_T(\bcX)
$$
which is a stable $\AF^1$-weak equivalence when $\bcX$ is a cofibrant $T$-spectra with respect to the positive projective model structure. This implies the existence of an isomorphism of endofunctors 
\begin{equation}\label{eqg15}
\Sym^n_{h,T}(\bcX)\into L\Sym^n_T
\end{equation}
on stable $\AF^1$-homotopy category $\SH_T(k)$, see \cite{GoGu11}.

\begin{remark}
By Theorem \ref{finth1} and \eqref{eqg15}, we get the following isomorphisms
 $$\Sym^n_{h,T}E_{\QQ}(X)\isom L\Sym^n_TE_{\QQ}(X)\isom  E_{\QQ}(\Sym^n X)\isom \Sym^n_{\pr}E_{\QQ}(X)$$
for any quasi-projective $k$-scheme $X$.
\end{remark}

\begin{example}
{\em Let $X$ be the $2$-dimensional affine space $\AF^2$ over $k$. Then, by Proposition \ref{papxz301}, the canonical morphism $L\vartheta_X: L\Sym^n h_X\isom L\Sym^n_g h_X$ is not an isomorphism in the unstable motivic category over $k$. However, by Theorem \ref{finth1}, $\vartheta_X$ induces an isomorphism $L\Sym^n_TE_{\QQ}(X)\isom  E_{\QQ}(\Sym^n_g X)\,.$
}
\end{example}

\section{Appendix}
\label{Appen}

Here, we study the canonical morphism $\vartheta^n_X:\Sym^nh_X\into \Sym^n_gh_X$, where $X$ is the spectrum $\Spec(L)$ and $L/k$ is a finite Galois extension. We  show that  $\vartheta^n_X$ is an isomorphism on sections (see Proposition \ref{abs.prgx}). In Proposition \ref{papxz301}, we show that the canonical morphism $\vartheta^n_X:\Sym^nh_X\into \Sym^n_gh_X$ is not always an $\AF^1$-weak equivalence. 
\medskip

{\it \noindent Categoric and geometric symmetric  powers do not coincide.}--- Let $\bcC$ be the category of quasi-projective schemes over a field $k$. 
We shall prove that, if $X$ is the $2$-dimensional affine space  $\AF^2$ over $k$, then the canonical morphism $\vartheta^n_X$ from $\Sym^nh_X$ to $\Sym_g^nh_X$ is not an $\AF^1$-weak equivalence in $\Deltaop\bcS$, see Proposition \ref{papxz301},. 
 
\begin{lemma}
Let $X$ be a scheme in $\bcC$. The morphism of simplicial presheaf $\vartheta^n_X: \Sym^nh_X\into\Sym_g^nh_X$ is an $\AF^1$-weak equivalence if and only if for every $\AF^1$-local simplicial presheaf $\bcZ$ the induced morphism $(\vartheta^n_X)^*:\bcZ(\Sym^n
X)\into\bcZ(X^n)^{\Sigma_n}$ is a weak equivalence of simplicial sets. 
\end{lemma}
\begin{proof}
By definition of $\AF^1$-weak equivalence, $\vartheta^n_X$ is an $\AF^1$-weak equivalence if and  only if  for every $\AF^1$-local simplicial presheaf the induced morphism $$(\vartheta^n_X)^*:\Map(\Sym^n_gh_X,\bcZ)\longrightarrow\Map(\Sym^nh_X,\bcZ)$$
 is a weak equivalence of simplicial sets. On one side, we have
 $$\Map(\Sym^n_gh_X,\bcZ)=\Map(h_{\Sym^nX},\bcZ)\isom \bcZ(\Sym^nX)\,,$$ 
 where the above isomorphism follows from the Yoneda's lemma. On the other hand, the functor $\Map(-,\bcZ)$ sends colimits to limits, in particular, we have $$\Map((h_X^{\times n})/\Sigma_n,\bcZ)\isom\Map(h_X^{\times n},\bcZ)^{\Sigma_n}\,.$$ Then, we have   
 $$\Map(\Sym^nh_X,\bcZ)\isom\Map(h_X^{\times n},\bcZ)^{\Sigma_n}\isom \Map(h_{X^n},\bcZ)^{\Sigma_n}\isom \bcZ(X^n)^{\Sigma_n}\,.$$
 Thus, the lemma follows. 
\end{proof}

\begin{proposition}\label{papxz301}
Let $X=\AF^2$ be the $2$-dimensional affine space over a field $k$.  Then, the natural morphism $\vartheta^n_X$ is not an $\AF^1$-weak equivalence.
\end{proposition}
\begin{proof}
 We recall that higher Chow groups $CH^i(-,m)$, for $i$ and $m$ in $\NN$, are $\AF^1$-homotopy invariant (see \cite[page~136]{MVW06}). For $m=0$ the  higher Chow group $CH^i(-,m)$ is nothing but the Chow group $CH^i(-)$. In particular,  the Chow group $CH^i(-)$ is $\AF^1$-homotopy invariant. Then $CH^i(-)$ is $\AF^1$-local as a simplicial presheaf. We shall take $\bcZ=CH^1(-)$ in the previous lemma. 
On one side, we have $X^2=\AF^4$, hence $CH^1(X^2)=CH^1(\AF^4)$ is zero.
On the other hand, $\Sym^2(\AF^2)$ is isomorphic to the product of $\AF^2$ with the quadric cone $\mathQ$ defined by the equation $uw-v^2=0$ in $\AF^3$. By the $\AF^1$-homotopy invariance, $CH^1(\AF^2\times\mathQ)$ is isomorphic to $CH^1(\mathQ)$. By Example 2.1.3 of \cite{Ful98}, $CH^1(\mathQ)=CH_1(\mathQ)$ it is isomorphic to $\ZZ/2\ZZ$. 
Then $(\vartheta^2_{\AF^2})^*$ is the morphism of constant simplicial sets induced by a morphism of sets 
$\ZZ/2\ZZ\into 0$. Since $\ZZ/2\ZZ$ consists of two points, the morphism $(\vartheta^2_{\AF^2})^*$ cannot be a weak equivalence.  We conclude that $\vartheta^2_{\AF^2}$ is not an isomorphism in the motivic $\AF^1$-homotopy category.
\end{proof}

\bigskip 
{\it \noindent Galois extensions.}--- 
Let $L/k$ be a finite Galois field extension and set $X=\Spec(L)$. Let $K$ be an algebraically closed field containing $L$ and let $U=\Spec(K)$. In the following paragraphs, we shall prove that for any integer $n\geq0$, the canonical morphism of sets
\[\vartheta^n_X(U): (\Sym^nh_X)(U)\into (\Sym^n_gh_X)(U)\,.\] 
is an isomorphism (see proposition \ref{abs.prgx}).

\begin{lemma}
Let $L/k$ be a finite Galois extension of degree $r\geq1$ and let $n$ be an integer $n\geq1$. The $k$-algebra $(L^{\tensor_k n})^{\Sigma_n}$ has dimension $\binom{r+n-1}{n}$ as $k$-vector space.
\end{lemma}
\begin{proof}
Since $L$ is a Galois extension over $k$ of degree $r$, the tensor product $L^{\tensor_k n}$ is isomorphic to $L^{\times r^{n-1}}$ as vector spaces over $k$.
Let $\{v_1,v_2,\dots,v_r\}$ be a $k$-basis of $L$. The family 
 $\left\lbrace v_{i_1}\tensor v_{i_2}\tensor\dots\tensor v_{i_n}\right\rbrace_{0\leq i_1,i_2,\dots, i_n\leq r}$ is a $k$-basis of $L^{\tensor_k n}$. An element of $(L^{\tensor_k n})^{\Sigma_n}$ has the form 
\[\sum_{0\leq i_1,\dots, i_n\leq r}a_{i_1,\dots, i_n}\cdot v_{i_1}\tensor \dots\tensor v_{i_n}\]
 such that 
 \[\sum_{0\leq i_1,\dots, i_n\leq r}a_{i_1,\dots, i_n}.v_{i_{\sigma(1)}}\tensor \dots\tensor v_{i_{\sigma(n)}}=\sum_{0\leq i_1,\dots, i_n\leq r}a_{i_1,\dots, i_n}.v_{i_1}\tensor \dots\tensor v_{i_n}\]
for all $\sigma\in\Sigma_n$. 
We have deduce from the above equality that 
\begin{equation}\label{abs.5ta}
a_{i_1,\dots, i_n}= a_{i_{\sigma(1)},\dots, i_{\sigma(n)}}
\end{equation}
for all  $\sigma\in\Sigma_n$. 
We recall that a combination of $\{1,2,\dots,r\}$ choosing $n$ elements is an unordered $n$-tuple $\{i_1,\dots, i_n\}$ allowing repetition of the elements $i_1,\dots, i_n$ in $\{1,2,\dots,r\}$. Let us denote by $C(r,n)$ the set of all repetitions of $\{1,2,\dots,r\}$ choosing $n$ elements, and fix $I=\{i_1,\dots, i_n\}$ in $C(r,n)$. Suppose $I$ has $p$ different elements $j_1,\dots, j_p$, where $1\leq p\leq n$, such that each there are $k_l$ repetitions of $j_l$ in $I$ for $1\leq l\leq p$. In particular, one has $\sum_{j=1}^lk_l=n$. Let us denote by $P(I)=P(i_1,\dots, i_n)$ the set of permutations with repetitions of $\{i_1,\dots, i_n\}$. By a computation in combinatorics, $P(I)$ has cardinal equal to $\frac{n!}{k_1!.\cdots.k_p!}$ elements.    
We have 
\[ \sum_{\{i'_1,\dots, i'_n\}\in P(i_1,\dots, i_n)}v_{i'_1}\tensor \dots\tensor v_{i'_n}=\frac{k_1!.\cdots.k_p!}{n!}\cdot\sum_{\sigma\in\Sigma_n}v_{\sigma(i_1)}\tensor \dots\tensor v_{\sigma(i_n)}\,.\]
and from \eqref{abs.5ta} we deduce that $a_{i'_1,\dots, i'_n}=a_{i_1,\dots, i_n}$ for all $\{i'_1,\dots, i'_n\}\in P(i_1,\dots, i_n)$, hence
\[\sum_{0\leq i_1,\dots, i_n\leq r}a_{i_1,\dots, i_n}.v_{i_1}\tensor \dots\tensor v_{i_n}=\sum_{\{i_1,\dots, i_n\}\in C(r,n)}a_{i_1,\dots, i_n}\left(\sum_{\{i'_1,\dots, i'_n\}\in P(i_1,\dots, i_n)}v_{i'_1}\tensor \dots\tensor v_{i'_n}\right)\]
Observe that the set \[\left\lbrace\sum_{\{i'_1,\dots, i'_n\}\in P(i_1,\dots, i_n)}v_{i'_1}\tensor \dots\tensor v_{i'_n}\right\rbrace_{\{i_1,\dots, i_n\}\in C(r,n)}\] 
is formed by linearly independent vectors in the $k$-vector space $L^{\tensor_k n}$. Hence, it is a basis of $(L^{\tensor_k n})^{\Sigma_n}$.  
Then the dimension of $(L^{\tensor_k n})^{\Sigma_n}$ is determined by the cardinal of $C(r,n)$, thus  $(L^{\tensor_k n})^{\Sigma_n}$ has dimension $|C(r,n)|=\binom{r+n-1}{n}$.
\end{proof}

\begin{example}
{\em In the previous lemma, if $L/k$ is a cubic extension i.e. $r=3$ with a $k$-basis $\{v_1,v_2,v_3\}$, and $n=2$; then the $k$-algebra $(L\tensor L)^{\Sigma_2}$ has dimension $6$ as $k$-vector space and its canonical basis is formed by the vectors
\[\left\lbrace v_1\tensor v_1,  v_2\tensor v_2,v_3\tensor v_3,(v_1\tensor v_2+v_2\tensor v_1), (v_1\tensor v_3+v_3\tensor v_1),  (v_2\tensor v_3+v_3\tensor v_2) \right\rbrace\,.\]  
 }
\end{example}

\begin{lemma}\label{abslmla}
Let $L/k$ be a finite Galois extension of degree $r\geq 1$ and set $X=\Spec(L)$. Let $K$ be an algebraically closed field containing $L$ and let $U=\Spec(K)$. Then, for any integer $n\geq0$, the set $h_{\Sym^nX}(U)$ is a finite set with $\binom{r+n-1}{n}$ elements. 
\end{lemma}
\begin{proof}
Since $(L^{\tensor_k n})^{\Sigma_n}$ is a sub-algebra of $L^{\tensor_k n}\isom L^{\times r^{n-1}}$,  the $k$-algebra $(L^{\tensor_k n})^{\Sigma_n}$ is isomorphic to a product $\prod_{j=1}^{r^{n-1}}L_{j}$, 
where each $L_j$ is a field extension of $k$ contained in $L$. By the previous lemma, we have that the sum $\sum_{j=1}^{r^{n-1}}\dim_kL_j$ is equal to $\binom{r+n-1}{n}$. 
Let $K$ be an  algebraically closed field containing $L$. One has,  
\begin{equation*}
\begin{split}
\Hom_{k}\left((L^{\tensor_k n})^{\Sigma_n}, K\right)&=\Hom_{k}\left(\prod_{1\leq j\leq r^{n-1}}L_{j}, K\right)\\&
\isom \coprod_{1\leq j\leq r^{n-1}}\Hom_{k}(L_{j}, K)
\end{split}
\end{equation*}
Since $L_j/k$ is a finite separable extension, $\Hom_{k}(L_{j}, K)$ is a finite set and its cardinal equal to $\dim_kL_j$ for all $j=1,\dots,r^{n-1}$. Hence, the set
$\Hom_{k}\left((L^{\tensor_k n})^{\Sigma_n}, K\right)$ is  finite of cardinal equal to
$\sum_{j=1}^{r^{n-1}}\dim_kL_j=\binom{r+n-1}{n}$. 
Let $U=\Spec(K)$. Hence, we have
\begin{equation*}
\begin{split}
h_{\Sym^nX}(U)&=\Hom_{k}(U, \Sym^nX)\\&
=\Hom_{\Spec(k)}\big(\Spec(K),\Spec((L^{\tensor_k n})^{\Sigma_n})\big)\\&
=\Hom_{k}\left((L^{\tensor_k n})^{\Sigma_n}, K\right)\\&
\isom \coprod_{1\leq j\leq r^{n-1}}\Hom_{k}(L_{j}, K)
\end{split}
\end{equation*}
Thus, we conclude that $h_{\Sym^nX}(U)$ is a finite set with $\binom{r+n-1}{n}$ elements. 
\end{proof}

In conclusion, we have the following proposition. 

\begin{proposition}\label{abs.prgx}
Let $L/k$ be a finite Galois extension and set $X=\Spec(L)$. Let $K$ be an algebraically closed field containing $L$ and let $U=\Spec(K)$. Then, for any integer $n\geq0$, the canonical morphism of sets
\[\vartheta^n_X(U): (\Sym^nh_X)(U)\into (\Sym^n_gh_X)(U)\] 
is an isomorphism. 
\end{proposition}
\begin{proof}
Suppose that $L=k(\alpha)$ where $\alpha$ is a root of an irreducible polynomial $P(t)$ of degree $r\geq 1$. 
Notice that  $(\Sym^nh_X)(U)=\Hom_k(L,K)^n/\Sigma_n$ 
is a finite set with $\binom{r+n-1}{n}$ elements. On the other hand, by lemma \ref{abslmla}, $(\Sym^nh_X)(U)$ is also a finite set with $\binom{r+n-1}{n}$ elements, then it is enough to prove the injectivity of the canonical morphism of sets from $\Hom_k(L,K)^n/\Sigma_n$ to $\Hom_k\left((L^{\tensor_kn})^{\Sigma_n},K\right)$,
defined by $\{f_1,\dots,f_n\}\mapsto(f_1\tensor\cdots\tensor f_n)|_{(L^{\tensor_kn})^{\Sigma_n}}$. Indeed, let $\{f_1,\dots,f_n\}$ and $\{f'_1,\dots,f'_n\}$ be two unordered $n$-tuple in $\Hom_k(L,K)^n/\Sigma_n$ such that 
\begin{equation}\label{abs.bgt}
(f_1\tensor\cdots\tensor f_n)|_{(L^{\tensor_kn})^{\Sigma_n}}=(f'_1\tensor\cdots\tensor f'_n)|_{(L^{\tensor_kn})^{\Sigma_n}}
\end{equation}
We put $\alpha_1=f_1(\alpha),\dots,\alpha_n=f_n(\alpha)$ and $\alpha'_1=f'_1(\alpha),\dots,\alpha'_r=f'_n(\alpha)$.
Then $\{\alpha_1,\dots,\alpha_n\}$ and $\{\alpha'_1,\dots,\alpha'_n\}$ are two unordered $n$-tuples formed by roots of $P(t)$ non necessarily distinct each other. 
Notice that to prove that the set $\{f_1,\dots,f_n\}$ is equal to $\{f'_1,\dots,f'_n\}$. It will be enough to prove that the set $\{\alpha_1,\dots,\alpha_n\}$ is equal to $\{\alpha'_1,\dots,\alpha'_n\}$, since a homomorphism of $k$-algebras  $L\into K$ is uniquely determined by a root of  $P(t)$. 
In fact, observe that the elements 
\begin{equation*}
\left\{
\begin{split}
&\sum_{i=1}^n \left(1\tensor \cdots\tensor1\tensor\underbrace{\alpha}_{\text{$i$th position}}\tensor1\tensor\cdots\tensor 1\right),\\&
\sum_{1\leq i<j\leq n}^n\left(1\tensor \cdots\tensor1\tensor\underbrace{\alpha}_{\text{$i$th position}}\tensor1\tensor \cdots\tensor1\tensor\underbrace{\alpha}_{\text{$j$th position}}\tensor1\tensor\cdots\tensor 1\right),\\&
\qquad\cdots\cdots\\&
\qquad\cdots\cdots\\&
\alpha\tensor \alpha\tensor \cdots\tensor \alpha
\end{split}
\right.
\end{equation*}
lie in $(L^{\tensor n})^{\Sigma_n}$. In view of the equality
$(f_1\tensor\cdots\tensor f_n)(a_1\tensor\cdots\tensor a_n)=a_1\cdot\cdots\cdot a_n$ for all elements $a_1,\dots, a_n$ in $L$, we deduce the following equalities,
\begin{align*}
\sum_{i=1}^n\alpha_i\hspace{1.2cm}&= (f_1\tensor\cdots\tensor f_n)\left(\sum_{i=1}^n 1\tensor \cdots\tensor\alpha\tensor\cdots\tensor 1\right)\\
\sum_{1\leq i<j\leq n}^n\alpha_i\cdot\alpha_j\hspace{0.3cm}&= (f_1\tensor\cdots\tensor f_n)\left(\sum_{1\leq i<j\leq n}^n 1\tensor \cdots\tensor\alpha\tensor \cdots\tensor\alpha\tensor\cdots\tensor 1\right)\\
\cdots&\cdots\\
\cdots&\cdots\\
\alpha_1\cdot\alpha_2\cdot \cdots\cdot \alpha_n&=  (f_1\tensor\cdots\tensor f_n)(\alpha\tensor \alpha\tensor \cdots\tensor \alpha)\,.
\end{align*}
Using \eqref{abs.bgt}, these equalities allow us to deduce the following,

\begin{align*}
\sum_{i=1}^n\alpha_i \hspace{1.2cm}&= \sum_{i=1}^n\alpha'_i\\
\sum_{1\leq i<j\leq n}^n\alpha_i\cdot\alpha_j\hspace{0.3cm}&=\sum_{1\leq i<j\leq n}^n\alpha'_i\cdot\alpha'_j\\
\cdots&\cdots\\
\cdots&\cdots\\
\alpha_1\cdot\alpha_2\cdot \cdots\cdot \alpha_n&= \alpha'_1\cdot\alpha'_2\cdot\cdots\cdot\alpha'_n\,.
\end{align*}
Notice also that these elements are in $k$, because they are invariants under $\Gal(L/k)$. 
Now, observe that $\alpha_1,\dots,\alpha_n$ are all solutions of the polynomial
\[P(t):=t^n-\left(\sum_{i=1}^n\alpha_i\right)\cdot t^{n-1}+\left(\sum_{1\leq i<j\leq n}^n\alpha_i\cdot \alpha_j\right)\cdot t^{n-2}+\dots +(-1)^n\cdot\alpha_1\cdot\cdots\cdot\alpha_n \] 
 in $k[t]$, whereas $\alpha'_1,\dots,\alpha'_n$ are all solutions of the polynomial
\[P'(t):=t^n-\left(\sum_{i=1}^n\alpha'_i\right)\cdot t^{n-1}+\left(\sum_{1\leq i<j\leq n}^n\alpha'_i\cdot\alpha'_j\right)\cdot t^{n-2}+\dots +(-1)^n\cdot\alpha'_1\cdot\cdots\cdot\alpha'_n \] 
 which is also in $k[t]$. Since $P(t)=P'(t)$, we conclude that  $\{\alpha_1,\dots,\alpha_n\}=\{\alpha'_1,\dots,\alpha'_n\}$, as required. 
\end{proof}

\begin{small}

\end{small}

\bigskip

\bigskip

\begin{small}

{\sc Department of Mathematical Sciences, University of Liverpool, Peach Street, Liverpool L69 7ZL, England, UK}

\end{small}

\medskip

\begin{footnotesize}

{\it E-mail address}: {\tt Joe.Palacios-Baldeon@liverpool.ac.uk}

\end{footnotesize}

\end{document}